\documentclass[11pt]{article}

\usepackage[left=1in,right=1in,top=1in,bottom=1in]{geometry}

\usepackage{graphicx}

\usepackage[normalem]{ulem} 
\usepackage[utf8]{inputenc} % allow utf-8 input
\usepackage[T1]{fontenc}    % use 8-bit T1 fonts
\usepackage{hyperref}       % hyperlinks
\usepackage{url}            % simple URL typesetting
\usepackage{booktabs}       % professional-quality tables
\usepackage{mathtools}
\usepackage{amsfonts}       % blackboard math symbols
\usepackage{nicefrac}       % compact symbols for 1/2, etc.
\usepackage{microtype}      % microtypography
\usepackage{comment}
\usepackage{graphicx, ulem}
\usepackage[dvipsnames]{xcolor}
\usepackage{algorithm,algpseudocode}
\usepackage{multirow}
\usepackage{enumitem}
\usepackage{amsmath,amsfonts,latexsym,amssymb,amsbsy}

\usepackage{footnote}
\usepackage{thmtools, thm-restate}
\usepackage[symbol]{footmisc}
\usepackage{bbm}
\usepackage{longtable}
\usepackage{graphicx, xcolor, ulem}
\usepackage{algorithm,algpseudocode}
\usepackage{multirow}
\usepackage{enumitem}
\usepackage{amsmath,amsfonts,latexsym,amssymb,amsbsy, amsthm, mathtools}
\usepackage{footnote}
\usepackage{thmtools, thm-restate}
\usepackage{authblk}
\usepackage{setspace}

\usepackage[textsize=tiny,textwidth=2cm,color=green!50!gray]{todonotes} %,disable

\newcommand \hm[1] {{\color{orange}(Hassan: #1)}}

\newcommand{\Poly}{\mathcal{P}}
\newcommand{\PWidth}{\operatorname{PWidth}}
\newcommand{\conv}{\operatorname{conv}}
\newcommand{\dist}{\operatorname{dist}}

\newcommand{\one}{\mathbf{1}}
\DeclareMathOperator{\supp}{supp}

\newcommand{\iter}[2]{\ensuremath{#1^{(#2)}}}

\usepackage{pifont}% http://ctan.org/pkg/pifont

\newtheorem{theorem}{Theorem}
\newtheorem{remark}{Remark}

\newtheorem{lemma}{Lemma}

\usepackage{thmtools, thm-restate}

\algnewcommand\INPUT{\item[{\textbf{Input:}}]}
\algnewcommand\OUTPUT{\item[{\textbf{Output:}}]}
\algnewcommand\RETURN{\item[{\textbf{Return:}}]}

\DeclareMathOperator*{\argmin}{arg\,min}
\DeclareMathOperator*{\argmax}{arg\,max}

\allowdisplaybreaks % To allow page breaks inside equation environment
\newcommand{\innerprod}[2]{\left \langle #1,  #2 \right \rangle}

\definecolor{blue}{rgb}{0,0,0}

\title{A Unifying Perspective on Descent Directions for Constrained Minimization and Switching Criteria \vspace{15 pt}}

\author[1]{Swati Gupta}
\author[2]{Hassan Mortagy}
\author[3]{Sebastian Pokutta}

\affil[1]{\small MIT Sloan School of Management, Cambridge, MA 02142 \protect \\
{\small \tt swatig@mit.edu}}

\affil[2]{\small School of Industrial and Systems Engineering, Georgia Institute of Technology, Atlanta, GA 30332 \protect \\
{\small \tt hmortagy@gatech.edu}}

\affil[3]{\small Zuse Institute Berlin and Technische Universit\"{a}t Berlin, 10623 Berlin \protect \\
{\small \tt pokutta@zib.de}}
\date{}
	\definecolor{yaleblue}{rgb}{0.06, 0.3, 0.57}
\definecolor{blue}{rgb}{0,0,0}
\definecolor{teal}{rgb}{0,0,0}
\begin{document}

\maketitle

\begin{abstract}
Descent directions are fundamental to first-order methods for constrained convex optimization. In this work, we provide a unified perspective of these directions as they arise within the projected and conditional gradient methods. We show that the directional derivative of the Euclidean projection of the negative gradient is the steepest feasible descent direction, and induces a natural optimality measure independent of polytope-specific geometric constants. We characterize the properties of the projection of the gradient over the polytope, and show that Frank–Wolfe vertices are essentially the limit of this projection curve. Motivated by these insights, we develop hybrid algorithms that interpolate between projected gradient descent and conditional gradients, using computable switching criteria. The proposed methods achieve linear convergence without relying on the pyramidal width or similar constants. We further show that the number of breakpoints in the projection curve over the simplex and the hypercube is linear in the dimension. 
% Enter your abstract
\end{abstract}

\section{Introduction} \label{sec: intro}
We\footnote{A preliminary version of this work appeared in Advances on Neural Information Processing Systems, NeurIPS 2020.} consider the optimization problem $\min_{{x} \in \mathcal{P}} h({x})$, where $\mathcal{P} \subseteq \mathbb{R}^n$ is a polytope, and $h: \mathcal{P} \to \mathbb{R}$ is a differentiable, smooth, and strongly convex function. Smooth convex optimization problems over polytopes represent an important class of optimization problems encountered in various settings, such as low-rank matrix completion \cite{Grigas_2015}, structured supervised learning \cite{Jaggi2013,Bashiri_2017}, electrical flows over graphs \cite{lyons2005probability}, video co-localization in computer vision \cite{Joulin2014}, traffic assignment problems \cite{Networkflows1993}, and submodular function minimization \cite{Fujishige_2009}. First-order methods in convex optimization rely on movement in the best local direction for descent (e.g., negative gradient), and this is enough to obtain linear convergence for unconstrained optimization. However, in constrained settings, the gradient may no longer be a feasible direction of descent, and there are two broad classes of methods traditionally: (i) projection-based methods that move in the direction of the negative gradient and then project to ensure feasibility, and (ii) conditional gradient methods that move in feasible directions which approximate the gradient. Projection-based methods such as projected gradient descent or mirror descent \cite{Nemirovski1983} enjoy dimension independent linear rates of convergence (assuming no acceleration and constant step size), for instance, $(1-\frac{\mu}{L})$ contraction in the objective per iteration (so that the number of iterations to get an $\epsilon$-accurate solution is $O(\frac{L}{\mu} \log \frac{1}{\epsilon})$), for $\mu$-strongly convex and $L$-smooth functions, but need to compute an expensive projection step in (almost) every iteration. On the other hand, conditional gradient methods (such as the Frank-Wolfe algorithm \cite{FrankWolfe1956}) solely rely on solving linear optimization (LO) problems in every iteration and the rates of convergence become dimension-dependent, e.g., the away-step Frank-Wolfe algorithm has a linear rate of $(1-\frac{\mu \rho^2}{LD^2})$, where $\rho$ is a polytope dependent geometric constant and $D$ is the diameter of the polytope \cite{Lacoste2015}.  

The standard Conditional Gradient method (CG) or the Frank-Wolfe algorithm (FW) \cite{FrankWolfe1956,Levitin_1966} has received a lot of interest from the ML community mainly because of its iteration complexity, tractability and sparsity of iterates. In each iteration, the CG algorithm computes the {\it Frank-Wolfe vertex} $\iter{v}{t}$ with respect to the current iterate and moves towards that vertex: 
\begin{align*}
    \iter{v}{t} &= \argmin_{{v} \in \mathcal{P}} \innerprod{\nabla h (\iter{x}{t})}{{v}},~~ \iter{x}{{t+1}} = \iter{x}{t} + \gamma_t (\iter{v}{t} - \iter{x}{t}), \quad \gamma_t\in[0,1]. 
\end{align*} 
CG's primary direction of descent is $\iter{v}{t} - \iter{x}{t}$ ($d_t^{\text{FW}}$ in Figure \ref{fig:proj}) and its step-size $\gamma_t$ can be selected, e.g., using line-search; this ensures feasibility of ${x}_{t+1}$.  {However, the FW} algorithm can only guarantee a sub-linear rate of convergence $O(1/t)$ for smooth and strongly convex optimization on a polytope \cite{FrankWolfe1956,Jaggi2013}, which is tight for this method \cite{Canon_1968,lan_2013}. An active area of research, therefore, has been to find other descent directions that can enable linear convergence. One reason for the standard CG's $O(1/t)$ rate is the fact that the algorithm might zig-zag as it approaches the optimal face, slowing down progress \cite{Lacoste2015,Canon_1968}. The key idea for obtaining linear convergence was to use the so-called {\it away-steps} that help push iterates to the optimal face: 
\begin{align}
    \iter{a}{t} &= \argmax_{{{v}} \in F} \innerprod{\nabla h(\iter{x}{t})}{{{v}}}, \text{ for } F \subseteq \mathcal{P}, \nonumber \\
    \iter{x}{t+1} &= \iter{x}{t} + \gamma_t (\iter{x}{t}-\iter{a}{t}), \text{ where } \gamma_t \in \mathbb{R}_+ \text{ such that } \iter{x}{t+1}\in \mathcal{P} \label{activeset},
\end{align}
thus, augmenting the potential directions of descent using directions of the form $\iter{x}{t} -\iter{a}{t}$, for some $\iter{a}{t} \in F$, where the precise choice of $F$ in \eqref{activeset} has evolved in CG variants. 

\begin{figure}[t]
    \vspace{-10 pt}
    \centering
    \includegraphics[scale = 0.5]{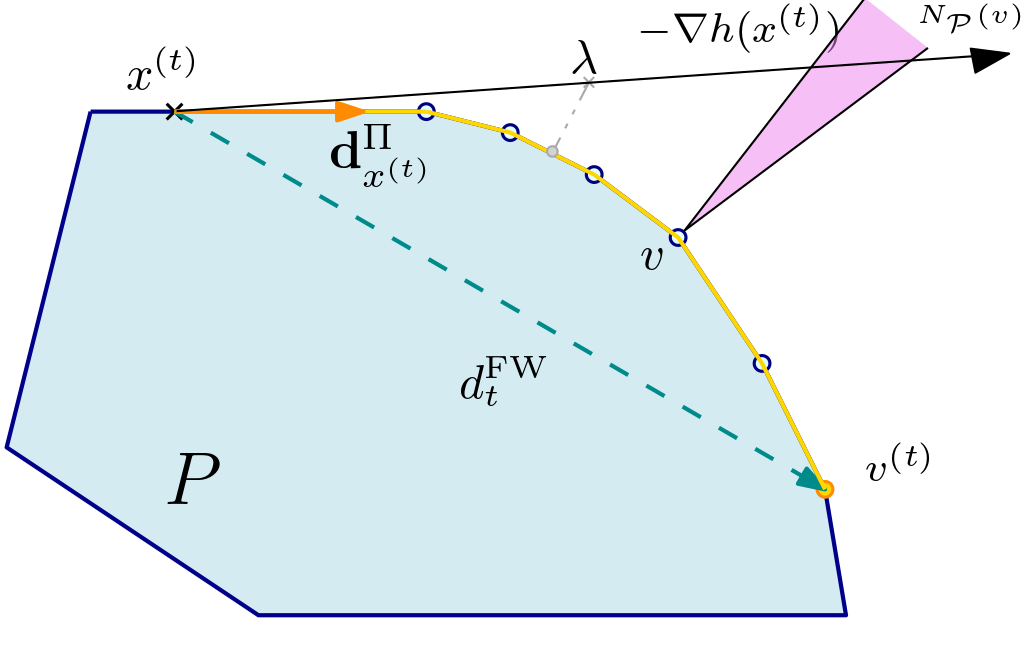}
    \hspace{20 pt}
    \includegraphics[scale = 0.53]{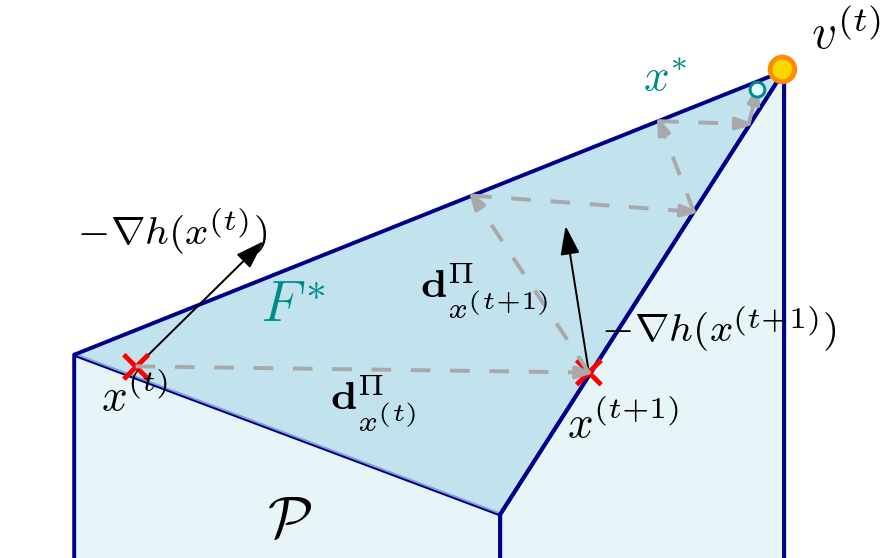}
     \vspace{-4 pt}
    \caption{ {Left:} Piecewise linearity of the parametric projections curve $g(\lambda) = \Pi_{\mathcal{P}}(\iter{x}{t} - \lambda \nabla h(\iter{x}{t}))$ (yellow line). The end point is the FW vertex $\iter{v}{t}$ and $d_t^{\text{FW}}$ is the FW direction. Note that $g(\lambda)$ does not change at the same speed as $\lambda$, e.g.,  $g(\lambda)={v}$ for each $\lambda$ such that $\iter{x}{t} - \lambda \nabla h(\iter{x}{t}) - {v}  \in N_{\mathcal{P}}({v})$ (purple normal cone). {Right:} Moving along the shadow as defined in \eqref{shadow def} might lead to arbitrarily small progress even once we reach the optimal face $F^* \ni x^*$. On the contrary, the away-step FW algorithm does not leave $F^*$ after a polytope-dependent iteration \cite{Marcotte1986}.}
    \label{fig:proj}
    \vspace{-10 pt}
\end{figure}

\subsection{Key Open Question:}

{In this work, we would like to contextualize various CG variants in the literature and elucidate the properties of feasible descent directions used. Although many conditional gradient variants obtain linear convergence over polyhedral sets (see Section \ref{sec:relatedwork}),} their rates depend on geometric, affine-variant constants tailored to polytopes -- such as the pyramidal-width \cite{Lacoste2015}, vertex-facet distance \cite{Beck2017}, or sparsity-dependent constants \cite{Bashiri_2017}, which have been shown to be essentially equivalent \cite{Pena2015}. For non-polyhedral sets (e.g., the $\ell_2$-ball), analogous constants can be arbitrarily small, leading to weak bounds. The iterates in CG algorithms are (basically) affine-invariant; this underpins the inevitability of a dimension-dependent factor in current discussions.

A natural question at this point is why are these different descent directions useful and which of these are necessary for linear convergence. If one had oracle access to the ``best'' local direction of descent for constrained minimization, what would it be and is it enough to obtain linear convergence like in unconstrained optimization? Additionally, can we circumvent convergence rates contingent upon the polytope's geometry (e.g., affine-variant constants like the pyramidal width)? We partially answer these questions below.

%\hm{Why do we have ``Key Open Question'' twice? Maybe get rid of the one above since we already have it in the section title? Or mabe make it in italics instead? }

\subsection{Contributions} {Our contributions in this work are three-fold: (1.) Descent Directions, (2.) Algorithmic contributions, and (3.) Special Cases and Experiments.} \vspace{0.2cm}

\begin{enumerate}
\item[(1.)] {{\bf Descent Directions (Section \ref{sec:descentdirections}):} We first demystify some of the directions of descent in constrained minimization. We show the following properties and relationships:}

\begin{enumerate}
\item {\bf Shadow Steps ($d_{\iter{x}{t}}^\Pi$)}: {We show that the ``best'' local feasible direction of descent, giving the maximum function value decrease in the diminishing neighborhood of the current iterate $\iter{x}{t}$, is the {\it directional derivative} ${d}_{\iter{x}{t}}^\Pi$ of the projection of the gradient, which we call the \emph{shadow} of the gradient:  

\begin{equation} \label{shadow def}
{d}_{\iter{x}{t}}^\Pi \coloneqq \lim_{\epsilon \downarrow 0} \frac{\Pi_{\mathcal{P}}(\iter{x}{t} - \epsilon \nabla h(\iter{x}{t})) - \iter{x}{t}}{\epsilon},    
\end{equation}
%\spok{The normalized shadow $d_{\iter{x}{t}}^\Pi/\|d_{\iter{x}{t}}^\Pi\|$, not $d_{\iter{x}{t}}^\Pi$ itself, solves this unit-norm maximization.}
where $\Pi_{\mathcal{P}}(y) = \argmin_{x\in \mathcal{P}}\|x-y\|^2$ is the Euclidean projection operator. These represent the best \emph{normalized feasible directions} of descent, i.e., {$${d}_{\iter{x}{t}}^\Pi/\|d_{\iter{x}{t}}^\Pi\| = \argmax_{d \in T_{\mathcal{P}}(\iter{x}{t}) : \|d\| = 1} \innerprod{-\nabla h(\iter{x}{t})}{d}$$} (Lemma \ref{bestlocaldirection}). Moreover, we show that $\|{d}_{\iter{x}{t}} ^\Pi\| = 0$ is a true measure of primal optimality gaps (in constrained optimization) without any dependence on geometric constants like those used in other CG variants (Lemma \ref{primalgapestimate}): (i) $\|{d}_{\iter{x}{t}} ^\Pi\| = 0$ if and only if $\iter{x}{t} \in \arg\min_{x\in \mathcal{P}}h(x)$ (with equality if $h(\cdot)$ is strongly convex) ; (ii) for $\mu$-strongly convex functions, we show that $\|{d}_{\iter{x}{t}} ^\Pi\|_2^2 \geq 2\mu(h(\iter{x}{t}) - h(x^*))$, which generalizes the well known PL inequality \cite{lojasiewicz1963topological}.} 

\item {\bf FW-Steps ($\iter{v}{t} - \iter{x}{t}$)}: Projected gradient steps provide a contraction in the objective, independent of the geometric constants or the polytope's facets; they correspond to the maximum descent (in terms of distance between points) that one can obtain on the polytope, by taking unconstrained gradient steps and then projecting back to the polytope. Under mild technical conditions (uniqueness of $\iter{v}{t}$), we show that the Frank-Wolfe vertices are in fact the projection of an infinite descent in the negative gradient direction (Theorem \ref{FWlimit}). This enables CG methods to greedily descend on the polytope maximally, compared to PGD methods, thereby providing a fresh perspective on FW-steps.

\item {\bf Away steps} and {\bf Pairwise-FW steps}: {We note that the shadow steps are the {\it best normalized away-direction} with steepest local descent. Shadow steps are in general convex combinations of potential active vertices minus the current iterate (Remark \ref{rem:away-step-obs}) and therefore lose combinatorial properties such as dimension drop in active sets \cite{Bashiri_2017}. Shadow steps can visit the same face multiple times in a zig-zagging manner, see Figure \ref{fig:proj} (right) for an example, unlike away-steps that use vertices and have a monotone decrease in dimension when they are consecutive. Further, we note that the pairwise steps ($\iter{v}{t} - \iter{a}{t}$) considered in CG variants bound the progress of the method using the inner product of the descent direction with the negative gradient. In this sense, pairwise steps are simply the {{\it sum of the FW-step and ``away'' directions}}, and the pairwise-FW algorithm, which uses only these steps and converges linearly (with a rate depending on geometric constants), but whose iteration complexity bound carries an additional $3n!$ factor due to the number of ``swap'' steps \cite{Lacoste2015,Bashiri_2017}.}
\end{enumerate}
\vspace{0.2cm}

\item[(2.)] {\bf Algorithmic Contributions (Sections \ref{sec:shadowwalk} and \ref{sec:interpolations})} We next use the above results to develop algorithms that cleanly interpolate between the projected gradient descent and conditional gradient methods. 

\begin{enumerate}
\item {\bf Algorithm} {\sc Shadow-Walk}: {In Section \ref{sec:shadowode}, we show that the shadow direction characterizes the infinitesimal movement induced by projected gradient steps, giving the continuous-time dynamics:
$$\dot{X}(t) = {d}_{X(t)}^\Pi,  \text{ with }X(0) = \iter{x}{0} \in \mathcal{P}.$$ However, the ODE is non-trivial to discretize due to non-differentiability of the curve. We propose the {\sc Shadow-Walk} algorithm in Section \ref{sec:shadow-walk-algorithm}, which traces the projections curve, until sufficient progress is achieved, and then updates the shadow based on the current iterate.

To unpack the movement of {\sc Shadow-Walk}, we characterize the structural properties of the curve $g(\lambda) = \Pi_{\mathcal{P}}({x} - \lambda \nabla h({x}))$ in Section \ref{sec:projections}. This is a piecewise linear curve \cite{murty1988linear,dantzig1968complementary}, where the breakpoints of the curve occur at points whenever there is a change in the normal cone. We show how to compute this curve for all $\lambda\geq0$, by giving the precise directions of each segment using shadow and in-face shadow steps (Theorem \ref{invariance}). Existing pivoting algorithms search for a basic feasible solution (BFS) of the KKT conditions at each breakpoint to find the next linear segment (in a similar vein as the simplex method) \cite{murty1978computational}. Wolfe's work on finding the nearest point within a polytope (i.e. projecting 0 onto a polytope) effectively mirrors the complementarity pivot operation \cite{Wolfe1976}, which was also shown to have exponential running time subsequently \cite{haddock_2018}. On the other hand, we simply characterize that the next linear is either the in-face direction or the shadow, depending on the interaction of the normal cone with the gradient. % conditions on the shadow and in-face projections to derive what the next segment is. 
We later show how our characterization simplifies the movements for special polytopes like hypercube and simplex, in Section \ref{sec:oracles}. We are thus able to break a single PGD step into multiple shadow descent steps, and show linear convergence with rate dependent on the number of facets, but independent of geometric constants like the pyramidal-width.}

\item {\bf Algorithm {\sc Shadow-CG}} (Section \ref{sec:shadowcg}):  {We next  combine these insights into a novel {\sc Shadow-CG} method which uses FW-steps and shadow steps (both over the tangent cone and minimal face), while enjoying linear convergence. This method takes FW steps if the following condition is true, and otherwise traces the projection curve up to sufficient progress: 
\begin{align}
\innerprod{-\nabla h(\iter{x}{t})}{{{d}_{\iter{x}{t}}^\Pi}/{\|{d}_{\iter{x}{t}}^\Pi\|}} \leq \innerprod{-\nabla h(\iter{x}{t})}{ {d}_t^{\text{FW}}}\label{eq:shadow-cg}
\end{align}

This algorithm generalizes the prior works of Garber and Meshi \cite{Garber_2016} and Bashiri and Zhang \cite{Bashiri_2017}. We can further tune an approximate version of this  condition to prioritize more FW steps if the gradient's norm is not too high: 
\begin{align}
{\sqrt{c}}\|\nabla h(\iter{x}{t})\| \leq \innerprod{-\nabla h(\iter{x}{t})}{{d}_t^{\text{FW}}}.
\label{eq:shadow-cg2}
\end{align}
If the condition \eqref{eq:shadow-cg2} is true, FW steps are taken and otherwise shadow is computed for the tracing the projections curve. This gives a computational speedup, while essentially retaining the rate of convergence of {\sc Shadow-CG}.}

\item {\bf Algorithm {\sc PGD-CG}} (Section \ref{sec:pgd-cg}): {Finally, we propose a direct mixing of the projected gradient steps and FW steps. We show that using condition \eqref{eq:shadow-cg2}, if we switch between FW steps (if condition true) and PGD steps (if condition false), we essentially retain the same rate of convergence while removing the dependence on the number of breakpoints. The parameter $c$ controls how aggressively the algorithm should compute projections.}  

{We summarize the runtimes of these methods in Table \ref{tab:complexities} below, along with the types of oracles required. Note that we remove the dependence on the pyramidal width and the number of breakpoints in the {\sc PGD-CG} algorithm. Note that PGD-CG is preferable when the problem admits a moderate or expensive projection oracle (i.e., $\Pi \gg L_O$) and asymmetry in the geometry leading to a small pyramidal width (i.e., $L_O \ll (D/\rho)^2$)}. One such example is the family of Cartesian products involving a matroid base polytope $B(f)$ and a permutahedron (see Example 1 in Section \ref{sec:pgd-cg} for details). %An important take-away from this analysis is the set of conditions that characterize when its better to take shadow steps (tracing the projections curve), projection steps, and Frank-Wolfe steps. %\hme{\sout{We summarize the intuition for these in Table \ref{tab:comparisons} in the appendix.} \hm{we removed this table from appendix for paper length reasons}}
%}
\end{enumerate}

\vspace{0.2cm}
\item[(3.)] {{\bf Computational Oracles and Experiments:} Finally, we exemplify the best of both regimes using the above interpolating algorithms. Our contributions here are the following:}  

\begin{enumerate}
\item {\bf Trace:} We introduce {\textsc{Trace}}, a single-gradient routine that traces the projection curve by chaining short descent
steps and segment-wise line-search until a target step size $s$ is reached. When invoked inside \textsc{Shadow-Walk} and \textsc{Shadow-CG} with $s=1/L$, {\textsc{Trace}} emulates a canonical PGD update: it traces the curve up to the $1/L$ scale and then minimizes along the active linear segment, guaranteeing progress at least as good as a PGD step. While the worst-case number of breakpoints along the projection curve is only known to admit a simple exponential bound in general (Theorem \ref{breakpoints}), we prove linear bounds for important special cases: the hypercube and the probability simplex.

\item {\bf Hypercube and Simplex:} 
We give oracles for computing projections and shadow directions for the special case of the simplex and the hypercube. We show that the shadow direction can be computed in $O(n)$ time for the hypercube (eq. \eqref{hypercube-shadow}) and in $O(n\log n)$ for the simplex (Lemma \ref{lem:simplex-shadow-structure}). This allows us to compute the entire projections curve for the hypercube and simplex in $O(n^2)$ and $O(n^2 \log n)$ time respectively, by showing structural results on the number of breakpoints in the projections curve. We further show that our proposed PGD-CG algorithm achieves a factor of $\Omega(n)$ reduction in iteration complexity and $\Omega(n/\log n)$ reduction in running time over the simplex, and an overall $\Omega(n^2)$ reduction in running time over the hypercube, compared to AFW. We summarize the iteration and runtime complexity of AFW, Shadow-Walk, Shadow-CG, PGD-CG and PGD in Table \ref{tab:complexities over special cases}\footnote{For the scaled simplex, the same oracles hold, but the pyramidal width hurts AFW runtime; we discuss this further in Example 1 in Section \ref{sec:pgd-cg}
, note that the linear optimization and shadow oracles are almost linear. In this case, PGD-CG attains a true interpolation between PGD and AFW, since $n \log n > n/c$ for appropriately chosen $c$, and beats both PGD (in reduced frequency of projection oracle calls) and AFW (in running time).}. 

\item {\bf Experiments}: We benchmark our algorithms on two tasks: video co-localization and Lasso regression. The video co-localization problem \cite{Joulin2014} minimizes a convex quadratic over a flow polytope. In this setting, we observe that
\textsc{Shadow-Walk} consistently requires the fewest iterations among all methods, benefiting from its ability to trace the projection curve up to the $1/L$ step size. Moreover, \textsc{Shadow-CG} and \textsc{PGD-CG} require more than an order of magnitude fewer iterations than the CG baselines AFW, PFW, and DICG, while remaining close to \textsc{PGD} in iteration count. Notably, although \textsc{PGD-CG} uses slightly more iterations than \textsc{PGD}, it achieves substantially lower wall-clock time—on par with CG methods—thereby combining the advantages of projection steps (fast progress per iteration) with the practical efficiency of projection-free updates. 

We observe similar behavior on the Lasso regression task, i.e., least squares over a scaled $\ell_1$ ball, which we handle via a simplex extended formulation \cite{combettes2020} to leverage our simplex oracle results. Since the simplex is a regime where the known AFW complexity bounds are essentially tight, it provides a particularly favorable baseline for away-step CG methods. Nevertheless, \textsc{Shadow-Walk} again achieves the best iteration counts, while \textsc{PGD-CG} significantly outperforms AFW in
iterations (by roughly two orders of magnitude) and attains substantially lower wall-clock time, again approaching the runtime of CG variants despite using projection information. The experiments confirm that the number of iterations spent inside {\textsc{Trace}} matches the $O(n)$ breakpoint bound proved. We remark that while both DICG and PFW are known to perform well in practice—and are outperformed by our methods in this work—they each come with limitations: PFW's linear convergence guarantee carries an exponential worst-case bound on the number of ``swap'' steps \cite{Lacoste2015}, and DICG is restricted to 0/1 polytopes \cite{Garber_2016}. 
\end{enumerate}
\end{enumerate}

\begin{table}[!t]
\centering
\begin{tabular}{|c|c|p{7cm}|}
\hline
\textbf{Algorithms} & \textbf{Runtime Complexity} & \textbf{Oracles} \\
\hline
\textsc{AFW} & $O\left((L_O + A_O)({D}/{\rho})^2\frac{L}{\mu}\log \frac{1}{\epsilon}\right)$ & Linear Optimization $L_O$, Away-Step Oracle $A_O$\\
\hline
\textsc{Shadow-CG} & $O\left((D^2{(L_O + S)} + \beta S) \frac{L}{\mu} \log (\frac{1}{\epsilon}) \right)$ & {Shadow Oracle $S$ for $\beta$ breakpoints, Linear Optimization Oracle $L_O$}\\
\hline
\textsc{Shadow-Walk} & $O\left(\beta S \frac{L}{\mu} \log (\frac{1}{\epsilon}) \right)$ & Shadow Oracle $S$ for $\beta$ breakpoints\\
\hline
\textsc{PGD-CG}  & $O\left((\frac{D^2}{c} L_O + \Pi) \frac{L}{\mu} \log (\frac{1}{\epsilon}) \right)$ & Projection Oracle $\Pi$, Linear Optimization Oracle $L_O$\\
\hline 
\textsc{PGD} & $O\left(\Pi \frac{L}{\mu} \log (\frac{1}{\epsilon}) \right)$ & Projection Oracle $\Pi$  \\
\hline
\end{tabular}
\caption{{Runtime complexity of various algorithms proposed in this work for obtained $\epsilon$-convergence, extrapolating between AFW and PGD regimes. Here, $\rho$ and $D$ is the pyramidal width and diameter of the polytope respectively \cite{Lacoste2015}, and $\beta$ is the number of breakpoints of the projections curve.}}
\label{tab:complexities}
\end{table}

\begin{table}[!t]
\centering
\begin{tabular}{|p{1.9 cm}||p{2.5 cm}|p{4.2 cm}|p{2.5 cm}|p{3.5 cm}|}
\hline
\multirow{2}{*}{\textbf{Algorithm}} & \multicolumn{2}{|c|}{\textbf{Simplex}} & \multicolumn{2}{|c|}{\textbf{Hypercube}} \\
\cline{2-5}
 & Iteration complexity & Runtime Complexity & Iteration complexity & Runtime Complexity \\
\hline \hline
\textsc{AFW} & $O\left(n \frac{L}{\mu} \log \left(\frac{1}{\epsilon}\right)\right)$ & $O\left(n^2 \frac{L}{\mu} \log \left(\frac{1}{\epsilon}\right)\right)$ & $O\left(n^2 \frac{L}{\mu} \log \left(\frac{1}{\epsilon}\right)\right)$ & $O\left(n^4 \frac{L}{\mu} \log \left(\frac{1}{\epsilon}\right)\right)$ \\ \hline

\textsc{Shadow-CG} & $O\left(\frac{L}{\mu} \log \left(\frac{1}{\epsilon}\right)\right)$ & $O\left(n^2 \log n \frac{L}{\mu} \log \left(\frac{1}{\epsilon}\right)\right)$ & $O\left(n \frac{L}{\mu} \log \left(\frac{1}{\epsilon}\right)\right)$ & $O\left(n^2 \frac{L}{\mu} \log \left(\frac{1}{\epsilon}\right)\right)$\\
\hline

\textsc{Shadow-Walk} & $O\left(\frac{L}{\mu} \log \left(\frac{1}{\epsilon}\right)\right)$ & $O\left(n^2 \log n \frac{L}{\mu} \log \left(\frac{1}{\epsilon}\right)\right)$ & $O\left(\frac{L}{\mu} \log \left(\frac{1}{\epsilon}\right)\right)$ & $O\left(n^2 \frac{L}{\mu} \log \left(\frac{1}{\epsilon}\right)\right)$\\
\hline

\textsc{PGD-CG} & { $O\left(\frac{{1}}{c} \frac{L}{\mu} \log \left(\frac{1}{\epsilon}\right)\right)$} & { $O\left( (\frac{n}{c} + n\log n) \frac{L}{\mu} \log \left(\frac{1}{\epsilon}\right)\right)$} & $O\left(\frac{n}{c} \frac{L}{\mu} \log \left(\frac{1}{\epsilon}\right)\right)$ & $O\left((\frac{n^2}{c} + n) \frac{L}{\mu} \log \left(\frac{1}{\epsilon}\right)\right)$\\
\hline

\textsc{PGD} & {$O\left(\frac{L}{\mu} \log \left(\frac{1}{\epsilon}\right)\right)$}  & {$O\left(n \log n \frac{L}{\mu} \log \left(\frac{1}{\epsilon}\right)\right)$}  & $O\left(\frac{L}{\mu} \log \left(\frac{1}{\epsilon}\right)\right)$ & $O\left(n \frac{L}{\mu} \log \left(\frac{1}{\epsilon}\right)\right)$\\
\hline

\end{tabular}
\caption{{Runtime complexity of various algorithms proposed in this work for obtained $\epsilon$-convergence, over the simplex and hypercube. The parameters we used are: $D^2 = 2$ (simplex) and $D^2 = n$ (hypercube); $\rho^2 \geq 1/n$ (simplex) and $\rho^2 \geq 1/n$ (hypercube); $S = \Pi  =  O(n\log n)$ (simplex) and $S = \Pi  =  O(n)$ (hypercube); $\beta =  O(n)$ (simplex) and $\beta  =  O(n)$ (hypercube); $L_0 =  O(n)$ (simplex) and $L_0  =  O(n)$ (hypercube); $A_0 =  O(n)$ (simplex) and $A_0  =  O(n^2)$ (hypercube). We also note that $c$ is the parameter we use to mix between projection and FW steps.
}
}
\label{tab:complexities over special cases}
\end{table}

\paragraph{Outline.} {We first discuss related work in Section \ref{sec:relatedwork} and next review preliminaries in Section \ref{sec:prelim}. We derive properties of descent directions in Section \ref{sec:descentdirections}. We then discuss our algorithmic contributions. We first present the continuous time dynamics for movement along the shadow and its discretization {\sc Shadow-Walk} in Section \ref{sec:shadowwalk}. We next discuss the properties of the projections curve in Section \ref{sec:projections}. We next interpolate between projected gradient descent and FW algorithms by proposing two algorithms: {\sc Shadow-CG} in Section \ref{sec:shadowcg}, and {\sc PGD-CG} in Section \ref{sec:pgd-cg}.  Both these methods mix FW steps and projected gradient or shadow steps appropriate to interpolate rates of convergence. We devote Section \ref{sec:oracles} to deriving specific oracles for the tracing the projections curve, and special cases of hypercube and simplex. We finally experimentally validate the trade-offs between these algorithms in Section \ref{sec:experiments}.} 

\section{Related Work}\label{sec:relatedwork} 
Linearly convergent Frank--Wolfe (FW) variants for polytopes all exploit active-set structure, differing in the face $F$ they descend over. As early as 1986, Gu{\'e}lat and Marcotte added away-steps (with $F$ the minimal face of the current iterate\footnote{The minimal face $F$ with respect to $\iter{x}{t}$ is a face of the polytope containing $\iter{x}{t}$ in its relative interior, i.e., all constraints active at $\iter{x}{t}$ are tight.}) and obtained asymptotic linear convergence \cite{Marcotte1986}. Garber and Hazan \cite{Garber_2013} gave the first \emph{global} linear-convergence guarantee by restricting the FW vertex to
a norm ball around the iterate---a mirror-descent-like regularization that, however, complicates the FW subproblem \cite{Garber_2016}. Lacoste-Julien and Jaggi \cite{Lacoste2015} proved linear convergence of away-step FW (with $F$ the
convex hull of the current active set), and Freund et al. \cite{Grigas_2015}
showed an $O(1/t)$ rate for convex objectives. Garber and Meshi \cite{Garber_2016}
showed pairwise steps over $0/1$ polytopes suffice for linear convergence, and
Bashiri and Zhang \cite{Bashiri_2017} extended this to general polytopes (at the
cost of two expensive oracles); both take $F$ to be the minimal face at
$\iter{x}{t}$. Braun et al. \cite{Braun2018} obtained linear convergence via a
simplex-projection subroutine aimed at better computational performance. See Bomze et al. \cite{bomze2024frank} for a recent survey, and
\cite{wirth2025fast, iommazzolinear} for refined linear-convergence guarantees
over polytopes and product polytopes. Our work is complementary: we use the
projection curve to give a unified view of projected-gradient and FW directions,
and then use this view to design algorithms that interpolate between the two
regimes.

A separate line of conditional gradient methods generalizes these techniques to convex feasible sets (not necessarily polyhedral). Conditional gradient sliding (CGS), for example, solves each projection subproblem approximately, within projected gradient descent, using Frank-Wolfe
steps, over general convex sets while retaining linear convergence
\cite{lan_2016}. We provide a complementary view with PGD-CG method, where we actively switch between a Frank-Wolfe step versus a projection step, prioritizing Frank-Wolfe steps when the alignment of the gradient is strong enough (irrespective of the iteration count). CGS, in contrast, strengthens the required approximation for projections as the iteration count increases, and therefore, prioritizes a smaller number of Frank-Wolfe steps in the earlier iterations. {Other related works establish the} convergence of vanilla CG over \emph{strongly-convex} domains: under a uniform lower bound on the gradient norm one obtains linear convergence, and Garber and Hazan later relaxed this assumption to obtain a quadratic rate \cite{garber2015faster}. 

We further note related work on computing projections over structured polytopes. In general, it was known that linear programming can be reduced to computing a minimum norm point over a simplex (Haddock uses multiple reductions to derive this \cite{haddock_2018}). However, note that an analogue of this already follows from our characterization of the Frank-Wolfe vertex (Theorem \ref{FWlimit}). Further, as early as 1986, Michelot gave efficient computations onto the simplex \cite{michelot1986finite} using the finite active set method, and more recently, viewing the simplex as a special case of cardinality-based polytopes, there have been fast algorithms that exploit submodularity of these polytopes \cite{wang2013projection, gupta2017combinatorial, lim2016efficient, moondra2021reusing}. Further, Duchi et al. 2008 \cite{duchi2008efficient} gave efficient algorithms for projecting onto the $\ell_1$-ball \cite{condat2016fast}. In this work, we add to this literature by analyzing the full parametric curve $\lambda \mapsto \Pi_{\Poly}(x-\lambda w)$, including its breakpoints, affine pieces composed of the shadow directions and the in-face shadow directions, and the limiting Frank-Wolfe endpoint. 

Next, the parametric projection curve can also be modeled through KKT conditions and linear complementarity \cite{cottle2009linear, murty1978computational, murty1988linear}, with roots in Wolfe's work \cite{wolfe1959simplex, Wolfe1976}. Such complementary-pivot methods, however, track the curve by changing complementary bases---analogous to simplex pivoting---so that even a \emph{single} linear segment may surface only after traversing several bases, and these methods can require exponentially many steps on special instances
\cite{dantzig1968complementary, murty1978computational}. By contrast, our
structural characterization (Theorem~6 and \textsc{Trace}) identifies the next
linear segment directly through shadow and in-face shadow directions. We conjecture that the number of breakpoints is polynomial\footnote{We further discuss the breakpoints conjecture in the context of extended formulations
\cite{fiorini2012extended, zadeh1973bad} in Appendix \ref{sec:appB}.} in the number of facets of a polytope, and show it is in fact linear in the dimension for the special cases of the simplex and hypercube. 

Following our work, \textsc{Trace}-like procedures have been used to avoid ``bad'' steps in CG variants, with rates governed by a slope condition rather than geometric constants \cite{rinaldi2020unifying, rinaldi2020avoiding}, and shadow directions have been used to accelerate FW \cite{kolmogorov2020practical}. A recent game-theory follow-up extends our simplex and hypercube results to \emph{acute} polytopes\footnote{A polytope is \emph{acute} if every pair of distinct constraint normals has non-positive dot product; the simplex and hypercube are both acute.}, showing $O(n)$ breakpoints in the projection curve \cite{ahunbay2024first}; finding broader polytope classes with few breakpoints remains open. This breadth contrasts with the pyramidal width governing AFW's convergence rate, which---reflecting the difficulty of computing it---has been evaluated explicitly only for the hypercube $(\rho=1/\sqrt{n})$ and the probability simplex $(\rho=2/\sqrt{n})$ \cite{Lacoste2015}. Our breakpoint bounds instead apply to all acute polytopes, a broader family.

\section{Preliminaries}\label{sec:prelim}
Let $\|\cdot\|$ denote the Euclidean norm and let $\mathcal{P}\subseteq \mathbb{R}^n$ denote a polytope defined in the form
\begin{equation} \label{polytope}
    \mathcal{P} = \{x \in \mathbb{R}^n: \innerprod{{a}_i}{x} \leq b_i \; \forall \; i \in [m]\}.
\end{equation}
We use $\mathrm{vert}(\mathcal{P})$ to denote the vertex set of $\mathcal{P}$. 
A function $h: \mathcal{P} \to \mathbb{R}$ is said to be \emph{$L$-{smooth}} if $h({y}) \leq h({x}) + \innerprod {\nabla h({x})}{{y}-{x}} + \frac{L}{2} \|{y}-{x}\|^2$ for all ${x},{y} \in \mathcal{P}$. Furthermore, $h: \mathcal{P} \to \mathbb{R}$ is said to be \emph{$\mu$-strongly-convex} if $h({y}) \geq h({x}) + \innerprod {\nabla h({x})}{{y}-{x}} + \frac{\mu}{2} \|{y}-{x}\|^2$ for all ${x},{y} \in \mathcal{P}$. Note that $0<\mu \leq L$. Throughout, we assume that $h$ is differentiable on an open neighborhood of $\mathcal{P}$, so that $\nabla h$ is well defined even when $\mathcal{P}$ is lower-dimensional. Let $D \coloneqq \sup_{{x},{y} \in \mathcal{P}} \|{x}- {y}\|$ be \emph{the diameter of $\mathcal{P}$}. We let ${x}^* = \arg\min_{x\in \mathcal{P}} h(x)$, where uniqueness follows from the strong convexity of the $h$. For any ${x} \in \mathcal{P}$, we let $I({x}) = \{i \in [m]: \innerprod{{a}_i}{{x}} =  b_i\}$ be the index set of active constraints at ${x}$. Similarly, let $J({x})$ be the index set of inactive constraints at ${x}$. Denote by $A_{I(x)} = [{a}_i]_{i \in I(x)}$ the sub-matrix of active constraints at $x$ and {${b}_{I(x)} = [b_i]_{i \in I(x)}$ the corresponding right-hand side}. The \emph{normal cone at a point ${x} \in \mathcal{P}$} is defined as 
\begin{equation} \label{def: normal cone}
\begin{aligned}
   N_{\mathcal{P}}({x}) &\coloneqq \{{y} \in \mathbb{R}^n : \innerprod{{y}}{{z} - {x}} \leq 0\; \forall {z} \in \mathcal{P}\} = \{{y} \in \mathbb{R}^n : \exists {\mu \geq 0} : {y} =  A_{I({x})}^\top \mu\},
\end{aligned}
\end{equation}
which is essentially the cone of the normals of constraints tight at ${x}$. The \emph{tangent cone} at a point $x \in \mathcal{P}$ is defined as
\begin{equation}\label{tangent cone}
    T_{\mathcal{P}}(x)\coloneqq \{d \in \mathbb{R}^n: A_{I(x)}d \leq 0 \},
\end{equation}
that is $T_{\mathcal{P}}(x) = \mathrm{cone}(\mathcal{P} - x)$ is the cone of feasible directions in $\mathcal{P}$ from ${x}$.

We next review some results on Euclidean projections over polytopes that we will use in this paper. Let $\Pi_{\mathcal{P}}({y}) = \argmin_{{x} \in \mathcal{P}} \frac{1}{2}\|{x}-{y}\|^2$ be the \emph{Euclidean projection operator}.
Using first-order optimality, we have ${x} = \Pi_{\mathcal{P}}( {y})$ if and only if
\begin{equation} \label{first-order optimality of projections}
    \innerprod{{y}-{x}}{{z}-{x}} \leq 0   \quad \forall {z} \in \mathcal{P} \quad \Longleftrightarrow \quad \; ({y}-{x}) \in N_{\mathcal{P}}({x}).
\end{equation}
It is well known that the Euclidean projection operator over convex sets is non-expansive (see e.g., \cite{bertsekas1997}): $\|\Pi_{\mathcal{P}}({y}) - \Pi_{\mathcal{P}}({x})\| \leq \|{y} - {x}\|$ for all ${x},{y} \in \mathbb{R}^n$. Given any point ${x} \in \mathcal{P}$ and ${w}\in \mathbb{R}^n$, let the \emph{directional derivative of the projection} of ${w}$ at $x$ be defined as: %(note it is the projection of $-w$ direction onto the tangent cone at $x$ for brevity of results):
\begin{equation}
    {d}_{{x}}^\Pi ({w}) \coloneqq \lim_{\epsilon \downarrow 0} \frac{\Pi_{\mathcal{P}}({x} - \epsilon {w}) - {x}}{\epsilon}.
\end{equation}

When ${w} = \nabla h(x)$, then we call ${d}_x^\Pi(\nabla h(x))$ the \emph{shadow of the negative gradient} at ${x}$, and use notation $d_{x}^{\Pi}$ for brevity. In \cite{Tapia_1972}, McCormick and Tapia show that ${d}_{{x}}^\Pi$ is the projection of $-\nabla h({x})$ onto the {tangent cone} at $x$: ${d}_{{x}}^\Pi = \argmin_{d \in T_{\mathcal{P}}(x)} \{ \|-\nabla h({x}) - d \|^2\} =  \argmin_{d \in \mathbb{R}^n} \{ \|-\nabla h({x}) - d \|^2 : {A}_{I({x})}{d} \leq {0}\}$, where the uniqueness of the solution follows from strong convexity of the objective. Further, let $\hat{d}^\Pi_x(\nabla h(x))$ be the projection of $-\nabla h({x})$ onto $\mathrm{Cone}(F - x) = \{d \in \mathbb{R}^n: {A}_{I({x})}{d} = {0}\}$, where $F$ is the minimal face of $P$ containing $x$\footnote{Note that $-({z} - x)$ is also a feasible direction at $x$ for any ${z} \in F$, since $x$ is in the relative interior of $F$ by definition. This implies that $\mathrm{Cone}(F - x) = -\mathrm{Cone}(F - x)$, and therefore $\mathrm{Cone}(F - x)$ is in fact a subspace.}. That is, $\hat{d}^\Pi_x(\nabla h(x))$ is the projection of $-\nabla h(x)$ onto the set of in-face feasible directions and can be computed in closed-form using: $\hat{d}^\Pi_x(\nabla h(x)) =  \argmin_{d \in \mathbb{R}^n} \{ \|-\nabla h({x}) - d \|^2 : {A}_{I({x})}{d}= {0}\} = ({I} -{A}_{I({x})}^\dagger {A}_{I({x})}) (-\nabla h(x)) $, where ${I} \in \mathbb{R}^{n \times n}$ is the identity matrix, and ${A}_{I({x})}^\dagger$ is the Moore-Penrose inverse of ${A}_{I({x})}$ (see Section 5.13 in \cite{meyer2000matrix} for more details). We refer to $\hat{d}^\Pi_x(\nabla h(x))$ as the \emph{in-face shadow} or \emph{in-face directional derivative}.

We will also use the following Moreau's decomposition theorem in our analysis:
\begin{theorem}[Moreau's Decomposition theorem \cite{moreau1962}] \label{moreau}
Let $\mathcal{K} \subseteq \mathbb{R}^n$ be a closed convex cone and let $\mathcal{K}^\circ$ be its polar cone,  that is, the closed convex cone defined by $\mathcal{K}^\circ=\{{y} \in \mathbb{R}^n \,\mid\, \innerprod{x}{y} \leq 0,\,\forall x \in\mathcal{K}\}$. Then, for $x,y, z \in \mathbb{R}^n$, the following statements are equivalent: 
\begin{itemize}
    \item [(i)]  $z=x+y$, $x \in \mathcal{K}$, $y \in \mathcal{K}^\circ$, and $\innerprod{x}{y}=0$; 
    \item [(ii)] $x = \Pi_{\mathcal{K}}(z)$ and $y = \Pi_{\mathcal{K}^\circ}(z)$.
\end{itemize}
\end{theorem} 

\section{Descent Directions}\label{sec:descentdirections} {We show three properties in this section: (i) shadow steps obtain the ``steepest'' feasible descent (Lemma \ref{bestlocaldirection}), (ii) shadow steps enjoy a direct primal gap estimate (Lemma \ref{primalgapestimate}), and (iii) Frank-Wolfe vertices can be viewed as an optimistic gradient descent (Theorem \ref{FWlimit}). We further discuss the implications of these results in interpreting the away directions and the pairwise FW steps. Lemmas \ref{bestlocaldirection} and Lemma \ref{primalgapestimate} will be useful in the convergence bounds of {\sc Shadow-CG} (Theorem \ref{thm:convergence-shadow-cg}), {\sc PGD-CG} (Theorem \ref{thm:convergence-pgd-cg}) algorithms in Section  \ref{sec:interpolations}.}

\subsection{Shadow Steps} We start with the first claim that the shadow is the best local feasible direction of descent in the following sense: it has the highest inner product with the negative gradient compared to any other normalized feasible direction. In other words, it is the analog of the negative gradient for constrained optimization. %{\color{red} (moved the proof from the appendix to here)}

\begin{restatable}[Steepest \underline{feasible} descent of Shadow Steps]{lemma}{bestlocaldirection}\label{bestlocaldirection} 
Let $\mathcal{P} \subseteq \mathbb{R}^n$ be a polytope defined as in \eqref{polytope} and let ${x} \in \mathcal{P}$ with gradient $\nabla h(x)$ be given. Let ${y}$ be any feasible {descent} direction at ${x}$ {with $\|y\| > 0$}, i.e., $\exists \gamma > 0$ s.t. ${x} + \gamma {y} \in \mathcal{P}$ {and $\innerprod{-\nabla h(x)}{y} > 0$}; note that the existence of such a $y$ forces $d_{{x}}^\Pi \neq {0}$, so the leftmost quantity below is well defined. Then,
\begin{equation} \label{eq:bestlocaldirection}
     \innerprod{- \nabla h({x})}{\frac{d_{{x}}^\Pi}{\|d_{{x}}^\Pi\|}}^2  \overset{(a)}{=}  \|d_{{x}}^\Pi \|^2 \overset{(b)}{\geq} \innerprod {d_{{x}}^\Pi}{\frac{y}{ \|y\|}}^2 \geq \innerprod {- \nabla h({x})}{\frac{y}{ \|y\|}}^2. 
\end{equation}
\end{restatable}

\textit{Proof.} {Moreau's decomposition theorem gives a unique decomposition $-\nabla h(x) = {p} + {d}_{{x}}^\Pi $ such that $\innerprod{{d}_{{x}}^\Pi}{{p}} = 0$, where ${p}$ is the projection of $-\nabla h(x)$ onto $N_{\mathcal{P}}(x)$. Therefore, the equality (a) follows from $\innerprod{-\nabla h(x)}{d_{x}^\Pi} = \|d_{x}^\Pi \|^2$. Next, note that (b) follows from the Cauchy-Schwarz inequality. So we are only left to show that for any descent direction $y$, we have 
\begin{equation} \label{bestdirection}
\innerprod{{d}_{x}^\Pi}{\frac{y}{\|y\|}}^2 \geq \innerprod{-\nabla h(x)}{\frac{y}{\|y\|}}^2.
\end{equation}}
Using first-order optimality of the shadow $d_x^{\Pi} = \argmin_{z \in T_{\mathcal{P}}(x)} \|-\nabla h(x) - z\|^2,$
we get: 
$$\innerprod{\nabla h(x)+ d_{x}^{\Pi}}{y - d_{x}^{\Pi}} \geq 0, \text{ for all } y\in T_{\mathcal{P}}(x).$$
Rearranging this, and using the fact that $\innerprod{-\nabla h(x)}{d_x^\Pi} = \|d_x^{\Pi}\|^2$, we get 
\begin{align}
\innerprod{d_x^{\Pi}}{y} \geq \|d_x^\Pi\|^2 + \innerprod{-\nabla h(x)}{y-d_x^\Pi} = \innerprod{-\nabla h(x)}{y} \label{new}.
\end{align}
Dividing both sides by $\|y\| > 0$ and then squaring, we obtain \eqref{bestdirection}. Since $y$ is a descent direction,  $\innerprod{-\nabla h(x)}{y} > 0$ and \eqref{new} implies $\innerprod{d_x^{\Pi}}{y}$ > 0 as well; hence, squaring preserves the inequality. \hfill $\square$  

%\spok{Same nonunique-minimizer issue: this should be membership in the solution set, or the statement should assume uniqueness.}

We next show that the shadow gives a true estimate of convergence to optimal, in the sense that $\|{d}_{x} ^\Pi\| = 0$ if and only if ${x} \in \arg\min_{x\in \mathcal{P}} h(x)$. On the other hand, note that $\|\nabla h({x})\|$ does not satisfy this property and can be strictly positive at the constrained optimal solution. Moreover, applying the Cauchy-Schwarz inequality to the left hand side of \eqref{eq:bestlocaldirection}, for any $x \in \mathcal{P}$, we have 
\begin{equation}\label{cauchyshadow}
    \|\nabla h(x)\|^2 \geq \|{d}_{x}^\Pi\|^2,
\end{equation}
and so we have a tighter primal gap bound. In addition, for $\mu$-strongly convex functions, we show that $ \|{d}_{x}^\Pi \|^2 \geq 2\mu (h(x) - h(x^*))$. Hence, $\|{d}_{\iter{x}{t}}^\Pi\|$ is a natural quantity for estimating primal gaps without any dependence on geometric constants like those used in other CG variants such as AFW. 
%\spok{Fix theorem statement: for merely convex $h$, $\|d_x^\Pi\|=0\Longleftrightarrow x\in\arg\min_{z\in\mathcal P}h(z)$. The displayed $x=x^*$ is only valid if the minimizer is unique.}
\begin{restatable}[Primal gap estimate]{lemma}{primalgapoptimality} \label{primalgapestimate}
Let $\mathcal{P} \subseteq \mathbb{R}^n$ be a polytope and fix any $x \in \mathcal{P}$. Consider any convex function $h: \mathcal{P} \to \mathbb{R}$ and let $x^* \in \argmin_{z\in \mathcal{P}} h(z)$. Then, $\|d_{{x}}^\Pi\| = 0$ if and only if $x \in \argmin_{z \in \mathcal{P}}h(z)$. Moreover, if $h$ is $\mu$-strongly convex over $\mathcal{P}$, then
\begin{equation} \label{deriv_performance_sc}
    \|{d}_{x}^\Pi \|^2 \geq 2\mu (h(x) - h(x^*)).
\end{equation}
\end{restatable}

\noindent \textit{Proof.} {If $\|{d}_{{x}}^\Pi\| = {0}$, then using Moreau's Decomposition we know that $-\nabla h(x) \in N_{\mathcal{P}}(x)$, which implies $x \in\arg\min_{z \in \mathcal{P}} h(z)$. Conversely, at any  optimum $x^*$, using first-order optimality we know that {$-\nabla h(x^*) \in N_{\mathcal{P}}(x^*)$, and therefore $\|{d}_{{x^*}}^\Pi\| = {0}$}.}

%Then, either \spok{handle trivial $x=x^*$ case before dividing by $\|x^*-x\|$; half-sentence sufficient} 
Next, assume that $h$ is $\mu$ strongly convex. Then, $x^*$ is unique, and $\|d_{{x}}^\Pi\| = 0$ implies $x = x^*$. Next suppose $x \neq x^*$. Then, using the strong convexity inequality applied with $y \leftarrow x + \gamma (x^*-x)$ and $x \leftarrow x$ we obtain: 
    \begin{align*}
        h(x + \gamma(x^*-x)) - h(x) &\geq \gamma \innerprod {\nabla h(x)}{x^*-x}  + \frac{\mu\gamma^2 \|x^*-x\|^2 }{2}\\
        &\geq -\frac{\innerprod {-\nabla h(x)}{x^*-x}^2} {2\mu \|x^*-x\|^2},
    \end{align*}
    where the last inequality follows by minimizing the right hand side over $\gamma$. As the RHS is independent of $\gamma$, we can set $\gamma = 1$ to get 
    \begin{equation}\label{eq:scprimalgap}
        h(x) - h(x^*) \leq \frac{\innerprod {-\nabla h(x)}{x^*-x}^2} {2\mu \|x^*-x\|^2}.
    \end{equation}
    {Noting that $\innerprod{-\nabla h(x)}{x^* - x} > 0$ by convexity, i.e. $x^* - x$ is a descent direction, and} applying Lemma \ref{bestlocaldirection} with $y = x^* - x$ completes the proof:
    $$    \innerprod{\nabla h(x)}{\frac{{d}_{x}^\Pi}{\|{d}_{x}^\Pi\|}}^2  =\|{d}_{x}^\Pi \|^2 \geq \innerprod {-\nabla h(x)}{\frac{x^* - x}{ \|x^* - x\|}}^2 \geq 2 \mu (h(x) - h(x^*)),$$
where the last inequality follows from \eqref{eq:scprimalgap}.
  \hfill $\square$   
 
The above lemma is a generalization of the PL-inequality \cite{lojasiewicz1963topological}, {since we can bound the primal-gap in terms of the norm of the shadow} in \eqref{deriv_performance_sc}, instead of the gradient in PL-inequality. This is also a stronger condition compared to the analogous inequalities which involves bounding with respect to the Frank-Wolfe gap. {We further highlight a connection between the shadow and the \emph{gradient mapping}. Given any scalar
$\eta > 0$, the gradient mapping is typically defined as $G_\eta(x) \coloneqq \eta (x - \Pi_{\mathcal{P}} (x -\nabla h(x)/\eta)) = \eta (x - g(1/\eta))$. Here, $\eta$ is the inverse step-size, so that the corresponding projected-gradient step uses step-size $1/\eta$. Since $d_x^\Pi=\lim_{\lambda\downarrow 0}(g(\lambda)-x)/\lambda$, we have
\[
    -G_\eta(x)=\frac{g(1/\eta)-x}{1/\eta}\longrightarrow d_x^\Pi
    \qquad\text{as }\eta\to\infty.
\]

%\spok{Fixed sign/scaling: the shadow is the limit of $-G_\eta$, not of $\eta^{-1}G_\eta$. Please confirm this convention matches the response-letter wording and the Nesterov/reduced-gradient citation. Check theorem number: this is Theorem 2.2.7 in the currently cited Introductory Lectures/\cite{Nesterov2013}; the same gradient-mapping property appears as Theorem 2.2.13 in Nesterov's 2018 Lectures.}
Equivalently, under our sign convention, $G_\eta(x)\to -d_x^\Pi$ as the projected-gradient step-size $1/\eta$ tends to zero; moreover, if $1/\eta$ lies on the first linear segment of the projection curve, then the identity $-G_\eta(x)=d_x^\Pi$ holds exactly. There are known finite-step projected-gradient/gradient-mapping bounds analogous to Lemma \ref{primalgapestimate} (e.g., Theorem 2.2.7 of \cite{Nesterov2013}), but those results are derived in the finite-step gradient-mapping regime and critically use step-size conditions involving smoothness, such as $\eta \geq L$. Therefore, they do not directly imply Lemma \ref{primalgapestimate}, which is formulated directly in terms of the infinitesimal shadow and does not require smoothness in order to establish our primal gap characterization.} For completeness, we note that when $h$ is additionally $L$-smooth, Lemma \ref{primalgapestimate} \emph{can} be recovered from the classical statement: applying it at $\eta=L$ and combining with the monotonicity of $\|G_\eta(x)\|$ in $\eta$ (see, e.g., Theorem 10.9 in \cite{beck2017first}), whose limit as $\eta\to\infty$ is $\|d_x^\Pi\|$, yields $\|d_x^\Pi\|^2 \geq \|G_L(x)\|^2 \geq 2\mu(h(x)-h(x^*))$. Our statement is therefore a strict generalization: it dispenses with smoothness altogether and is phrased directly in terms of the infinitesimal shadow.

%\hm{below is the old text, which had a mistake that I caught while writing the response document. This was the old text for your reference:\\ \\Indeed, as $\eta \rightarrow 0$, the reduced gradient $\frac{1}{\eta}G_{\eta}(x_t)$ converges to the shadow direction. There are known bounds analogous to Lemma \ref{primalgapestimate} (e.g., see Theorem 2.2.7 of \cite{Nesterov2013}), but they critically use the condition that $\eta \geq L$, and therefore do not imply Lemma \ref{primalgapestimate}, which does not need the function $h(\cdot)$ to be smooth to begin with.} }
\subsection{Frank-Wolfe Steps}

We next show that {under mild technical conditions, the FW vertex is the projection of an infinite descent in the direction of the negative gradient. Thus, FW vertices are able to obtain the maximum movement in the negative gradient direction while remaining feasible.} 

\begin{restatable}[Optimism of FW vertex]{theorem}{FWlimit}\label{FWlimit}
{Let $\mathcal{P} \subseteq{R}^n$ be a polytope and let ${x} \in \mathcal{P}$, and let $F$ be the face of $\mathcal{P}$ that is the $\argmin_{{v} \in \mathcal{P}} \innerprod{\nabla h({x})}{{v}}$. Consider the projections $g(\lambda) = \Pi_{\mathcal{P}}({x} - \lambda \nabla h({x}))$ for any given $\lambda\geq 0$. Then, the projection at $\lambda \to \infty$ is $
    \lim_{\lambda \to \infty} g(\lambda) = {v}^* = \argmin_{{v} \in F} \|{x} - {v}\|^2. 
$ In particular, if $F$ is a vertex, then $\lim_{\lambda \to \infty} g(\lambda) = {v}^*$, which is the FW vertex.}
\end{restatable}

 \noindent \textit{Proof.}
{From the definition of the face $F$, we know that $F \neq \emptyset$, and $\argmin_{v\in F}\|x-v\|^2$ exists. Let us call this $v^*$. If $F = \mathcal{P}$ (which happens for example when $\nabla  h(x) =0$, or more generally when $\nabla  h(x)$ is orthogonal to the affine hull of $\mathcal{P}$), then  $g(\lambda) = x$ for all $\lambda \geq 0$ by first order optimality:
$$\innerprod{(x - \lambda \nabla h({x}))- x}{z - x} = \innerprod{-\lambda \nabla h(x)}{z - x} = 0 \qquad \forall \; z \in \mathcal{P},$$
where the last equality follows from the assumption that $F=\mathcal{P}$. Therefore, the theorem holds trivially in this case, so assume otherwise ($F \subsetneq \mathcal{P}$).

\begin{itemize}

    \item[(i)] We first show that for $\bar{\lambda}$ large enough, we have $g(\lambda) \in F$ for all $\lambda > \bar{\lambda}$. Define 
    $$\delta_F \coloneqq
    \min_{w \in \mathrm{vert}(\mathcal{P}) \setminus \mathrm{vert}(F)}
    \innerprod{\nabla h(x)}{w - v^*}, \quad \text{and set} \quad 
    \bar{\lambda}
    \coloneqq
    \frac{D^2}{\delta_F},
    $$
    where $D$ is the diameter of $\mathcal P$. This quantity is well-defined because $F \subsetneq \mathcal{P}$, and therefore the set $\mathrm{vert}(\mathcal{P}) \setminus \mathrm{vert}(F)$ is non-empty. Moreover, for any $w \in \mathrm{vert}(\mathcal{P}) \setminus \mathrm{vert}(F)$, we have $w \notin F$, and by the definition of $F$, $\innerprod{\nabla h(x)}{w}> \innerprod{\nabla h(x)}{v^*}$. Hence $\innerprod{\nabla h(x)}{w-v^*}>0$, and so $\delta_F>0$.

    Now, suppose for a contradiction that there exists some $\hat{\lambda} > \bar{\lambda}$ such that $\hat{v} \coloneqq g(\hat{\lambda}) \notin F$. Note that we can rewrite $g(\lambda) = \argmin_{v \in \mathcal{P}} \| x - \lambda \nabla h(x) - v\|^2$ using the proximal view of Euclidean projections as: 
    \begin{equation}\label{proximal view in proof}
        g(\lambda) = \argmin_{v \in \mathcal{P}} \phi_\lambda(v), \qquad \text{where } \phi_\lambda(v) := \|x - v\|^2 + 2\lambda \langle \nabla h(x), v \rangle.
    \end{equation}
    We will show that $\phi_{\hat{\lambda}}(\hat{v}) > \phi_{\hat{\lambda}}({v}^*)$, which contradicts the optimality of $\hat{v}$ as the projection $g(\hat \lambda)$.

    Write $\hat v$ as a convex combination of vertices: $\hat v = \sum_{u\in \mathrm{vert}(F)}\alpha_u u+ \sum_{w\in \mathrm{vert}(\mathcal P)\setminus \mathrm{vert}(F)}
    \beta_w w$, where $\alpha_u,\beta_w\ge 0$ and the coefficients sum to $1$. Let $\beta \coloneqq \sum_{w\in \mathrm{vert}(\mathcal P)\setminus \mathrm{vert}(F)}
    \beta_w$. Because $\hat v\notin F$, we have $\beta>0$. Now define $v_O \coloneqq \frac{1}{\beta} \sum_{w\in \mathrm{vert}(\mathcal P)\setminus \mathrm{vert}(F)} \beta_w w$. If $\beta<1$, define $v_F \coloneqq \frac{1}{1-\beta} \sum_{u\in \mathrm{vert}(F)}\alpha_u u$. Then $v_F\in F$ (since $\sum_{u\in \mathrm{vert}(F)}\alpha_u = 1- \beta$) and $\hat v=(1-\beta)v_F+\beta v_O$. If $\beta=1$, choose any $v_F\in F$. In this case, $\hat v=v_O$. Hence, since $\innerprod{\nabla h(x)}{\cdot}$ is constant on $F$,
we have $\innerprod{\nabla h(x)}{v_F} = \innerprod{\nabla h(x)}{v^*}$. Moreover, by the definition of $\delta_F$ and since $v_O$ is a convex combination of vertices outside $F$, $\innerprod{\nabla h(x)}{v_O-v^*} \geq \delta_F$. Therefore
\begin{align}
    \innerprod{\nabla h(x)}{\hat v - v^*} &= \innerprod{\nabla h(x)}{(1-\beta)v_F+\beta v_O-v^*} \nonumber \\
    &= (1-\beta)\innerprod{\nabla h(x)}{v_F-v^*} + \beta\innerprod{\nabla h(x)}{v_O-v^*} \\
    &\geq \beta\delta_F. \label{linear gap outside mass}
\end{align}

Next, using the definition of $v^*$, we have $\|x-v_F\|^2 \geq \|x-v^*\|^2$. Also, $\hat v-v_F = \beta(v_O-v_F)$, and since $v_O,v_F\in \mathcal P$, we have $\|\hat v-v_F\| \leq \beta D$. Therefore,
    \begin{align}
    \|x-\hat v\|^2-\|x-v^*\|^2
    &\ge
    \|x-\hat v\|^2-\|x-v_F\|^2
    &&(\text{using} \|x-v_F\|^2 \geq \|x-v^*\|^2) \nonumber\\
    &=
    \innerprod{(x-\hat v)+(x-v_F)}{v_F-\hat v}
    &&(\text{using } \|a\|^2-\|b\|^2=\innerprod{a+b}{a-b}) \nonumber\\
    &\ge
    -
    \bigl\|(x-\hat v)+(x-v_F)\bigr\|\,\|v_F-\hat v\|
    &&(\text{by Cauchy-Schwarz: } \innerprod{a}{b} \geq -\|a\|\|b\|) \nonumber\\
    &\geq
    -
    \bigl(\|x-\hat v\|+\|x-v_F\|\bigr)\|v_F-\hat v\|
    &&(\text{by the triangle inequality}) \nonumber\\
    &\ge
    -2D\cdot \beta D
    &&(\text{using } \|\hat v-v_F\|\le \beta D) \nonumber\\
    &=
    -2\beta D^2.
    \label{quadratic gap outside mass}
\end{align}
Combining \eqref{linear gap outside mass} and \eqref{quadratic gap outside mass}, we see that the linear term in the projection dominates as follows:
    \begin{align*}
        \phi_{\hat{\lambda}}(\hat v)-\phi_{\hat{\lambda}}(v^*)
        &= \|x-\hat v\|^2-\|x-v^*\|^2 + 2\hat\lambda \innerprod{\nabla h(x)}{\hat v-v^*}\\
        &\geq -2\beta D^2 + 2\hat\lambda\beta\delta_F\\
        &= 2\beta(\hat\lambda\delta_F-D^2) >0,
    \end{align*}
    where the last inequality follows from $\beta>0$ and $\hat\lambda>\bar\lambda=D^2/\delta_F$. Thus, $\phi_{\hat{\lambda}}(\hat v)>\phi_{\hat{\lambda}}(v^*)$,
    which contradicts the optimality of $\hat v$. Hence $g(\lambda)\in F$ for all $\lambda>\bar\lambda$.

    \item[(ii)] In (i), we have shown that $g(\lambda) \in F$ for any {$\lambda > \bar{\lambda}$}. We now additionally show that $\hat{v} \coloneqq g(\hat{\lambda})$ necessarily satisfies $\hat{v} = v^*$ for any $\hat{\lambda} > \bar{\lambda}$. Again, suppose for a contradiction that $\hat{v} \neq v^*$. Then,
    \begin{align*}
        \phi_{\hat{\lambda}}(\hat{v}) &= \|x - \hat{v}\|^2 + 2\hat{\lambda} \langle \nabla h(x), \hat{v} \rangle\\
        & = \|x - \hat{v}\|^2 + 2\hat{\lambda} \langle \nabla h(x), v^* \rangle &&(\hat{v} \in F)\\
        & >  \|x - v^* \|^2 + 2\hat{\lambda} \langle \nabla h(x), v^* \rangle &&(\text{using definition of $v^*$})\\
        & = \phi_{\hat{\lambda}}({v}^*),
    \end{align*}
    which again contradicts the optimality of $\hat{v}= g(\hat{\lambda})$.\hfill $\square$  
\end{itemize}
}

\subsection{Away-Steps Approximate the In-Face Shadow Steps}
We now observe that in-face shadow steps can be interpreted as the optimal away-step direction. Consider any point $x \in \mathcal{P}$, and let $I$ denote the set of active constraints at $x$. Let the \emph{minimal face} of $\mathcal{P}$ containing $x$ be $$ F = \{\, y \in \mathcal{P} \mid A_I y = b_I \,\}.$$
Define $\mathcal{A}(x) := \{S \subseteq \mathcal{V}(\mathcal{P}) : x \text{ is a proper convex combination of all } v \in S \}$ to be the collection of all active sets representing $x$. We know that every $S \in \mathcal{A}(x)$ satisfies $S \subseteq \mathcal{V}(F)$, i.e., any active set for $x$ must lie entirely within its minimal face \cite{Bashiri_2017}. Further, recall that away steps are defined by choosing a vertex in this face that maximizes alignment with the gradient:  
\begin{align}
    \iter{a}{t} &= \argmax_{{{v}} \in F} \innerprod{\nabla h(\iter{x}{t})}{{{v}}}, \text{ for } F \subseteq \mathcal{P}, \nonumber \\
    \iter{x}{t+1} &= \iter{x}{t} + \gamma_t (\iter{x}{t}-\iter{a}{t}), \text{ where } \gamma_t \in \mathbb{R}_+ \text{ such that } \iter{x}{t+1}\in \mathcal{P} \label{activeset1},
\end{align}
where $F$ could be the entire minimal face, or any active set containing $x$. Thus, away steps search for a feasible descent direction within $F$, emanating from vertices, that maximizes the inner product with the gradient. 

\begin{remark}\label{rem:away-step-obs} The \emph{in-face shadow} is precisely the projection of the negative gradient onto the minimal face, and therefore gives the steepest feasible descent direction in $F$. Hence, away steps can be viewed as an approximation to computing the in-face shadow.
\end{remark}

Further, we note that the pairwise steps ($\iter{v}{t} - \iter{a}{t}$) considered in CG variants bound the progress of the method using the inner product of the descent direction with the negative gradient. In this sense, pairwise steps are simply the {{\it sum of the FW-step and ``away'' directions}}, and the pairwise-FW algorithm, which only uses these steps does converge linearly, but the iteration bound carries an additional $3n!$ factor from the number of ``swap steps'' \cite{Lacoste2015,Bashiri_2017}.

\section{Algorithm {\sc Shadow-Walk}}\label{sec:shadowwalk}
We now turn our attention to converting these directional insights into algorithms. In Section \ref{sec:descentdirections}, we showed that the shadow of the negative gradient ${d}_{\iter{x}{t}}^\Pi$ is indeed the best ``local'' direction of descent (Lemma \ref{bestlocaldirection}), i.e., it is the steepest descent direction for constrained optimization and is a true measure of the primal gap since convergence in $\|{d}_{\iter{x}{t}} ^\Pi\|$ implies optimality (Lemma \ref{primalgapestimate}). In the unconstrained setting, it is known that the algorithm obtained by walking along the negative gradient has linear convergence for smooth and strongly convex functions \cite{Karimi_2016}. The natural question is if a shadow-descent algorithm that walks along the directional derivative with respect to the negative gradient at iterate $\iter{x}{t}$, converges linearly.

We start by answering that question positively in  continuous-time (in Section \ref{sec:shadowode}). We next show a discretization of this method that converges in linear time, and call it {\sc Shadow-Walk} in Section \ref{sec:shadow-walk-algorithm}. The number of oracle calls for computing the shadow required in this method depends on the number of breakpoints of the parametric projections curve, which we discuss in Section \ref{sec:projections}.%\footnote{{\color{red}removed mirror descent stuff.}}}

\subsection{ODE for Moving in the Shadow of the Gradient}\label{sec:shadowode}
We present the analysis for continuous time dynamics of the ODE for tracing the shadow. Let $X(t)$ denote the continuous-time trajectory of tracing the shadow, and $\dot{X}$ denote the time-derivative of $X(t)$, i.e., $\dot{X}(t) = \frac{d}{dt}X(t)$. The continuous-time dynamics of tracing the shadow or the directional derivative are simply:
\begin{equation}
    \dot{X}(t) = {d}_{X(t)}^\Pi. \label{ODe for shadow}
\end{equation} 
These dynamics are initialized with $X(0) = \iter{x}{0}\in \mathcal{P}$. 
\iffalse
We will show that these are equivalent to the projected gradient descent's continuous time dynamics, given by: 
\begin{equation} \label{pgd-ode}
    \dot{Z}(t) = -\nabla h(X(t)), \quad X(t) = \argmax_{x \in \mathcal{P}} \{\innerprod{x}{Z(t)} - {\frac{1}{2}\|x\|^2}\} = \argmin_{x\in \mathcal{P}} \|Z(t)-x\|^2,  
\end{equation}
and initialized at $Z(0)=X(0) =x^{(0)} \in \mathcal{P}$. 
\begin{restatable}{theorem}{odetheorem}
\label{ode theorem}{The dynamics for projected gradient descent \eqref{pgd-ode} are equivalent to the shadow dynamics $\dot{X}(t) = {d}_{X(t)}^\Pi$ with the same initial conditions $Z(0)=X(0) = \iter{x}{0}\in \mathcal{P}$.}
\end{restatable} 
\hm{Check theorem 3}
Theorem \ref{ode theorem} holds for more general differentiable Legendre-type mirror maps and we include the general theorem with its proof in the . 
\fi
As established in Section \ref{sec:prelim}, the
shadow direction is the directional derivative of the Euclidean projection operator and is equivalently the projection of the negative gradient onto the tangent cone. Thus, the shadow ODE captures the infinitesimal feasible movement induced by projected-gradient steps, while operating directly in the primal space. An analogous primal-space representation holds for mirror descent under differentiable Legendre-type mirror maps; we establish this connection in
Theorem \ref{ode theorem full} in Appendix \ref{sec:appendixODE}.

Next, using the primal gap estimate from Lemma \ref{primalgapestimate}, we show that these shadow dynamics enjoy a global linear convergence rate for strongly convex objectives, without requiring smoothness:
\begin{restatable}{theorem}{odeconver}
\label{ode_conver} 
Let $P \subseteq \mathbb{R}^n$ be a polytope and suppose that
$h : P \to \mathbb{R}$ is $\mu$-strongly convex over $P$.
Let $X(t)$ be any absolutely continuous solution of the shadow dynamics
$\dot X(t)=d^\Pi_{X(t)}$ with initial condition $X(0)=x^{(0)}\in P$.
Then, for each $t\geq 0$, we have $X(t)\in P$. Moreover, the primal gap associated with the shadow dynamics decreases as: $$h(X(t)) - h({x}^*) \leq e^{-2\mu t} {(h(X(0)) - h({x}^*))}.$$
\end{restatable}

\noindent \textit{Proof.}
    Define $w(X(t))\coloneqq h(X(t)) - h({x}^*)$. First, any absolutely continuous solution of the shadow dynamics satisfies $X(t) \in \mathcal{P}$ for all $t \geq 0$, since $\dot X(t) \in T_{\mathcal{P}}(X(t))$ by construction. %$X(t) \in \mathcal{P}$ for all $t \geq 0$ since the dynamics of the shadow ODE are the same as those of projected gradient descent, by Theorem \ref{ode theorem}. 
    To derive the convergence rate, we use a Lyapunov argument. We let $w(X(t))$ be the Lyapunov potential function. Using the chain rule we have 
\begin{align}
    \frac{d w(X(t))}{dt} &= \innerprod{\nabla h(X(t))}{\dot{X}(t)}\\
    &= \innerprod{\nabla h(X(t))}{{d}_{X(t)}^\Pi} \label{ode1}\\
     &= -\|{d}_{X(t)}^\Pi\|^2 \label{ode2}\\
     &\leq -2 \mu w(X(t)), \label{ode3}
\end{align}
where we used the fact that $\dot{X}(t) = {d}_{X(t)}^\Pi$ in \eqref{ode1}, the fact that $\innerprod{-\nabla h(X(t))}{d_{X(t)}^\Pi} = \|d_{X(t)}^\Pi \|^2$ in \eqref{ode2}, and finally the primal gap estimate from Lemma \ref{primalgapestimate} in \eqref{ode3}. Integrating both sides of the above inequality (and using Gr\"{o}nwall's inequality \cite{gronwall1919note}) yields the result.
\hfill $\square$    

\subsection{{\sc Shadow-Walk} Algorithm}\label{sec:shadow-walk-algorithm}

Note that the shadow ODE might visit a large number of small facets (see Figure \ref{fig:proj}) with decreasing step-sizes. It can also visit the same face a large number of times. In both cases, the feasible step size along the shadow might be arbitrarily small, in which case we cannot show a sufficient decrease in primal gap to prove a global convergence rate. We remark that McCormick and Tapia  \cite{Tapia_1972} proved that the limit point of the sequence of iterates obtained by walking along the shadow (with appropriately chosen step sizes to ensure feasibility) converges to a stationary point. However, they did not give global convergence rates, and this still remains an open question. To resolve these two issues, we borrow the notion of ``enough progress'' from projected gradient descent. We consider the parametric projections curve $g_{x, {w}}(\lambda) = \Pi_{\mathcal{P}}({x} - \lambda {w})$, for $\lambda \geq 0$, under Euclidean projections, for any given direction $w \in \mathbb{R}^n$.

If one can compute the directions in which the projection curve moves from any iterate $x$ for $w = \nabla h(x)$, and walks along this curve until a PGD step-size of $1/L$ is effectively achieved in an unconstrained descent (where $L$ is the smoothness constant), then one can still show convergence. This conceptual process remedies the aforementioned issue of diminishing step-sizes on the same face, and enables us to prove a global linear convergence rate, which depends on the number of breakpoints of the projections curve (we explore the latter in Section \ref{sec:projections}). \\

\noindent 
{\bf {\sc Shadow-Walk} Algorithm} (Alg. \ref{alg:shadowwalk}): {At each iterate $\iter{x}{t}$, trace the projections curve until ``enough'' progress equivalent to $1/L$ step size is ensured. This gives iterate $\iter{x}{t+1}$.}

{The natural question at this point is if the projections curve is easy to compute, and if it breaks down into a ``few'' descent directions of the polytope. We discuss this in more detail in Section \ref{sec:projections}, where we show that (i) indeed, each projections curve is piecewise linear, and (ii) is made up of (computable) shadow and in-face shadow directions. Given a shadow oracle to compute these directions, one can ``trace'' the projections curve appropriately. 

We discuss a {\sc Trace} oracle for tracing the projections curve in Section \ref{sec:oracles}. From the standard analysis of PGD, each trace of the projections curve results in a $(1-\frac{\mu}{L})$ contraction of the primal gap. Further, each trace requires $\beta$ shadow oracle calls, where $\beta$ is the bound on number of pieces of the projections curve. This is a start towards developing an analysis of descent algorithms, which do not rely on pyramidal width, but the number of breakpoints of the projections curve $\beta$, as we discuss next.} 

\begin{algorithm}[t]
\caption{{\sc Shadow-Walk} Algorithm}
\label{alg:shadowwalk}
\begin{algorithmic}[1]
\INPUT Polytope $\mathcal{P} \subseteq \mathbb{R}^n$, function $h: \mathcal{P} \to \mathbb{R}$ and initialization $\iter{x}{0} \in \mathcal{P}$.
\For{$t = 0, \dots. T$}
    \State ={$\iter{x}{t+1} \coloneqq$ {\textsc{Trace}}$(\iter{x}{t}, \nabla h(\iter{x}{t}), {1/L})$ . \Comment{ \texttt{trace projections curve using Theorem \ref{invariance} until total accrued step-size $\theta^*$ is at least $1/L$.}}}
\EndFor
\RETURN $\iter{x}{T+1}$
\end{algorithmic}
\end{algorithm} 

\subsection{Structure of the Projections Curve}
\label{sec:projections}
Although the shadow dynamics characterize the infinitesimal behavior of projected-gradient steps and achieve linear convergence, it is unclear how to discretize this continuous-time process and obtain a linearly convergent algorithm. To ensure feasibility, we may have arbitrarily small step-sizes, and therefore, cannot show sufficient progress in such cases. This issue occurs in existing FW variants that use some form of away-steps. In fact, many recent papers \cite{Lacoste2015, Bashiri_2017, combettes2020, Braun2017, Pena2015, Beck2017, garber2020revisiting, Garber_2016, dvurechensky2023generalized, pedregosa2020linearly} use a similar way of showing convergence by bounding the number of such `bad' steps with dimension reduction arguments, which crucially rely on maintaining iterates as a convex combination of vertices. {However, unlike away-steps in CG variants, we explore the relationship between shadow direction and the projected gradient descent by unpacking the structure of the projections curve next. This will enable us to develop a novel geometry-independent proof of convergence.} 

\subsubsection{Projections Curve as a Linear Complementarity Problem}\label{sec:projectioncurve}
{Define the parametric projections curve $g_{x, {w}}(\lambda) = \Pi_{\mathcal{P}}({x} - \lambda {w})$, for $\lambda \geq 0$, under Euclidean projections, for any given direction $w \in \mathbb{R}^n$. For brevity, we use $g(\cdot)$ when ${x}$ and ${w}$ are clear from context.} To see an example of the projections curve, we refer the reader to Figure \ref{fig:Normal Cone}. Even though our results hold for any ${w}\in \mathbb{R}^n$, we will prove the statements for ${w} = \nabla h(x_0)$ for readability in the context of this section.

The problem of computing the parametric projections curve $g(\lambda)$ for any $\lambda \in \mathbb{R}_+$ is: 
\begin{equation} \label{lcp-main}
    g(\lambda) = \begin{aligned}
        \quad \argmin~ &\frac{1}{2} \|x_0 - \lambda \nabla h (x_0)- x\|_2^2 \\
        \text{s.t. } &A x \leq b.\quad
    \end{aligned}
    \quad = \begin{aligned} \quad 
        \argmin~ &\frac{1}{2} \|x\|_2^2 - \innerprod{x_0 - \lambda \nabla h (x_0)}{x}\\
        \text{s.t. } &A x \leq b.\\
    \end{aligned}
\end{equation}
{Note that this is an instance of  a quadratic program. Let's consider the KKT conditions of \eqref{lcp-main}, using dual variables $u,v,$and $y$:
\begin{equation} \label{lcp4}
    \begin{aligned}
    &x = x_0 - \lambda \nabla h (x_0) - A^\top y,\\
    &v = b - Ax,\\
    &y^\top v = 0,\\
    &v\geq 0, y\geq 0. 
    \end{aligned}
\end{equation}

\noindent 
In particular, by setting
$$w = \left(\begin{array}{c}
     0  \\
     v 
\end{array}\right), \quad q_{\lambda} = \left(\begin{array}{c}
     -x_0 + \lambda \nabla h (x_0)  \\
     b
\end{array}\right)
, \quad  M = \left(\begin{array}{cc}
     I & A^\top \\
     -A & 0 
\end{array}\right), \quad z = \left(\begin{array}{c}
     x  \\
     y 
\end{array}\right), 
$$
we get that \eqref{lcp4} is equivalent to the following system: 
\begin{equation} \label{lc}
    \begin{aligned}
&w = q_{\lambda} + Mz, \qquad w^\top z =0, \qquad w, z[1:y] \geq 0. 
\end{aligned} \tag{LC}
\end{equation}}
{which is also known as the linear complementarity problem (see for e.g. \cite{murty1988linear}). This reformulation points to an intriguing connection between Wolfe's early work on solving quadratic programs using the simplex method \cite{wolfe1959simplex} and his subsequent algorithm for finding the nearest point within a polytope \cite{Wolfe1976}. The initial approach \cite{wolfe1959simplex} leverages the fundamental insight that the Karush–Kuhn–Tucker (KKT) conditions of a quadratic optimization problem translate into a linear programming (LP) formulation, as we discussed above. To understand the solutions of the linear complementarity problem, note that $w - Mz = q_{\lambda}$ can be written as: 
\begin{align}
Ax = q_{\lambda}, \text{ for } A = [I, -M], \label{eq:clp}
\end{align}
Then, the basic feasible solutions of the system are simply those where $x_{i}. x_{i+n} = 0$ (i.e., complementarity of $w, z$), and for a corresponding ``complementary basis''\footnote{For any basic variable $x_i$, the corresponding complement $x_{i+n}$ should be non-basic.} $B$, 
\begin{align}
x_B = A_B^{-1}q_{\lambda}. \label{eq:basis}
\end{align}

Once the optimal basis is known \eqref{eq:basis}, the optimal solution, where the projected gradient in the affine space vanishes, is readily found. However, if the optimal basis $B$ remains unidentified, a complementarity pivot becomes necessary. This pivot moves the algorithm to a new LP basis while carefully preserving complementarity conditions, a critical aspect of the method.  Therefore, any given $\lambda$, the \emph{complementary pivot algorithm} {(similar to the Simplex method)} can solve the \eqref{lc} for positive semidefinite $M$ (as in our case) \eqref{lc} \cite{murty1988linear}. The algorithm is known to take exponential steps on special instances (similar in vein to Klee-Minty Cubes for the simplex method) \cite{murty1978computational}.  

\begin{figure}[t]
\vspace{-12 pt}
    \centering
    \includegraphics[scale = 0.14]{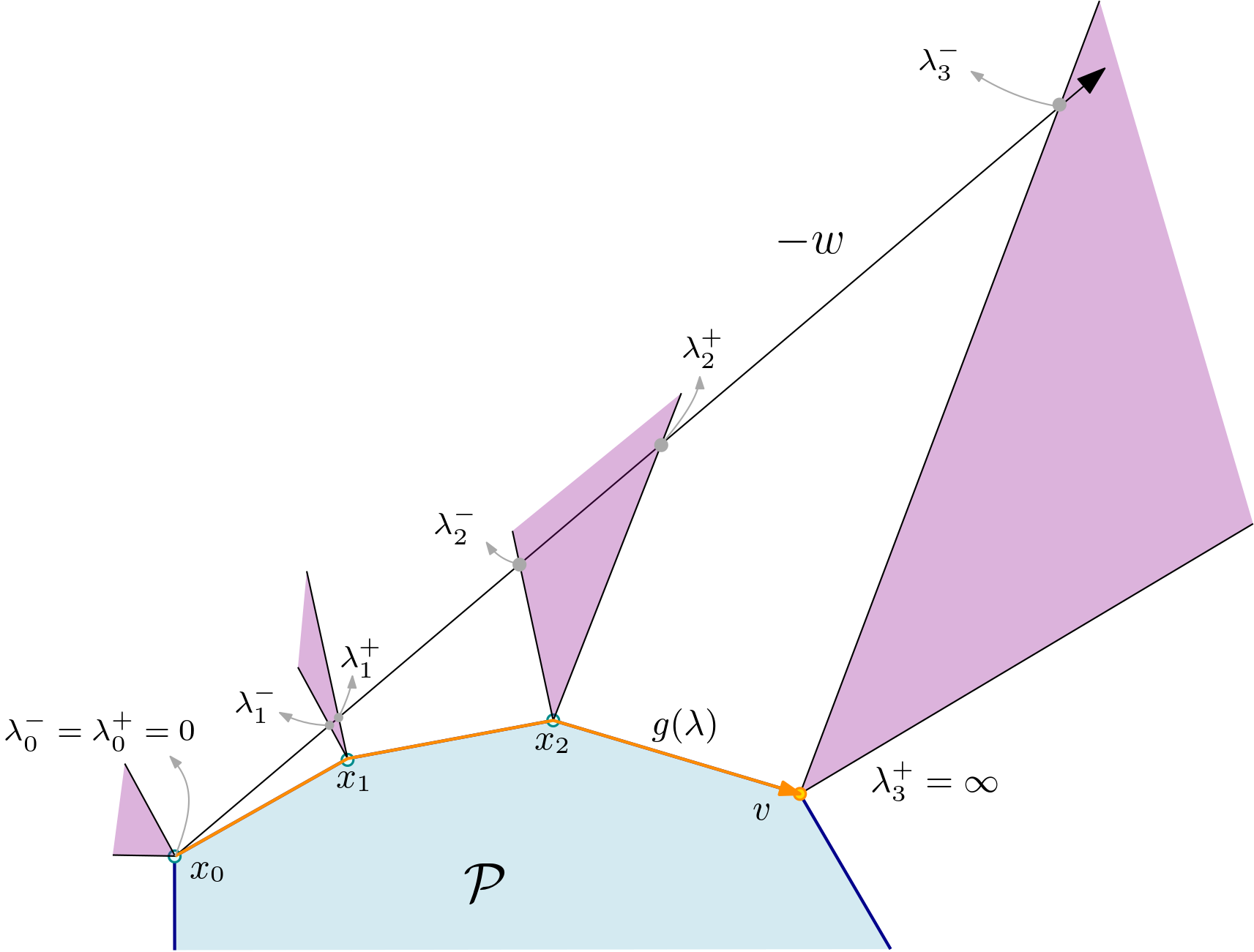}
    \caption{ {Figure showing the structure of the projections curve $g(\lambda) = \Pi_{\mathcal{P}}(x_0 - \lambda {w})$ for $\lambda \geq 0$, which is depicted by the orange line. Breakpoints in the curve correspond to $x_i$ with $g(\lambda_i^-) = g(\lambda_i^+) = x_i$, and $\lambda_3^+ = \infty$ since $\lim_{\lambda \to \infty} g(\lambda) = {v} = \argmin_{y\in \mathcal{P}} \innerprod{y}{{w}}$. Consider the first linear segment from $x_0$ to $x_1$. We have $N_{\mathcal{P}}(g(\lambda)) = N_{\mathcal{P}}(g(\lambda^\prime)) \subset N_{\mathcal{P}}(x_0)$ for all $\lambda, \lambda^\prime \in (0,\lambda_1^{-})$. Then, another constraint becomes tight at the end point of the first segment $x_1$, and thus we have $N_{\mathcal{P}}(g(\lambda)) \subset N_{\mathcal{P}}(x_1)$ for all $\lambda \in (0,\lambda_1^{-})$. This process of dropping and adding constraints  continues until we reach the endpoint ${v}$.}}
    \label{fig:Normal Cone}
\end{figure}

The LCP view of the parametric projections curve immediately implies its piecewise linearity, {which in turn allows us to prove a bound on the number of pieces of the curve of $O(2^m)$} (Theorem \ref{breakpoints}).} 

\begin{remark}\label{rem:piecewise} {We note that the linear complementarity problem and its form \eqref{eq:clp} already shows us that if the basis $B$ stays the same, then the projections curve is linear as $q_{\lambda}$ is perturbed. Therefore, the projections curve is piecewise linear. In fact, it is easy to show that if the normal cone of two points on the projections curve is the same, i.e., $N_{\mathcal{P}}(g(\lambda)) = N_{\mathcal{P}}(g(\lambda^\prime))$ for some $\lambda < \lambda^\prime$, then the curve between $g(\lambda)$ and $g(\lambda^\prime)$ is linear, i.e., $g(\delta\lambda + (1-\delta)\lambda^\prime) = \delta g(\lambda) + (1-\delta)g(\lambda^\prime)$, for $\delta\in[0,1]$. {This claim readily follows from the first-order optimality of Euclidean projections.}}
\end{remark}

Though it is an open question to bound the number of breakpoints for any general polytope, we show a simple bound next which depends on the number of faces of the polytope:
\begin{restatable}[Bound on breakpoints in parametric projections curve]{theorem}{breakpoints}
\label{breakpoints}
Let $\mathcal{P} \subseteq \mathbb{R}^n$ be a polytope with $m$ facet inequalities (e.g., as in \eqref{polytope}). Then, the total number of breakpoints of the projections curve is $O(2^m)$.
\end{restatable}
 \noindent \textit{Proof.} Note that due to the equivalence of normal cones at two projections implies the projections curve is linear between the two points (which must necessarily lie on the same face). Further, the structure of the optimal solutions of the parametric LCP problem \eqref{eq:basis} shows that the number of breakpoints is at most the number of bases of the problem, which is $O(2^m)$. Equivalently, the structure of the optimal solutions of the parametric LCP \eqref{eq:basis} shows that the number of breakpoints is at most the number of complementary bases of the problem, which is again $O(2^m)$.
\hfill $\square$

Murty \cite{murty1978computational} has constructed an example where the LCP algorithm has to go through all $2^n$ BFS to finally reach the optimal basis, which in our case would correspond to searching through all BFS (potentially) to find the next feasible direction. In contrast to this approach, we will characterize the precise the directions along which this curve moves, by using the shadow directions, rather than searching over complementary bases.

\subsubsection{Directions of the Piecewise Linear Segments}
{Consider the projections curve $g(\lambda)$ as $\lambda$ varies in $[0, \infty)$. We know that this is piecewise linear, and it starts at $x_0$, and ends at the point $x_k = \arg \min_{v\in F} \|x_0-v\|^2$, where $F = \arg \min_{v\in \mathcal{P}} \innerprod{\nabla h}{v}$, and $k \leq 2^m$ (Theorem \ref{FWlimit}). Let the remaining breakpoints of the projections curve be $x_1, \hdots, x_{k-1}$. Consider the minimum and maximum $\lambda$ values which correspond to any breakpoint $x_i$, and refer to them as $\lambda_i^{-}$ and $\lambda_i^+$ (see Figure \ref{fig:Normal Cone}). Consider two points $y,z$ on different piecewise linear segments of the curve $\lambda$. Then, the normal cones $N_{\mathcal{P}}(y) \neq N_{\mathcal{P}}(z)$, since otherwise would be linear between $y$ and $z$ (see Remark \ref{rem:piecewise}).}

{In addition, note that $N_{\mathcal{P}}(g(\lambda_i^- -\epsilon)) \neq N_{\mathcal{P}}(g(\lambda_i^+ + \epsilon))$ for $\epsilon>0$ sufficiently small. In fact, for $\epsilon >0$ sufficiently small, the direction $d \coloneqq \frac{g(\lambda_{i}^+ + \epsilon) - g(\lambda_{i}^+)}{\epsilon} \in T_{\mathcal{P}}(x_i)$ is a feasible direction from $x_i$. Hence, $A_{I(x_i)}d \leq 0$ for tight constraints $I(x_i)$. Starting with $x_i$ movement along $d$ can therefore only drop some tight constraints. Therefore,  $N_{\mathcal{P}}(g(\lambda_{i}^+ + \epsilon)) \subseteq N_{\mathcal{P}}(
x_{i})$. Similarly, consider the direction $d \coloneqq g(\lambda_{i}^-) - g(\lambda_{i-1}^+)$, i.e., the linear segment between breakpoints $x_{i-1}$ and $x_i$. Consider the point $y = (x_{i} + x_{i-1})/2$, and note that the normal cone is constant in a small ball around $y$. Therefore, the direction $d$ satisfies $A_{I(y)}d =0$. This means that $N_{\mathcal{P}}(g(\lambda_i^- - \epsilon)) \subseteq N_{\mathcal{P}}(\lambda_i^-)$ for small enough $\epsilon>0$, as some constraints may become tight at the breakpoint $x_i$. The latter happens when the movement along direction $d$ is maximal in the polytope.}

{Our key result in this section is that the piecewise segments of the projections curve can be characterized, using properties of the shadow of the gradient. In particular, we formally show that: 
\begin{itemize}
\item[(i)] At a breakpoint $x_i$, if the shadow is orthogonal to normal of the projection at $x_i$, then the next linear segment of the projections curve is given by a maximal movement along the \emph{shadow} direction;  
\item[(ii)] Otherwise, the next linear segment is obtained by moving along the \emph{``in-face'' shadow} direction until the normal cone changes.  
\end{itemize}
Note that the only difference between the shadow and the in-face shadow is that the former projects the negative gradient onto the tangent cone at $x$ (i.e., on $A_{I(x)}d \leq 0$), whereas the latter projects the negative gradient onto the minimal face of $x$ (i.e., on $A_{I(x)}d = 0$). We formalize this next.} 

\begin{figure}[t]
    \centering
        \vspace{-10 pt}
    \includegraphics[scale = 0.65]{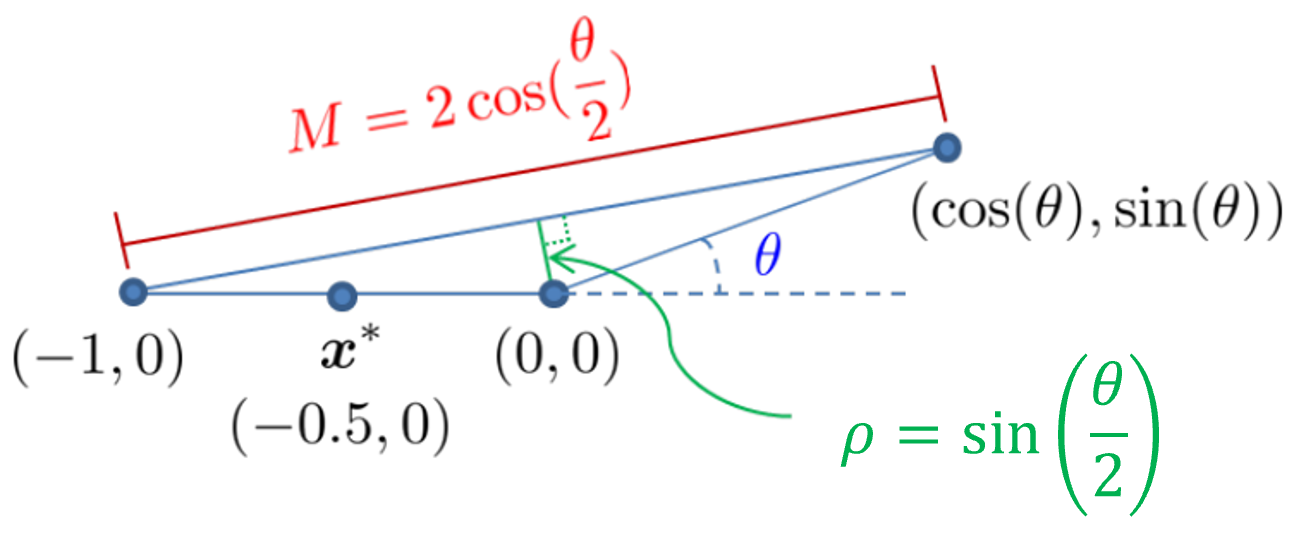}
    \hspace{10 pt}
    \includegraphics[scale = 0.3]{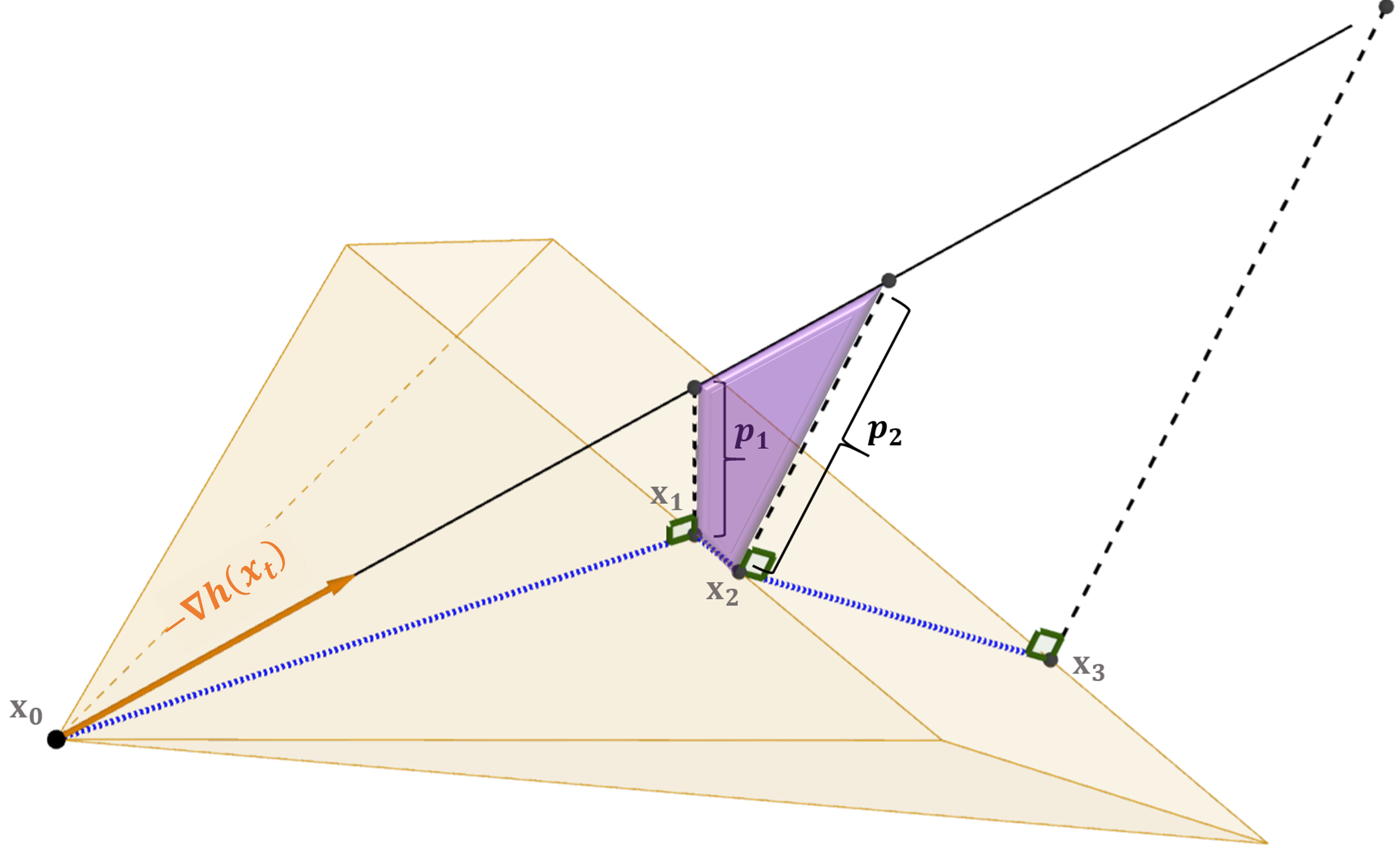}
    \caption{{Left:} Figure from Lacoste-Julien and Jaggi's work \cite{Lacoste2015} showing the pyramidal-width $\delta$ of a simple triangle domain as a function of the angle $\theta$. {Right:} An example of the projections curve, showing how the structure of the normal vectors ${p}_{i} \coloneqq x_0- \lambda_{i}\nabla h(x_0) - x_{i}$ determine whether we take an in-face step or shadow step\protect\footnotemark.}
    \label{fig:pyramidal width}
\end{figure}
\footnotetext{This example is inspired by Damiano Zeffiro.}

\begin{restatable}[Directions in the Projections Curve]{theorem}{invariance}
\label{invariance}%
Let $\mathcal{P} \subseteq \mathbb{R}^n$ be a polytope. Let $x_{i-1} \in \mathcal{P}$ be the $i$th breakpoint in the projections curve $g(\lambda) = \Pi_{\mathcal{P}}(x_0 -\lambda \nabla h(x_0))$, with $x_{i-1}=x_0$ {for $i=1$}. Let $\lambda_{i-1}^+$ be the maximum $\lambda$ for which $g(\lambda) = x_{i-1}$, and {$p_{i-1} \coloneqq x_0-\lambda_{i-1}^+\nabla h(x_0)-x_{i-1}$}. Then, if $d_{x_{i-1}}^\Pi(\nabla h(x_0)) = {0}$ (or equivalently, if $-\nabla h(x_0) \in N_{\mathcal{P}}(x_{i-1})$), then $\lim_{\lambda \to \infty}g(\lambda) = x_{i-1}$ is the end point of $g(\lambda)$. Otherwise, $d_{x_{i-1}}^\Pi(\nabla h(x_0)) \neq {0}$, in which case the following holds: 
\begin{enumerate} 
 \vspace{-1 pt}
\setlength{\itemsep}{-1 pt}
 \item [(a)] {\bf Shadow steps:} If $\innerprod{{p_{i-1}}}{d_{x_{i-1}}^\Pi (\nabla h(x_0))} = 0$, then the next breakpoint of the projections curve can be obtained by moving maximally in the shadow direction, i.e.,  
 $$x_{i} \coloneqq {x}_{i-1} + \theta^* {d}_{{x}_{i-1}}^\Pi(\nabla h(x_0)),$$ 
 where $\theta^* = \max \{\delta : {x}_{i-1} + \delta {d}_{{x}_{i-1}}^\Pi(\nabla h(x_0)) \in \mathcal{P}\}$. Moreover, $\lambda_i^- = \lambda_{i-1}^+ + \theta^*$.

\item [(b)] {\bf In-face steps:} Otherwise, $\innerprod{p_{i-1}}{d_{x_{i-1}}^\Pi(\nabla h(x_0))} \neq 0$. Define the face exposed by $p_{i-1}$ as $F_{p_{i-1}} \coloneqq \{z\in \mathcal{P} \mid \innerprod{p_{i-1}}{z-x_{i-1}}=0\}$, and assume that $F_{p_{i-1}}$ is also the minimal face containing $x_{i-1}$ (this holds under primal non-degeneracy; see Remark~3 for the general case). Then, the next breakpoint is obtained by moving in the in-face direction $\hat{d}_{x_{i-1}}^{\Pi}$ until the corresponding projections lie in the normal cone $N_{\mathcal{P}}(x_{i-1})$, i.e.,
$$ x_i \coloneqq {x}_{i-1} + \theta^* \hat{d}_{{x}_{i-1}}^\Pi(\nabla h(x_0)),$$
where $\theta^* = \max \{\theta \mid  N_{\mathcal{P}}(g(\lambda_{i-1}^+ + \theta)) \supseteq N_{\mathcal{P}}(x_{i-1})\}$. Strict containment only holds when the movement is maximal along the direction $\hat{d}_{x_{i-1}}^{\Pi}$ (which can be computed using line search).  In general, when $\mathcal{P} = \{x \mid Ax \leq b\}$, the in-face movement can be solved using the following program: 
\begin{align}
      \tag{In-Face Movement}  &\max~ \theta \label{in-face-lp}\\
    \textrm{subject to } & (x_0 - (\lambda_{i-1}^+ + \theta) \nabla h (x_0)) - (x_{i-1} + \theta \hat{d}_{x}^\Pi) = A_{I(x_{i-1})}^\top \mu,\\
    & A(x_{i-1} + \theta \hat{d}_{x_{i-1}}^\Pi) \leq b, \\
    &\mu \geq 0.
\end{align}
\end{enumerate}
\end{restatable}

We refer the reader to Figure \ref{fig:pyramidal width}(right) for an illustration of the theorem. At the first breakpoint $x_1$, in Figure \ref{fig:pyramidal width}(right) we have $\innerprod{{p}_{1}}{{d}_{{x}_{1}}^\Pi(\nabla h(x_0))} \neq 0$ in which case we take an in-face step, whereas at $x_2$ we have $\innerprod{{p}_{2}}{{d}_{{x}_{2}}^\Pi(\nabla h(x_0))} = 0$, in which case we take a shadow step. We will now prove Theorem \ref{invariance}:\\

\noindent
\textit{Proof of Theorem \ref{invariance}.}
We first characterize the end point of the projections curve in Case 1, followed by Case 2 on the directions of movement specified in (a) and (b) in the theorem. \\
{\bf Case 1}. We first claim that if at breakpoint $x_{i-1}$, the shadow $d_{x_{i-1}}^\Pi(\nabla h(x_0)) = {0}$, then $\lim_{\lambda \to \infty}g(\lambda) = x_{i-1}$ is the end point of the projections curve $g(\lambda)$.

Since $d_{x_{i-1}}^\Pi(\nabla h(x_0)) ={0}$, using Moreau's decomposition theorem (see Theorem \ref{moreau}) we have $-\nabla h(x_{0}) \in N_{\mathcal{P}}(x_{i-1})$. Note that the first-order optimality condition of $ g(\lambda_{i-1}^+) = x_{i-1}$ is $\innerprod{{p}}{{z} - x_{i-1}} \leq 0 $ for all ${z} \in \mathcal{P}$. Since $-\nabla h(x_{0}) \in N_{\mathcal{P}}(x_{i-1})$ (i.e., $\innerprod{{-}\nabla h(x_0)}{{z} - x_{i-1}} \leq 0 \; \forall {z} \in \mathcal{P}$) and $\lambda \geq \lambda_{i-1}^+$, we get $\innerprod{x_0-\lambda \nabla h(x_0) - x_{i-1}}{{z} - x_{i-1}}  \leq 0 \; \forall {z} \in \mathcal{P}$. Thus,
$x_{i-1}$ satisfies the first-order optimality condition for $g(\lambda)$ when $\lambda \geq \lambda_{i-1}^{+}$.

\noindent 
{\bf Case 2:} Using piecewise linearity of the projection curve, we know that $g(\lambda_{i-1}^+ + \epsilon) = x_{i-1} + \epsilon d$ for some $d \in \mathbb{R}^n$, and $\epsilon > 0$ sufficiently small. %For brevity, let the shadow at $x_{i-1}$ be $d_{x_{i-1}}^\Pi := d_{x_{i-1}}^\Pi (\nabla h(x_0))$, and in-face shadow be $\hat{d}_{x_{i-1}}^\Pi := \hat{d}_{x_{i-1}}^\Pi (\nabla h(x_0))$. 
Define the normals of the projections ${p}_{i-1} := x_0- \lambda_{i-1}^+ \nabla h(x_0) - x_{i-1}$ and ${p}_{\epsilon} := x_0- (\lambda_{i-1}^+ + \epsilon) \nabla h(x_0) - g(\lambda_{i-1} + \epsilon)$ at $x_{i-1}$ and $g(\lambda_{i-1} + \epsilon)$ respectively. 

We now proceed with two {sub}cases:
\vspace{0.2cm}
\begin{enumerate}
    \item[{\bf Case 2(a).}] \textit{Suppose that $\innerprod{{p}_{i-1}}{d_{x_{i-1}}^\Pi} = 0$.} In this case, we will prove that $d = d_{x_{i-1}}^\Pi$, by showing that $x_{i-1} + \epsilon d_{x_{i-1}}^\Pi$ satisfies first-order optimality for $g(\lambda_{i-1} + \epsilon)$. Indeed, for any ${z} \in \mathcal{P}$, we have
    \vspace{-5 pt}
    \begin{align*}
        &\innerprod{x_0- (\lambda_{i-1} + \epsilon) \nabla h(x_0) - x_{i-1} - \epsilon d_{x_{i-1}}^\Pi}{{z} - x_{i-1} - \epsilon d_{x_{i-1}}^\Pi}\nonumber\\
        &=\innerprod{{p}_{i-1} + \epsilon(-\nabla h(x_0) - d_{x_{i-1}}^\Pi)}{{z} - x_{i-1} - \epsilon d_{x_{i-1}}^\Pi}\nonumber\\
        &= -\epsilon\underbrace{\innerprod{{p}_{i-1}}{d_{x_{i-1}}^\Pi}}_{\overset{(i)}{=} 0}
    - \epsilon^2\underbrace{\innerprod{-\nabla h(x_0) - d_{x_{i-1}}^\Pi}{d_{x_{i-1}}^\Pi}}_{\overset{(ii)}{=}0}\\
    &\hspace{5 pt}+\underbrace{\innerprod{{p}_{i-1}}{{z} - x_{i-1}}}_{\overset{(iii)}{\leq} 0} +\epsilon \underbrace{\innerprod{-\nabla h(x_0) - d_{x_{i-1}}^\Pi}{{z} - x_{i-1}}}_{\overset{(iv)}{\leq} 0} \leq 0,
    \end{align*}
where we used our assumption $\innerprod{{p}_{i-1}}{d_{x_{i-1}}^\Pi} = 0$ in $(i)$, the definition of the shadow $d_{x_{i-1}}^\Pi := \Pi_{T_{\mathcal{P}}(x_{i-1})}(-\nabla h(x_0))$ in $(ii)$, the first-order optimality at $x_{i-1}$ in $(iii)$, and the fact that $-\nabla h(x_0) - d_{x_{i-1}}^\Pi \in N_{\mathcal{P}}(x_{i-1})$ (by Moreau's decomposition theorem) in $(iv)$. 

Since the above argument holds for any $\epsilon$ such that $x_{i-1} + \epsilon d_{x_{i-1}}^\Pi \in \mathcal{P}$, {it follows that the next breakpoint $x_{i} \coloneqq {x}_{i-1} + \theta^* {d}_{{x}_{i-1}}^\Pi(\nabla h(x_0)),$ 
 where $\theta^* = \max \{\delta : {x}_{i-1} + \delta {d}_{{x}_{i-1}}^\Pi(\nabla h(x_0)) \in \mathcal{P}\}$. At this point, there is a change in the normal cone due to addition of the tight constraint restricting further movement. Therefore, $\lambda_i^- = \lambda_{i-1}^+ + \theta^*$.  This proves case {2(a)}.}

\vspace{0.2cm}
\item[{\bf Case 2(b).}]\textit{Now suppose that $\innerprod{{p}_{i-1}}{d_{x_{i-1}}^\Pi} \neq 0$.} We now show that in this case $d = \hat{d}_{x_{i-1}}^\Pi$.

{Since $g(\lambda_{i-1}^+ + \epsilon) = x_{i-1} + \epsilon d$ {for some $d$ and small enough $\epsilon>0$}, from first-order optimality we have: $\innerprod{p_{\epsilon}}{z - x_{i-1}-\epsilon d} \leq 0$ {for all $z \in \mathcal{P}$}. Setting $z = x_{i-1}$, we get $\innerprod{p_{\epsilon}}{d} \geq 0$. Since $p_{\epsilon} \in N_{\mathcal{P}}(x_{\epsilon})$ and $d \in T_{\mathcal{P}}(x_{\epsilon})$, we also get $\innerprod{p_{\epsilon}}{d} \leq 0$, which together imply $\innerprod{p_{\epsilon}}{d}$ = 0.}

We next claim that this holds in the limit as $\epsilon \rightarrow 0$. Suppose $\innerprod{p_{i-1}}{d} = \delta \neq 0$. Using the definition of $p_{\epsilon}$, we can write $p_{i-1} = p_{\epsilon} - \epsilon(-\nabla h(x_0) - d)$. Taking inner products with $d$ gives $\innerprod{ p_{i-1}}{d} = \innerprod{p_\epsilon}{d}
  + \epsilon\innerprod{ \nabla h(x_0) + d}{d}
= \epsilon\innerprod{\nabla h(x_0) + d}{d}$,
where we used $\innerprod{ p_\epsilon}{d} = 0$ in the last equality. Since we assumed that $\innerprod{ p_{i-1}}{d} = \delta$, we have $\delta = \epsilon \langle \nabla h(x_0) + d, d\rangle$. Letting $\epsilon \to 0$, the right-hand side converges to $0$, while the left-hand side remains the fixed nonzero constant $\delta$, a contradiction. To summarize, we have:
\begin{align}
\innerprod{p_{\epsilon}}{d} = 0, \text{ and }\quad {\innerprod{p_{i-1}}{d} = 0.} \tag{***}
\end{align}

From first-order optimality at $g(\lambda_{i-1}^+ + \epsilon)$ we have:  
    \begin{align} 
        &\hspace{-5 pt} \innerprod{{x_0- (\lambda_{{i-1}}^+ + \epsilon) \nabla h(x_0) - x_{i-1} - \epsilon d}}{{z} - x_{i-1} - \epsilon d} \leq 0 \quad \forall\; {z} \in \mathcal{P}, \\
    &\Rightarrow -\epsilon\underbrace{\innerprod{{p}_{i-1}}{d}}_{\overset{(i)}{=} 0}-\epsilon^2 \underbrace{\innerprod{-\nabla h(x_0) - d}{d}}_{\overset{(ii)}{=}0}
    + \innerprod{{p}_{i-1}}{{z} - x_{i-1}} +\epsilon \innerprod{-\nabla h(x_0) - d}{{z} - x_{i-1}} \leq 0 \quad \forall\; {z} \in \mathcal{P} \nonumber \\  
    &\Rightarrow \innerprod{{p}_{i-1}}{{z} - x_{i-1}} +\epsilon \innerprod{-\nabla h(x_0) - d}{{z} - x_{i-1}} \leq 0 \quad \forall\; {z} \in \mathcal{P}, \label{eq6 thm 2}
    \end{align}
where we used (***) in $(i)$ and $(ii)$. Whenever $\innerprod{p_{i-1}}{z - x_{i-1}} = 0$ for $z \in \mathcal{P}$, we get $\innerprod{-\nabla h(x_0) - d}{{z} - x_{i-1}} \leq 0$. This happens for all $z$ in the face $F_{p_{i-1}} = \{z \in \mathcal{P} \mid \innerprod{p_{i-1}}{z-x_{i-1}} = 0\}$. {In other words, we have $\innerprod{-\nabla h(x_0)-d}{y}\leq 0$ for all $y\in \mathrm{Cone}(F_{p_{i-1}}-x_{i-1})$. Together with $\innerprod{-\nabla h(x_0)-d}{d}=0$, this gives $\innerprod{-\nabla h(x_0)-d}{y-d}\leq 0$ for all $y\in \mathrm{Cone}(F_{p_{i-1}}-x_{i-1})$, which is the first-order optimality condition for projecting $-\nabla h(x_0)$ onto $\mathrm{Cone}(F_{p_{i-1}}-x_{i-1})$. Since we assumed $F_{p_{i-1}}$ to be the minimal face $F$ containing $x_{i-1}$, it follows that $d=\hat d_{x_{i-1}}^\Pi$.}

Now that we have shown $d=\hat{d}_{x_{i-1}}^\Pi(\nabla h(x_0))$, it remains to characterize the next breakpoint in Case 2(b). By piecewise linearity of $g(\cdot)$, there exists $\theta^*>0$ such that
$$ g(\lambda) = x_{i-1}+(\lambda-\lambda_{i-1}^+)\hat d_{x_{i-1}}^\Pi(\nabla h(x_0)) \qquad
\forall \lambda\in[\lambda_{i-1}^+,\lambda_i^-],$$
where $\lambda_i^- \coloneqq \lambda_{i-1}^+ + \theta^*$, and the next breakpoint $x_i \coloneqq x_{i-1}+\theta^*\hat d_{x_{i-1}}^\Pi(\nabla h(x_0)) = g(\lambda_i^-)$. We first claim that $\theta^* = \max \left \{\theta \geq 0:\, N_{\mathcal P}(g(\lambda_{i-1}^+ + \theta)) \supseteq N_{\mathcal P}(x_{i-1}) \right\}.$ Indeed, for every $0 \leq \theta < \theta^*$, the point
$$x(\theta):=x_{i-1}+\theta \hat d_{x_{i-1}}^\Pi(\nabla h(x_0))$$
lies in the relative interior of the same minimal face as $x_{i-1}$, since $x(\theta)$ moves from $x_{i-1}$ along the in-face direction $\hat d_{x_{i-1}}^\Pi(\nabla h(x_0))$ non-maximally. Hence $I(x(\theta))=I(x_{i-1})$, which implies $N_{\mathcal P}(x(\theta))=N_{\mathcal P}(x_{i-1}).$
At $\theta=\theta^*$, either the movement is not maximal along the face, in which case $N_{\mathcal P}(x_i)=N_{\mathcal P}(x_{i-1})$,
or the movement is maximal and a new constraint becomes tight at $x_i$, in which case $N_{\mathcal P}(x_i)\supsetneq N_{\mathcal P}(x_{i-1})$. Thus, in both cases, $\theta^*$ is precisely the largest value for which the normal cone of the projection continues to contain $N_{\mathcal P}(x_{i-1})$.

We now finally show that $\theta^*$ is exactly the optimal value of the linear program in the theorem. The first-order optimality condition for $x(\theta)$ being the projection of $x_0-(\lambda_{i-1}^+ + \theta)\nabla h(x_0)$ %onto $\mathcal P$ 
is:
$$ x(\theta) = \Pi_{\mathcal{P}}(x_0-(\lambda_{i-1}^+ + \theta)\nabla h(x_0)) \quad \iff \quad
(x_0-(\lambda_{i-1}^+ + \theta)\nabla h(x_0)) - x(\theta)\in N_{\mathcal P}(x(\theta)).
$$
For every $\theta < \theta^*$, we have already shown that $N_{\mathcal P}(x(\theta))= N_{\mathcal P}(x_{i-1}) = \{A_{I(x_{i-1})}^\top \mu:\mu\geq 0\}$.
Therefore, for all such $\theta$, the projection condition is equivalent to the existence of $\mu \geq 0$ such that
$$(x_0-(\lambda_{i-1}^+ + \theta)\nabla h(x_0)) - (x_{i-1}+\theta \hat d_{x_{i-1}}^\Pi) = A_{I(x_{i-1})}^\top \mu,$$
together with feasibility: $A(x_{i-1}+\theta \hat d_{x_{i-1}}^\Pi)\leq b$. Hence every $\theta<\theta^*$ is feasible for the in-face LP.

We next show that the linear program cannot produce any feasible $\tilde\theta > \theta^*$. Suppose otherwise. Then there exists $\mu\geq 0$ such that
$$(x_0-(\lambda_{i-1}^+ + \tilde\theta)\nabla h(x_0)) - (x_{i-1}+\tilde\theta \hat d_{x_{i-1}}^\Pi)
= A_{I(x_{i-1})}^\top \mu \in N_{\mathcal P}(x_{i-1}),$$
along with $A(x_{i-1}+\tilde\theta \hat d_{x_{i-1}}^\Pi)\leq b$.
Thus $x_{i-1}+\tilde\theta \hat d_{x_{i-1}}^\Pi\in \mathcal P$ and satisfies the first-order optimality condition for the projection of
$x_0-(\lambda_{i-1}^+ + \tilde\theta)\nabla h(x_0)$ onto $\mathcal P$.
Therefore, $g(\lambda_{i-1}^+ + \tilde\theta) = x_{i-1}+\tilde\theta \hat d_{x_{i-1}}^\Pi$, which contradicts the definition of $\theta^*$ as the last parameter value on the current linear segment.

Finally, we verify feasibility of the in-face LP at $\theta=\theta^*$. If the movement is not maximal, then $N_{\mathcal P}(x_i)=N_{\mathcal P}(x_{i-1})$, so feasibility follows directly from the projection optimality condition. Now assume that the movement is maximal, i.e. $x_i$ is reached by hitting the boundary of the current face. For every $\theta<\theta^*$ sufficiently close to $\theta^*$, we still have $N_{\mathcal P}(x(\theta))=N_{\mathcal P}(x_{i-1})$, 
and hence
$$ (x_0-(\lambda_{i-1}^+ + \theta)\nabla h(x_0))-x(\theta)
\in N_{\mathcal P}(x_{i-1}).$$
Since $N_{\mathcal P}(x_{i-1})$ is a closed polyhedral cone, passing to the limit as $\theta\uparrow\theta^*$ gives $(x_0-\lambda_i^- \nabla h(x_0))-x_i \in N_{\mathcal P}(x_{i-1})$. Together with $x_i\in\mathcal P$, this proves that $\theta^*$ is feasible for the in-face LP. Since no $\tilde\theta>\theta^*$ is feasible, the LP optimum is exactly $\theta^*$, completing the proof for part (b). \hfill $\square$
\end{enumerate}

\noindent
\begin{remark}
In the degenerate case, $F_{p_{i-1}}$ need not be the minimal face containing $x_{i-1}$. The next segment is then obtained by projecting $-\nabla h(x_0)$ onto the feasible directions of the exposed face: $d=\Pi_{\mathrm{Cone}(F_{p_{i-1}}-x_{i-1})}\left(-\nabla h(x_0)\right)$. If $p_{i-1}=\sum_{j\in I_+}\alpha_j a_j$ with $\alpha_j>0$, where $I_+\subseteq I(x_{i-1})$ indexes the normals generating the minimal face of $N_{\mathcal P}(x_{i-1})$ containing $p_{i-1}$, then $\mathrm{Cone}(F_{p_{i-1}}-x_{i-1}) =
\{d \mid A_{I_+}d=0,\;A_{I(x_{i-1})\setminus I_+}d\leq0\}$. Accordingly, the in-face program is modified by replacing $\hat d^\Pi_{x_{i-1}}$ with $d$ and $A_{I(x_{i-1})}^\top\mu$ with $A_{I_+}^\top\mu$. When $F_{p_{i-1}}$ is the minimal face containing $x_{i-1}$, we have $I_+=I(x_{i-1})$ and recover the in-face direction and program in part~(b).
\end{remark}

%\hm{I completed the proof above addressing the comments below. The new proof addition could definitely be shortened, but I kept it verbose and detailed so that its easy to understand. The proof idea is exactly that from our discussion at dancing goats. It essentially boils down to the fact that since we know we are moving in-face $I(x_i) = I(g(\lambda_{i-1}^+ + \epsilon))$ and that FOO holds with if and only if.\\
%The text below was the old text and NEEDS TO BE REMOVED} \\

%{\color{red} I'm not sure if the characterization of the next breakpoint holds.}

%Moreover, the next breakpoint \hme{$x_i := g(\theta^*) ={x}_{i-1} + (\theta^* - \lambda_{i-1})\hat{d}_{{x}_{i-1}}^\Pi$ by the definition of $\theta^*$} and the fact that the projections curve leaves the minimal face after this point, and thus direction change in the projections curve must happen {(by Remark \ref{rem:piecewise})}. This proves case {2(b)}. 

%\textcolor{red}{Add details about why we can use the LP.}

\section{Interpolations between PGD and FW Algorithms}\label{sec:interpolations}
PGD iterates are computed by projections of steps in the negative gradient directions, and {\sc Shadow-walk} walks along feasible descent directions of shadow and in-face shadow steps (to mimic PGD steps), whereas the FW algorithm walks along greedy projection steps or FW steps. Away-steps approximate shadow steps and pairwise-FW steps combine the latter two. Through the design of two algorithms, Shadow-CG and PGD-CG, we next discuss mathematical conditions of when PGD, Shadow and FW steps should be prioritized. We further show provable convergence of these algorithms.

\subsection{Algorithm {\sc Shadow-CG}}\label{sec:shadowcg}
{FW steps can help achieve coarser progress by skipping tracing a large number of breakpoints of the projections curve, and the shadow steps can help reduce zig-zagging towards the end of the algorithm.} To obtain a computable proxy for when to switch between these two steps, we note that Frank-Wolfe directions become close to orthogonal to the negative gradient towards the end of the algorithm. However, in this case the norm of the shadow also starts diminishing (Lemma \ref{primalgapestimate}). Therefore, we propose to choose the FW direction whenever $$\innerprod{-\nabla h(\iter{x}{t})}{ {d}_t^{\text{FW}}} \geq \innerprod{-\nabla h(\iter{x}{t})}{{{d}_{\iter{x}{t}}^\Pi}/{\|{d}_{\iter{x}{t}}^\Pi\|}} = \|{d}_{\iter{x}{t}}^\Pi\|, $$ and the shadow directions otherwise. This is sufficient to give provable linear convergence. We summarize the {\sc Shadow-CG} algorithm in Algorithm \ref{alg:shadow-cg}. 

\begin{algorithm}[t] 
\caption{Shadow Conditional Gradient ({\sc Shadow-CG})}
\label{alg:shadow-cg}
\begin{algorithmic}[1]
\INPUT Polytope $\mathcal{P} \subseteq \mathbb{R}^n$, function $h : \mathcal{P} \to \mathbb{R}$, $x_0 \in \mathcal{P}$ and accuracy parameter $\varepsilon$.
\For{$t = 0, \dots. T$}
        \State Let $\iter{v}{t} \coloneqq \argmin_{{v} \in P} \innerprod{\nabla h(\iter{x}{t})}{{v}}$ and ${d}_t^{\text{FW}} \coloneqq \iter{v}{t} - \iter{x}{t}$. \Comment{\texttt{FW direction} }
    \State {\bf if} $\innerprod{-\nabla h(\iter{x}{t})}{ {d}_t^{\text{FW}}} \leq \varepsilon$ {\bf then } \Return $\iter{x}{t}$
     \Comment{\texttt{ primal gap is small enough}}
    \State Compute the derivative of projection of the gradient ${d}_{\iter{x}{t}}^\Pi$
    \State {\bf if} $ \innerprod{-\nabla h(\iter{x}{t})}{{{d}_{\iter{x}{t}}^\Pi}/{\|{d}_{\iter{x}{t}}^\Pi\|}} \leq \innerprod{-\nabla h(\iter{x}{t})}{ {d}_t^{\text{FW}}}$
    \State \hspace{0.6cm} $\iter{x}{t+1} \coloneqq \iter{x}{t} + \gamma_t {d}_t^{\text{FW}}$ ($\gamma_t \in [0,1]$). \Comment{\texttt{use line-search}}
    \State \textbf{else} 
    \State \hspace{0.6cm} {$\iter{x}{t+1} \coloneqq$ {\textsc{Trace}}$(\iter{x}{t}, \nabla h(\iter{x}{t}), {1/L})$ . \Comment{ \texttt{trace projections curve using Theorem \ref{invariance} until total accrued step-size $\theta^*$ is at least $1/L$.}}}
\EndFor
\RETURN $\iter{x}{t+1}$
\end{algorithmic}
\end{algorithm}

\vspace{0.2cm}
\begin{restatable}{theorem}{algconver2}\label{thm:convergence-shadow-cg}
 Let $\mathcal{P} \subseteq \mathbb{R}^n$ be a polytope with diameter $D$ and suppose that $h : \mathcal{P} \rightarrow \mathbb{R}$ is $L$-smooth and $\mu$-strongly convex over $\mathcal{P}$. Then, {the primal gap of {\sc Shadow-CG} decreases} geometrically: $$h(\iter{x}{t+1}) {- h(x^*)} \leq \left(1- \min \left\{\frac{1}{2},\frac{\mu}{ LD^2}, \frac{\mu}{L}\right\} \right){(h(\iter{x}{t})-h(x^*))},$$
with each iteration of the {\sc Shadow-CG} algorithm. Moreover, the running time to compute an $\epsilon$-approximate solution is $O\left((D^2{(L_O + S)} + \beta S) \frac{L}{\mu} \log (\frac{1}{\epsilon}) \right)$, where $\beta$ is the number of breakpoints of the parametric projections curve, {$S$ is the running time complexity of computing the shadow}, and $L_O$ is the time for linear optimization.
\end{restatable}

\noindent \textit{Proof.} In the algorithm, we either enter the {\textsc{Trace}} procedure (in which case we have $\iter{x}{t+1} \coloneqq {\textsc{Trace}}(\iter{x}{t}, \nabla h(\iter{x}{t}), 1/L)$, or we take a Frank-Wolfe step ($\iter{x}{t+1} \coloneqq \iter{x}{t} + \gamma_t (\iter{v}{t} - \iter{x}{t})$ for some $\gamma_t \in [0,1]$). We split the
proof of convergence into two cases depending on which case happens:
\begin{enumerate} [leftmargin = 10pt]
    \item[] \textbf{Case 1:} \textit{We enter the {\textsc{Trace}} procedure.} Since $\iter{x}{t+1} \coloneqq \textsc{Trace}(\iter{x}{t}, \nabla h(\iter{x}{t}), 1/L)$ and $\textsc{Trace}(\iter{x}{t}, \nabla h(\iter{x}{t}), 1/L)$ traces the whole curve of $g(\lambda) =  \Pi_P(\iter{x}{t} - \lambda \nabla h(\iter{x}{t}))$ until we hit the $1/L$ step size with exact line-search at the active segment, it follows that $h(\iter{x}{t+1}) \leq h(g(1/L))$, and we are thus guaranteed to make at least as much progress per iteration as that of projected gradient descent (PGD) step with a fixed-step size of $1/L$. Hence we get the same standard rate $(1 - \frac{\mu} {L})$ of decrease as PGD with fixed step size $1/L$ \cite{Karimi_2016}.\\
    
    \item[] \textbf{Case 2:} \textit{We take a Frank-Wolfe step.}  Let $\gamma_t$ be the step size chosen by line-search for moving along the chosen FW direction $d_t \coloneqq \iter{v}{t} - \iter{x}{t}$, and let $\gamma_t^{\max} = 1$ be the maximum step-size that one can move along $d_t$. Using the smoothness of $f$, we have $h(\iter{x}{t+1}) \leq h(\iter{x}{t}) + \gamma_t \innerprod{\nabla h(\iter{x}{t})}{d_t} + \frac{L\gamma_t^2}{2} \|d_t\|^2$. %(see Appendix \ref{sec: conver prelims} for more details)
    Define $\gamma_{d_t} \coloneqq \frac{\innerprod{-\nabla h(\iter{x}{t})}{d_t}} {L \|{d}_{t}\|^2}$ to be the step-size minimizing the RHS of the previous inequality. It is important to note that $\gamma_{d_t}$ is not the step-size used in the algorithm obtained from line-search. It is used to only lower bound the progress obtained from the line-search step, since, plugging in $\gamma_{d_t}$ in the smoothness inequality yields:
\begin{equation} \label{progress from smoothness1}
    h(\iter{x}{t}) - h(\iter{x}{t+1}) \geq \frac{\innerprod{-\nabla h(\iter{x}{t})}{d_t}^2} {2L \|d_t\|^2}.
\end{equation}
    We now split the proof depending on whether $\gamma_{t} < \gamma_t^{\max}$ or not:\\
    
    \begin{enumerate}
        \item[$(i)$] \textit{First, suppose that $\gamma_{t} < \gamma_t^{\max}$}. We claim that we can use the step size from $\gamma_{d_t}$ to lower bound the progress even if $\gamma_{d_t}$ is not a feasible step size (i.e. when $\gamma_{d_t} > 1$). To see this, note that the optimal solution of the line-search step is in the interior of the interval $[0, \gamma_t^{\max}]$. Define $x_\gamma \coloneqq \iter{x}{t} + \gamma d_t$. Then, because $h(x_\gamma)$ is convex in $\gamma$, we know that $\min_{\gamma \in [0,\gamma_t^{\max}]} h(x_\gamma) = \min_{\gamma \geq 0} h(x_\gamma)$ and thus $\min_{\gamma \in [0,\gamma_t^{\max}]} h(x_\gamma) = h(\iter{x}{t+1}) \leq h(x_\gamma)$ for all $ \gamma \geq 0$. In particular, $h(\iter{x}{t+1}) \leq h(x_{\gamma_{d_t}})$. Hence, we can use \eqref{progress from smoothness1} to bound the progress per iteration as follows:
    \begin{align}
    h(\iter{x}{t}) - h(\iter{x}{t+1}) &\geq \frac{\innerprod {-\nabla h(\iter{x}{t})}{d_t^{\text{FW}}}^2} {2L \|d_t^{\text{FW}}\|^2}  && \text{(using smoothness)} \\
    &\geq \frac{\innerprod{-\nabla h(\iter{x}{t})}{\frac{d_{\iter{x}{t}}^\Pi}{\|d_{\iter{x}{t}}^\Pi\|}}^2} {2L D^2} && \text{(choice of descent)} \\
    &\geq  \frac{\mu} {L D^2} (h(\iter{x}{t}) - h(x^*)),
    \label{last}
\end{align}
where \eqref{last} follows by Lemma \ref{primalgapestimate}.

\item[$(ii)$] \textit{We have a boundary case: $\gamma_{t} = \gamma_t^{\max}$}. We further divide this case into two sub-cases:

 \begin{enumerate}
        \item[(a)] First assume that $\gamma_{d_t} \leq \gamma_t^{\max}$ so that the step size from smoothness is feasible. Then, using the same argument as above we also have a $(1 - \frac{\mu} {LD^2})$-geometric rate of decrease.
        
        \item[(b)] Finally assume that $\gamma_{d_t} > \gamma_t^{\max}$ and $d_t = d_t^{\mathrm{FW}}$. Observe that $\gamma_{d_t} = \frac{\innerprod {-\nabla h(\iter{x}{t})}{d_t^{\mathrm{FW}}}}{L \|d_t^{\mathrm{FW}}\|^2} > \gamma_t^{\max} = 1$ implies that $\innerprod  {-\nabla h(\iter{x}{t})}{d_t^{\mathrm{FW}}} \geq L\|d_t^{\mathrm{FW}}\|_2^2$. Hence, using the fact that $\gamma_t = \gamma_t^{\max} =1$ in the smoothness inequality given previously, we have:
        \begin{align*}
            h(\iter{x}{t}) - h(\iter{x}{t+1}) &\geq  \innerprod  {-\nabla h(\iter{x}{t})}{d_t^{\mathrm{FW}}} - \frac{L}{2} \|d_t^{\mathrm{FW}}\|_2^2
         \geq \frac{h(\iter{x}{t}) -h(x^*)}{2}, %&& \text{(using Wolfe gap \eqref{Wolfe-gap})}
        \end{align*}
        where the last inequality follows using the convexity of $h$ as follows:
        \begin{equation} \label{Wolfe-gap}
            h(\iter{x}{t}) - h(x^*)  \leq \innerprod {-\nabla h(\iter{x}{t})}{x^* - \iter{x}{t}} \leq \max_{{v} \in \mathcal{P}} \innerprod {-\nabla h(\iter{x}{t})}{{v} - \iter{x}{t}}.
        \end{equation}
        Hence, we get a geometric rate of decrease of 1/2.
        \end{enumerate}
        \end{enumerate}
\end{enumerate} 
The iteration complexity stated in the theorem now follows using the above rate of decrease in the primal gap.
\hfill $\square$
\newline

We next highlight a few implications of this result. Note that we remove the dependence of the running time on the pyramidal width. This implies that the running time is now affine-invariant, since $\beta$ does not increase with affine transformations. For example, for a scaled simplex, $\sum_{e \neq e^\prime} x_e + Mx_{e^\prime}\leq 1$, $x_e\geq 0$, $M\gg 1$, as $M$ increases, the pyramidal width goes to zero, making the AFW running time {prohibitively} large, whereas the number of breakpoints remains $O(n)$. Second, since Frank-Wolfe steps skip a lot of projections, the number of projections required is much less than the worst-case in the above bound, as can be seen in our computations.

\begin{remark}
Garber and Meshi \cite{Garber_2016} and Bashiri and Zhang \cite{Bashiri_2017} both compute the best away vertex in the minimal face containing the current iterate, whereas the shadow step recovers the best convex combination of such vertices aligned with the negative gradient. Therefore, these previously mentioned CG methods can {\it both} be viewed as approximations of {\sc Shadow-CG}. Moreover, Garber and Hazan \cite{Garber_2013} emulate a shadow computation by constraining the FW vertex to a ball around the current iterate. Therefore, their algorithm can be interpreted as an approximation of {\sc Shadow-Walk}.
\end{remark}

\subsection{Algorithm {\sc PGD-CG}}\label{sec:pgd-cg}
{With the aim of improving the computational complexity of our algorithms,
we now propose the {\sc PGD-CG} algorithm, whose description is given in Algorithm \ref{alg:PGD-CG}. In \textsc{Shadow-CG}, each iteration computes the shadow direction $d_{x^{(t)}}^{\Pi}$ to decide whether to take a Frank--Wolfe (\textsc{FW}) step or to trace the projections curve using Theorem \ref{invariance} until the total accrued stepsize is at least $1/L$. To reduce the per-iteration cost while preserving linear convergence, we replace this exact test with a cheap surrogate that avoids computing $d_{x^{(t)}}^{\Pi}$. Applying the Cauchy-Schwarz inequality to the left hand side of \eqref{eq:bestlocaldirection}, we have $$\|\nabla h(x^{(t)})\| \;\ge\; \|d_{x^{(t)}}^{\Pi}\|.$$
Hence, we can use ${\sqrt{c}}\,\|\nabla h(x^{(t)})\|$ with a scalar $c\in(0,1)$ as a conservative proxy for $\|d_{x^{(t)}}^{\Pi}\|$.

The {\sc PGD-CG} algorithm is similar to \textsc{Shadow-CG} but uses a different decision rule: take an \textsc{FW} step whenever $\big\langle -\nabla h(x^{(t)}),\, d_{t}^{\mathrm{FW}} \big\rangle \;\ge\; {\sqrt{c}}\,\|\nabla h(x^{(t)})\|$. Otherwise, we compute the PGD direction $d_t \;\coloneqq\; \Pi_{\mathcal{P}}\!\Big(x^{(t)} - \tfrac{1}{L}\,\nabla h(x^{(t)})\Big) - x^{(t)}$, and perform an exact line search to obtain $x^{(t+1)} \;\coloneqq\; \argmin_{\gamma\in[0,1]} h\!\big(x^{(t)}+\gamma d_t\big)$. This modification:
1) avoids computing the shadow direction in iterations that ultimately take a \textsc{FW} step; 2) removes the dependence of the convergence rate on the number of breakpoints $\beta$ by ``wrapping'' around the polytope via a single projection followed by line search. We remark that while $\|d_{x^{(t)}}^{\Pi}\|$ goes to 0 as we approach the optimal solution (the shadow is a true estimate of convergence as stated in Section \ref{sec:descentdirections}), ${\sqrt{c}}\,\|\nabla h(x^{(t)})\|$ does not (if the optimal is constrained). Therefore, towards the end of the algorithm, we will prioritize projection steps to quickly converge to the optimal solution. The parameter $c$ tunes the \textsc{FW} bias: smaller $c$ favors cheaper \textsc{FW} steps (at the cost of more conservative descent), while larger $c$ does the opposite. The next result quantifies this trade-off and shows that linear convergence is retained.

\begin{algorithm}[t] 
\caption{Projected Gradient Descent - Conditional Gradient ({\sc PGD-CG})}
\label{alg:PGD-CG}
\begin{algorithmic}[1]
\INPUT Polytope $\mathcal{P} \subseteq \mathbb{R}^n$, $L$-smooth function $h : \mathcal{P} \to \mathbb{R}$, $x_0 \in \mathcal{P}$, {tuning parameter $c \in (0,1)$}, and accuracy parameter $\varepsilon$.
\Statex \textit{... same as Algorithm \ref{alg:shadow-cg}, except changing lines {\color{blue}{4-8}} as follows...}
\setcounter{ALG@line}{4}
    \State {{\bf if} ${\sqrt{c}}\|\nabla h(\iter{x}{t})\| \leq \innerprod{-\nabla h(\iter{x}{t})}{ {d}_t^{\text{FW}}}$}
    \State \hspace{0.6cm}${d}_t : = {d}_t^{\text{FW}}$
    \State \textbf{else} 
    \State \hspace{0.6cm} $d_t \coloneqq \Pi_{\mathcal{P}}(\iter{x}{t} - \nabla h(\iter{x}{t})/L) - \iter{x}{t}$. \Comment{ \texttt{PGD step}}
    \State $\iter{x}{t+1} \coloneqq \iter{x}{t} + \gamma_t {d}_t$ ($\gamma_t \in [0,1]$). \Comment{\texttt{use line-search}}
    \RETURN $\iter{x}{t+1}$
\end{algorithmic}
\end{algorithm}

\begin{restatable}{theorem}{algconverg4}
\label{thm:convergence-pgd-cg} Let $\mathcal{P} \subseteq \mathbb{R}^n$ be a polytope with diameter $D$ and suppose that $h : \mathcal{P} \rightarrow \mathbb{R}$ is $L$-smooth and $\mu$-strongly convex over $\mathcal{P}$. Then, {the primal gap of {\sc PGD-CG} decreases} geometrically: $$h(\iter{x}{t+1}) {- h(x^*)} \leq \left(1- \min \left\{\frac{1}{2},\frac{c \mu}{ LD^2}, \frac{\mu}{L}\right\} \right){(h(\iter{x}{t})-h(x^*))},$$
with each iteration of the {\sc PGD-CG} algorithm. Moreover, the running time to compute an $\epsilon$-approximate solution is $O\left((\frac{D^2}{c} L_O + \Pi) \frac{L}{\mu} \log (\frac{1}{\epsilon}) \right)$, where $\Pi$ is the running time complexity of computing a projection, and $L_O$ is the time for linear optimization.
\end{restatable}
The proof of this result is exactly the same as that of Theorem \ref{thm:convergence-shadow-cg} after noting that $\|\nabla h(x^{(t)})\| \;\ge\; \|d_{x^{(t)}}^{\Pi}\|$ by the Cauchy-Schwarz inequality. 
}

We now discuss the implication of this result and compare the performance of {\sc PGD-CG} to the performance of {\sc PGD} and {\sc AFW}. Lacoste-Julien and Jaggi \cite{Lacoste2015} proved a contraction of $\left(1 - \frac{\mu}{L} \left( \frac{\rho}{D}\right)^2 \right)$ in the objective to get an $\epsilon-$approximate solution for the away-step FW algorithm, where $\rho$ is the pyramidal-width of the domain. Although FW is known to be an affine-invariant algorithm, the pyramidal width is affine-variant (e.g., see Figure \ref{fig:pyramidal width} for an example) and there are no known worst-case lower bounds on $\rho$ for general polytopes. In contrast, our rates remove this dependence on affine-variant geometric quantities. Therefore, {\sc PGD-CG} enjoys the best of both projection and projection-free regimes.  We make these comparisons with {\sc PGD} and {\sc AFW}, and benchmark all three methods, through a series of concrete examples.

\paragraph{Example 1 (Simplex).}  Recall from Table~\ref{tab:complexities over special cases} that, over the probability simplex, the running time guarantee for {\sc PGD-CG} is
$O\left( (\frac{n}{c} + n\log n) \frac{L}{\mu} \log \left(\frac{1}{\epsilon}\right)\right)$,
whereas {\sc PGD} enjoys the bound $O\left(  n\log n \frac{L}{\mu} \log \left(\frac{1}{\epsilon}\right)\right)$.
In particular, for an appropriate choice of \(c\) (e.g., such that \(n/c < n\log n\)), the {\sc PGD-CG} guarantee effectively interpolates between the \(\frac{n}{c}\) and \(n\log n\) terms. In contrast, the corresponding running time for {\sc AFW} over the simplex is
$O\!\left(n^2\,\frac{L}{\mu}\log\!\left(\frac{1}{\epsilon}\right)\right)$,
which is asymptotically worse than both {\sc PGD-CG} and {\sc PGD}. Moreover, if we consider the scaled simplex
$\sum_{e\neq e^\prime} x_e + Mx_{e^\prime}\leq 1$ with $x_e\geq 0$ and $M\gg 1$, then its pyramidal width satisfies $\rho=\Theta(1/M)$, which vanishes as \(M\) grows and further degrades the {\sc AFW} bound. By contrast, the running-time guarantees for both {\sc PGD-CG} and {\sc PGD} are invariant under this scaling and therefore remain unchanged.
We remark that for the relevant range of $c$ the {\sc PGD-CG} asymptotically improves over {\sc AFW}, while obtaining a reduced frequency of projection-oracle calls, which we observe in our experiments.

\paragraph{Example 2 (Base Polytopes).}
We will construct a family of integral submodular base polytopes $B(f) \subseteq \mathbb{R}^{2n}$, where ${B}(f) =\{x \in \mathbb{R}_+^{2n} \mid \sum_{i\in S}x_i \leq f(S) \text{ for all } S \subseteq [2n], \sum_{i\in E} x_i = f(E)\}$ for normalized positive integral submodular functions $f: 2^{[2n]} \rightarrow \mathbb{R}_+$. We will show that the pyramidal width of this family is $\rho = \Theta(n^{-1.5})$, diameter $D = \Theta(n^{1.5})$. Using the known results on the cost of computing Euclidean projections over general submodular polytopes is $O(n^4 \log n \cdot \text{EO} + n^5 \log^2 n)$ (using $O(n)$ submodular function minimizations \cite{jiang2020improved}, \cite{gupta2017combinatorial}), linear optimization $L_O$ using Edmonds's greedy algorithm in $O(n \log n + n \cdot \text{EO})$ time, and away-step computation in $O(n^2)$ \cite{Beck2017}, we will show a sharp trade-off between the running time of the three algorithms for $L$-smooth and $\mu$-strongly convex functions:
\vspace{10pt}
\begin{center}
    \begin{tabular}{|c|c|c|} \hline
        \sc{AFW} &  \sc{PGD} & \sc{PGD-CG} \\ \hline \hline
         $\mathcal{O}\!\left(n^8 \,\frac{L}{\mu}\,\log\frac{1}{\epsilon}\right)$ &
         $\displaystyle \mathcal{O}\!\left(n^5 \log^2 n \,\frac{L}{\mu}\,\log\frac{1}{\epsilon}\right)$ &
         $\mathcal{O}\!\left(\Big(\frac{n^4\log n}{c}+n^5\log^2 n\Big)\,\frac{L}{\mu}\,\log\frac{1}{\epsilon}\right)$ \\ \hline
    \end{tabular}
\end{center}
\vspace{10pt}
For $n$ even, consider the family of submodular functions $\mathcal{F} = \{f(S) = f_1(S \cap [n]) + f_2(|S \cap ([2n]\setminus [n]|), f_1: 2^{[n]} \rightarrow \mathbb{Z}_+, f_2(i) = \sum_{j=1}^i(n-j+1)\}$, where $f_1$ is the rank function of any matroid defined on $[n]$ with rank $O(n)$ (e.g., graphic matroid on $[n]$ edges), and $f_2$ is the cardinality-based submodular function corresponding to the permutahedron. The base polytope of functions in $\mathcal{F}$ are integral, and we claim that the pyramidal width of any integral submodular base polytope is $\Omega(n^{-1.5})$. Further, we show that for the permutahedron, the pyramidal width is $\Theta(n^{-1.5})$ (see Appendix \ref{sec:app-pyramidal}). Therefore, the pyramidal width of base polytopes for $\mathcal{F}$ is also $\Theta(n^{-1.5})$, as these are Cartesian products with the permutahedron, and the pyramidal width of these is at most the pyramidal width of its factors. The diameter of these polytopes is $O(n^{1.5})$ as long as there are two-disjoint bases of size $O(n)$, and this gives us the running times in the table above.   

In this example, both the AFW and PGD algorithms incur high running times due to projections in PGD and $(D/\rho)^2$ dependence of pyramidal width and diameter in AFW's running time. On the other hand, \textsc{PGD-CG} can trade-off some of the projection costs with $(n^4\log n)/c$ from its CG updates. The tuning parameter $c$ controls the bias toward Frank--Wolfe versus projection steps: larger $c$ suppresses the \textsc{CG} overhead (recovering {\sc PGD}-like behavior), whereas smaller $c$ increases reliance on cheaper {\sc LO} steps and reduces projection frequency in practice. Thus, when projections are relatively expensive but linear optimization is cheap, \textsc{PGD-CG} provides a principled interpolation between {\sc FW}- and projection-based updates while maintaining linear convergence. %These scalings instantiate a regime where projections are comparatively expensive (penalizing {\sc PGD}), the dependence of \textsc{AFW} on geometric constants such as $(D/\rho)$ can be unfavorable, and hybrid methods such as \textsc{PGD-CG} can capitalize on cheap linear optimization, while amortizing projection cost through the mixing parameter $c$.

To summarize, {\sc PGD-CG} is expected to perform particularly well when (i) (Euclidean) projections onto the feasible set are moderately expensive, so reducing their frequency is beneficial, and (ii) linear optimization over the feasible region is comparatively cheap, making conditional-gradient steps attractive. In addition, {\sc PGD-CG} can be preferable in regimes where AFW is adversely affected by unfavorable geometry---for instance, when the feasible set exhibits asymmetry leading to a small pyramidal width~$\rho$ and AFW's complexity deteriorates with the factor $D/\rho$. More broadly, when the geometry is ``unbalanced'' so that CG steps are effective only along certain directions, the hybrid structure of {\sc PGD-CG} can exploit those directions while falling back on projection steps when needed to quickly converge to the optimal solution. Seen this way, \textsc{PGD-CG} combines the \emph{optimism} of Frank--Wolfe steps---arising from the projection of an infinite gradient step---with the more \emph{pessimistic} projected-gradient steps, which are more local movements.

\section{Computational Complexity of Oracles}\label{sec:oracles}

\subsection{The {\textsc{Trace}} Procedure}  \label{app: trace opt}

\setlength{\textfloatsep}{2pt}
\begin{algorithm}[t]
\caption{Tracing Projections Curve {upto Step-Size}: \textsc{Trace}$({x}, {w}, {s}$)}
\label{alg:tracingPWproj}
\begin{algorithmic}[1]
\INPUT Polytope $\mathcal{P} \subseteq \mathbb{R}^n$, $L$-smooth function $h : \mathcal{P} \to \mathbb{R}$, initialization $x \in \mathcal{P}$, {direction $w \in \mathbb{R}^n$, and target step-size $s$.}
\State  Let $\iter{x}{0} = x$ and $\gamma^{\text{total}} = 0.$ \Comment{\texttt{fix starting point and initialize total step size}}
\State Compute ${d}_{{x}}^\Pi \coloneqq \lim_{\epsilon \downarrow 0} \frac{\Pi_P({x}- \epsilon w) - {x}}{\epsilon}$.
\While {${d}_{{x}}^\Pi \neq {0}$} \Comment{\texttt{check if we are at endpoint}}
\If{$\innerprod{\iter{x}{0} - \gamma^{\text{total}}{w} - x}{{d}_{{x}}^\Pi} = 0$} \Comment{\texttt{determine if we take shadow step}}
\State Compute $\gamma^{\max} = \max \{\delta \mid {x} + \delta {d}_{{x}}^\Pi \in \mathcal{P}\}$ \Comment{ \texttt{feasibility line-search in shadow direction}}
\State Set $d = {d}_{{x}}^\Pi$. \Comment{\texttt{take shadow step}}
\Else  %\State $\hat{d}_{{x}}^\Pi, \gamma^{\max} = \textsc{Trace-In-Face}(\iter{x}{0}, {x}, {\nabla h (\iter{x}{0})},\gamma^{\text{total}})$.  \Comment{\texttt{in-face step}}
\State Compute $\hat{d}_{{x}}^\Pi = ({I}- {A}_{I({x})}^\dagger {A}_{I({x})})(-{w})$ \Comment{\texttt{in-face directional derivative}}
\State {Evaluate $\theta^* = \max \{\theta \mid  N_{\mathcal{P}}(g(\gamma^{\text{total}} + \theta)) \supseteq N_{\mathcal{P}}(x)\}$ \Comment{\texttt{see the \eqref{in-face-lp}} LP}}
\State set $d = \hat{d}_{{x}}^\Pi$ and {$\gamma^{\max} = \theta^*$}.  \Comment{\texttt{in-face step}}
\EndIf
\If{$\gamma^{\text{total}} + \gamma^{\max} \geq {s}$ and $d \neq 0$} 
\State Compute $\gamma^* \in  \argmin_{\gamma \in [0,\gamma^{\max}]} h({x} + \gamma d)$.  \Comment{ \texttt{optimality line-search}}
\State Update ${x}= {x}+ \gamma^*d$
\State \textbf{break} \Comment{\texttt{we made sufficient descent progress}}
\Else
\State Update ${x}= {x}+ \gamma^{\max}d$ and $\gamma^{\text{total}} = \gamma^{\text{total}} + \gamma^{\max}$. \Comment{\texttt{keep track of total step-size accrued}}
\State Recompute ${d}_{{x}}^\Pi \coloneqq \lim_{\epsilon \downarrow 0} \frac{\Pi_P({x}- \epsilon {\nabla h (\iter{x}{0})}) - {x}}{\epsilon}$  
\EndIf
\EndWhile
\RETURN ${x}$
\end{algorithmic}
\vspace{0.5cm}
\end{algorithm}

We now introduce the \textsc{Trace} algorithm, which chains together successive short descent steps along the projections curve, interleaved with line-search over its linear segments, until a prescribed step-size $s$ is reached. This routine is invoked within \textsc{Shadow-Walk} and \textsc{Shadow-CG} to emulate a projected gradient descent (PGD) step: the curve is traced until it reaches the canonical PGD step-size $1/L$ (where $L$ is the smoothness constant), followed by line-search minimization along the active linear segment. %In this way, {\textsc{Trace}} guarantees progress at least as good as a PGD step of fixed step-size $1/L$.  
The additional line-search is not needed for feasibility or correctness. Rather, after the routine has traced the projections curve to the segment containing the $1/L$ point, a single line-search over that segment can only improve on the corresponding PGD iterate. Thus, the returned point guarantees progress no worse than the fixed-step PGD point, and it is often strictly better in practice.%\footnote{{Deleting - this improvement is one reason for the empirical advantage of \textsc{Shadow-Walk} over PGD observed in the next section}}. 

Note that \textsc{Trace} requires only a single gradient oracle query. The correctness of the procedure follows directly from Theorem~\ref{invariance}. 

Though bounding the number of iterations (breakpoints) of {\textsc{Trace}} for a general polytope remains open beyond the simple exponential bound of Theorem \ref{breakpoints}, polynomial bounds can be established for important special cases. In particular, we show that for both the probability simplex and the hypercube, the projection curve has at most \(n\) breakpoints, where \(n\) is the ambient dimension. Moreover, \cite{ahunbay2024first} extends this phenomenon beyond these two examples to the broader class of \emph{acute} polytopes. A polytope $\mathcal{P} \coloneqq \left\{ x \in \mathbb{R}^d \;\middle|\; \innerprod{a_j}{x} \leq b_j \ \ \forall\, j \in [m] \right\}$
is called \emph{acute} if \(\innerprod{a_j}{a_{j'}} \leq 0\) for every distinct \(j,j'\in[m]\). (Both the simplex and the hypercube are acute.) Recall that the convergence rate of AFW depends on the pyramidal width of \(\mathcal P\). To the best of our knowledge, this quantity has been explicitly evaluated only for the hypercube \(\bigl(\rho = 1/\sqrt{n}\bigr)\) and the probability simplex \(\bigl(\rho = 2/\sqrt{n}\bigr)\), reflecting the difficulty of computing it in general \cite{Lacoste2015}.

Combining these known pyramidal-width values with our breakpoint bounds, we show that \textsc{PGD-CG} enjoys an \(\Omega(n^2)\) improvement in convergence rates and iteration complexity over \textsc{AFW}; see Table~\ref{tab:complexities over special cases}. More broadly, since the breakpoint bound extends to any acute polytope (not just the simplex and hypercube), the projection-curve framework yields provable guarantees on a family of polytopes that goes well beyond the standard benchmark examples for AFW. We remark that these geometric bounds might also be of independent interest to the discrete geometry community.

\subsection{Oracles for the Hypercube}

Consider the $n$-dimensional hypercube $\mathcal H_n \coloneqq \{x\in\mathbb R^n \mid 0\leq x_i \leq 1 \ \forall i\in[n]\}$, and
and fix any $x\in \mathcal H_n$. Recall that for any polytope $\mathcal P$ and any direction $w\in\mathbb R^n$, the shadow at $x$ with respect $w$ is:
\begin{equation}\label{eq:proj-deriv}
d_x^\Pi(w)
\coloneqq \lim_{\epsilon\downarrow 0}\frac{\Pi_{\mathcal P}(x-\epsilon w)-x}{\epsilon}
= \argmin_{d\in T_{\mathcal P}(x)} \| -w - d\|^2,
\end{equation}
i.e., $d_x^\Pi(w)$ is the Euclidean projection of $-w$ onto the tangent cone $T_{\mathcal P}(x)$. For $\mathcal P=\mathcal H_n$, the tangent cone is $T_{\mathcal H_n}(x)= \bigl\{d\in\mathbb R^n \mid d_i\geq 0 \text{ if } x_i=0,\ \ d_i\leq 0 \text{ if } x_i=1 \bigr\}.$
Thus,
\[
d_x^\Pi(w)=\argmin_{d\in T_{\mathcal H_n}(x)} \|-w-d\|^2
=
\left\{
\begin{aligned}
&\argmin_{d\in\mathbb R^n} && \|-w-d\|^2\\
&\text{subject to} && d_i\geq 0 \quad \text{if } i\in I_1,\\
& && d_i\leq 0 \quad \text{if } i\in I_2,
\end{aligned}
\right.
\]
where $I_1\coloneqq\{i\in[n]\mid x_i=0\}$ and $I_2\coloneqq\{i\in[n]\mid x_i=1\}$ are the index sets of tight constraints at $x$.

The objective is separable across coordinates and the constraints are coordinate-wise, so the minimizer is obtained by projecting each scalar $[-w]_i$ onto the corresponding feasible set for $d_i$:
\begin{equation}\label{hypercube-shadow}
[d_x^\Pi(w)]_i
=
\begin{cases}
0 & \text{if } i\in I_1 \text{ and } [-w]_i\leq 0,\\
0 & \text{if } i\in I_2 \text{ and } [-w]_i\geq 0,\\
[-w]_i & \text{otherwise}.
\end{cases}
\end{equation}
Indeed, if $x_i=0$ then we require $d_i\geq 0$, so the closest feasible value to $[-w]_i$ is $\max\{[-w]_i,0\}$; similarly, if $x_i=1$ then $d_i\leq 0$ and the closest feasible value is $\min\{[-w]_i,0\}$; and if $0<x_i<1$ there is no constraint and the minimizer is $d_i=[-w]_i$. Note that $d_x^\Pi(w)$ is defined via an $\epsilon$-step in the \emph{negative} $w$ direction, which is why the expression involves $-w$. Equation \eqref{hypercube-shadow} can be evaluated coordinate-wise in $O(n)$ time.

Interestingly, we can further show that there are no in-face steps in the projection curve, and therefore, the number of breakpoints of the projections curve is at most $n$:

\begin{restatable}[Breakpoints for the  Hypercube]{theorem}{hypercubebreakpoints}
\label{hypercube breakpoints}
{Consider the $n$-dimensional hypercube  $\mathcal{H}_n \coloneqq \{ x \in \mathbb{R}^n \mid 0 \leq x_i \leq 1\; \forall\; i \in [n]\}$, and fix ${x}_0 \in \mathcal{H}_n$ and direction $w \in \mathbb{R}^n$. Then, the projections curve $g(\lambda) = \Pi_{P}(x_0 - \lambda w)$ has only $O(n)$ breakpoints (for $\lambda\geq 0$).}
\end{restatable}

\noindent\textit{Proof.}
Let $g(\lambda)\coloneqq \Pi_{\mathcal H_n}(x_0-\lambda w)$ be the projections curve over the hypercube. Define $x_i$ to be the $i$-th breakpoint for the projections curve and the shadow $d_i \coloneqq d_{x_i}^\Pi(w)$ for brevity, where $i \in \{ 0, 1, \dots, \beta\}$. We first show by induction that the inner-product test of Theorem~\ref{invariance} holds with equality at
every breakpoint:
\begin{equation}\label{eq:cube-ip-test}
\innerprod{x_0-\lambda_i^- w-x_i}{d_i}=0\qquad\forall\, i=0,1,\dots,\beta-1.
\end{equation}
To simplify the proof, we use $\lambda_i^-$ (rather than $\lambda_i^+$ as in Theorem~\ref{invariance}), which is sufficient here: whenever
\eqref{eq:cube-ip-test} holds and $d_i\neq 0$, \textsc{Trace} takes a (maximal) shadow step immediately and
leaves the breakpoint, so $\lambda_i^-=\lambda_i^+$ (i.e., there is no trivial in-face movement). The only
exception is the terminal case: $d_i=0$, in which the curve has ended and \textsc{Trace} terminates.

For the base case $i=0$, since $\lambda_0^- = 0$, the result is immediate: $\innerprod{x_0-\lambda_0^- w-x_0}{d_0}=\innerprod{0}{d_0}=0$. Now, for the inductive step, fix $i\geq 1$ and assume \eqref{eq:cube-ip-test} holds for all $j\leq i$. Consider the $k$-th coordinate and suppose $[d_i]_k\neq 0$. By the hypercube shadow formula
\eqref{hypercube-shadow}, this implies that coordinate $k$ is not saturated, i.e. not blocked by an active constraint at $x_i$, and
hence
\begin{equation}\label{eq:cube-di-equals-minusw}
[d_i]_k=[-w]_k.
\end{equation}
Using the induction hypothesis that we took a shadow step at breakpoint $i$, we have $x(\lambda_{i+1}^-)=x_i+(\lambda_{i+1}^--\lambda_i^+)d_i = x(\lambda_{i+1}^-)=x_i+(\lambda_{i+1}^--\lambda_i^-)d_i$, where the second equality holds because $\lambda_i^- = \lambda_i^+$ as we took a shadow step and $d_i \neq 0$. Hence,
\begin{equation}\label{eq:cube-residual-expand}
x_0-\lambda_{i+1}^-w-x(\lambda_{i+1}^-)
=\bigl(x_0-\lambda_i^- w-x_i\bigr)-(\lambda_{i+1}^--\lambda_i^-)\bigl(w+d_i\bigr).
\end{equation}

Taking the $k$-th coordinate yields
\begin{equation}\label{eq:cube-coordinate-residual}
[x_0-\lambda_{i+1}^-w-x(\lambda_{i+1}^-)]_k
=
[x_0-\lambda_i^- w-x_i]_k
-(\lambda_{i+1}^--\lambda_i^-)\bigl([w]_k+[d_i]_k\bigr).
\end{equation}

Now suppose $[d_i]_k\neq 0$. Then \eqref{eq:cube-di-equals-minusw} implies $[w]_k+[d_i]_k=0$, so the second term
in \eqref{eq:cube-coordinate-residual} vanishes. Moreover, $[d_i]_k\neq 0$ means that coordinate $k$ is
unclamped on $[\lambda_{i}^-,\lambda_{i+1}^-]$ by induction since we only took shadow steps to $x_i$, and therefore:
\[
[g(\lambda)]_k=[x_0-\lambda w]_k\qquad\forall\,\lambda\in[\lambda_i^-,\lambda_{i+1}^-].
\]
In particular,
\[
[x_0-\lambda_{i+1}^-w-x(\lambda_{i+1}^-)]_k=0
\qquad\text{whenever }[d_i]_k\neq 0.
\]
Consequently,
\[
\innerprod{x_0-\lambda_{i+1}^-w-x(\lambda_{i+1}^-)}{d_i}
=\sum_{k=1}^n [x_0-\lambda_{i+1}^-w-g(\lambda_{i+1}^-)]_k\,[d_i]_k
=0,
\]
{Thus, $[p_{i+1}]_k=0$ for every $k\in\supp(d_i)$. Moreover, since $w$ is fixed along the projection curve, once a coordinate becomes clamped at a bound it remains clamped thereafter; hence
$\supp(d_{i+1})\subseteq\supp(d_i)$. Therefore,
$[p_{i+1}]_k=0$ for every $k\in\supp(d_{i+1})$, and consequently
$\innerprod{p_{i+1}}{d_{i+1}}=0$, proving the inductive claim at $i+1$.} This completes the induction and proves
\eqref{eq:cube-ip-test}. By Theorem~\ref{invariance}, \textsc{Trace} therefore takes only maximal shadow
steps on the hypercube (i.e., no in-face steps occur).

It remains to bound the number of maximal shadow steps. Each maximal shadow step ends precisely when at
least one previously non-tight coordinate becomes tight (hits $0$ or $1$). Once a coordinate is tight, it
remains tight thereafter: {since the direction $w$ is fixed along the whole curve, a coordinate reaches its bound only while moving with velocity $[-w]_k$ directed into that bound, and by \eqref{hypercube-shadow} it is therefore blocked at every subsequent breakpoint}. Hence every maximal shadow step can newly activate at least one
coordinate, and there are at most $n$ coordinates in total. After at most $n$ such activations we must have
$d_i=0$, which corresponds to reaching the endpoint of the projection curve. Thus $\beta\leq n$, and in
particular the number of breakpoints is $O(n)$.
\hfill$\square$

By the coordinate-wise form of the shadow direction in \eqref{hypercube-shadow}, a single shadow computation can be carried out in $S=O(n)$ time. Moreover, by the preceding theorem the projection curve has $\beta=O(n)$ breakpoints. The resulting complexity implications are summarized in Table~\ref{tab:complexities over special cases}.

\subsection{Oracles for the Simplex}
In the previous section we proved that, for the hypercube, the projection curve has only a linear number of breakpoints in the ambient dimension. We now show that the same phenomenon holds for the probability simplex.

\begin{restatable}[Breakpoints for the Simplex]{theorem}{simplexbreakpoints}
\label{simplex breakpoints}
Let $\Delta_n \coloneqq \{x\in\mathbb R^n \ |\ \sum_{i=1}^n x_i = 1,\ \ x_i\geq 0\ \forall i\in[n]\}$ be the $(n-1)$-dimensional simplex, and fix $x_0\in\Delta_n$ and a direction $w\in\mathbb R^n$.
Then the projection curve $g(\lambda)=\Pi_{\Delta_n}(x_0-\lambda w)$ has $O(n)$ breakpoints for $\lambda\geq 0$.
\end{restatable}

The proof proceeds in two steps. First, we characterize the \emph{structure} of the simplex shadow direction and show that it can be computed in $O(n\log n)$ time. Second, using this structure, we show that the projection curve (as in the hypercube case) does \emph{not} take in-face steps; instead, it always performs \emph{maximal} shadow steps. Consequently, each maximal shadow step can tighten at least one previously non-tight coordinate, and therefore at most $n$ such steps can occur, yielding the desired breakpoints bound.

\begin{lemma}[Structure and cost of the simplex shadow]\label{lem:simplex-shadow-structure}
Let $\Delta_n\coloneqq\{x\in\mathbb R^n:\ \mathbf 1^\top x=1,\ x\geq 0\}$. Fix $x\in\Delta_n$ and $w\in\mathbb R^n$, and set $v\coloneqq -w$. Let $d\coloneqq d_x^\Pi(w)$ denote the shadow direction at $x$ with respect to $w$, and define the active-index set $I(x)\coloneqq\{i\in[n]: x_i=0\}$. Let $I^\star(x)\coloneqq\{i\in I(x): d_i=0\}$ and let $u_x\in\{0,1\}^n$ be the indicator vector of $[n]\setminus I^\star(x)$. Then, the shadow satisfies\footnote{We use $\odot$ to denote the Hadamard product.}
\begin{equation}\label{eq:simplex-shadow-closed-form}
d \;=\; v\odot u_x \;-\;\frac{\langle v,u_x\rangle}{\|u_x\|^2}\,u_x,
\end{equation}
and the constraints $d_i\geq 0$ for $i\in I(x)$ are equivalent to the threshold inequalities
$$v_i \leq \frac{\langle v,u_x\rangle}{\|u_x\|^2}\ (i\in I^\star(x)),\qquad
v_i > \frac{\langle v,u_x\rangle}{\|u_x\|^2}\ (i\in I(x)\setminus I^\star(x)). $$
Moreover, the shadow direction $d$ can be computed in $O(n\log n)$ time.
\end{lemma}

\noindent\emph{Proof.} We first show the closed-form characterization given in the lemma, and then use that characterization to prove the runtime bound.
\vspace{5 pt}

\noindent{\textbf{Closed-form characterization.}}
Consider the convex QP to compute $d$, the shadow at $x$ with respect to $w$:
\begin{equation}\label{eq:simplex-shadow-QP}
\min_{d\in\mathbb R^n}\ \frac12\|v-d\|^2
\quad\text{s.t.}\quad
\mathbf 1^\top d=0,\qquad d_i\geq 0 \ \ \forall i\in I(x).
\end{equation}
Let $d^\star$ denote its minimizer and define $I^\star(x)\coloneqq\{i\in I(x)\mid d^\star_i=0\}$.
Let $\eta\in\mathbb R$ be the multiplier for $\mathbf 1^\top d=0$ and let $\mu\in\mathbb R_+^{|I(x)|}$ be the
multipliers for the inequalities. The KKT conditions are:
\[
d^\star-v+\eta\mathbf 1 - \tilde\mu = 0,
\qquad
\mathbf 1^\top d^\star=0,
\qquad
d^\star_i\geq 0\ (i\in I(x)),
\qquad
\tilde\mu_i\geq 0\ (i\in I(x)),
\qquad
\tilde\mu_i d^\star_i=0\ (i\in I(x)),
\]
where $\tilde\mu\in\mathbb R^n$ extends $\mu$ by setting $\tilde\mu_i=0$ for $i\notin I(x)$.
If $i\notin I^\star(x)$, then either $i\notin I(x)$ or $d^\star_i>0$ with $i\in I(x)$; in both cases,
$\tilde\mu_i=0$ by complementary slackness, and therefore $d^\star_i=v_i-\eta$.
If $i\in I^\star(x)$, then $d^\star_i=0$. Let $u_x$ be the indicator of $[n]\setminus I^\star(x)$. Then the
stationarity condition gives $d^\star=v\odot u_x-\eta u_x$. Enforcing $\mathbf 1^\top d^\star=0$ yields
\[
0=\innerprod{v}{u_x}-\eta\|u_x\|^2
\quad\Longrightarrow\quad
\eta=\frac{\innerprod{v}{u_x}}{\|u_x\|^2},
\]
which gives \eqref{eq:simplex-shadow-closed-form}. Finally, feasibility $d^\star_i\geq 0$ for $i\in I(x)$ is
equivalent to $v_i-\eta > 0$ on $I(x)\setminus I^\star(x)$ and $v_i-\eta\leq 0$ on $I^\star(x)$, i.e., the stated
threshold inequalities. Since $d^\star$ is the unique minimizer of \eqref{eq:simplex-shadow-QP}, we conclude
$d^\star=d_x^\Pi(w)$.

\noindent\emph{\textbf{Runtime bound.}}
The KKT system shows that the shadow is determined by the threshold
$\eta=\frac{\innerprod{v}{u_x}}{\|u_x\|^2}$ and an active set $I^\star(x)\subseteq I(x)$ satisfying the consistency
inequalities $v_i\leq \eta$ for $i\in I^\star(x)$ and $v_i > \eta$ for $i\in I(x)\setminus I^\star(x)$.
Equivalently, $I^\star(x)$ is obtained by thresholding the values $\{v_i:i\in I(x)\}$ at $\eta$. Thus we can compute
$I^\star(x)$ by sorting $\{v_i:i\in I(x)\}$ in $O(|I(x)|\log|I(x)|)$ time and scanning for the unique split point at
which the implied $\eta$ is consistent with these inequalities. Once $\eta$ and $u_x$ are determined, \eqref{eq:simplex-shadow-closed-form} evaluates $d_x^\Pi(w)$ in $O(n)$ time. Thus, shadow computation costs
$O(n\log n)$ time.
\hfill$\square$

\paragraph{Proof of Theorem~\ref{simplex breakpoints}.} Let $v \coloneqq -w$ and $g(\lambda)\coloneqq \Pi_{\Delta_n}(x_0 + \lambda v)$ be the projections curve over the the simplex. Define $x_i$ to be the $i$-th breakpoint for the projections curve, and denote the shadow $d_i \coloneqq d_{x_i}^\Pi(w)$ for brevity, where $i \in \{ 0, 1, \dots, \beta\}$. For each $i$, let $u_i\coloneqq u_{x_i}\in\{0,1\}^n$ be the indicator vector from
Lemma~\ref{lem:simplex-shadow-structure}. Then
\begin{equation}\label{eq:simplex-di-form}
d_i \;=\; v\odot u_i \;-\;\frac{\innerprod{v}{u_i}}{\|u_i\|^2}\,u_i.
\end{equation}

Similar to the proof for the hypercube, we prove by induction on $i\geq 0$ that the inner-product test of Theorem~\ref{invariance} holds with equality at
every breakpoint:
\begin{equation}\label{eq:simplex-ip-test}
\innerprod{x_0+\lambda_i^- v-x_i}{d_i}=0\qquad \forall\, i=0,1,\dots,\beta-1.
\end{equation}
Using $\lambda_i^-$ (rather than $\lambda_i^+$) is again sufficient: whenever \eqref{eq:simplex-ip-test} holds and
$d_i\neq 0$, \textsc{Trace} immediately takes a (maximal) shadow step and leaves the breakpoint, so
$\lambda_i^-=\lambda_i^+$; the only exception is the terminal case $d_i=0$, in which the projection curve has
ended and \textsc{Trace} terminates.

\smallskip
\noindent\emph{Base case ($i=0$).} Since $\lambda_0^-=0$ and $x_0=g(\lambda_0)$,
$\innerprod{x_0+\lambda_0^- v-x_0}{d_0}=\innerprod{0}{d_0}=0$.

\smallskip
\noindent\emph{Inductive step.} Fix $i\geq 1$ and assume \eqref{eq:simplex-ip-test} holds for all $j\leq i-1$.
Then, by Theorem~\ref{invariance}, \textsc{Trace} has taken only shadow steps up to breakpoint $x_i$, and hence $x_i \;=\; x_0+\sum_{j=0}^{i-1}(\lambda_{j+1}^- - \lambda_j^-)\,d_j$. Therefore
$$x_0+\lambda_i^- v-x_i
=\sum_{j=0}^{i-1}(\lambda_{j+1}^- - \lambda_j^-)\,(v-d_j), $$
and taking the inner product with $d_i$ yields
\begin{equation}\label{eq:simplex-ip-sum}
\innerprod{x_0+\lambda_i^- v-x_i}{d_i}
=
\sum_{j=0}^{i-1}(\lambda_{j+1}^- - \lambda_j^-)\,\innerprod{v-d_j}{d_i}.
\end{equation}

We now show that each term $\innerprod{v-d_j}{d_i}$ vanishes. Since each iteration of \textsc{Trace} uses a
\emph{maximal} shadow step, at least one additional coordinate becomes tight at each new breakpoint and remains
tight thereafter; equivalently, using the structure of the shadow in Lemma \ref{lem:simplex-shadow-structure}, the supports of the indicator vectors are nested:
\begin{equation}\label{eq:simplex-support-nesting}
\supp(u_i)\subseteq \supp(u_j)\qquad \text{for all } j<i.
\end{equation}
Fix $j<i$ and write $u\coloneqq u_j$ and $s\coloneqq u_i$. Then $\supp(s)\subseteq\supp(u)$ implies
$u\odot s=s$ and $\innerprod{u}{s}=\|s\|^2$. Using \eqref{eq:simplex-di-form} we expand
\[
\innerprod{v-d_j}{d_i}
=
\innerprod{\,v-\bigl(v\odot u-\tfrac{\innerprod{v}{u}}{\|u\|^2}u\bigr)\,}
{\,v\odot s-\tfrac{\innerprod{v}{s}}{\|s\|^2}s\,}
=
\innerprod{v-v\odot u}{d_i}
+
\frac{\innerprod{v}{u}}{\|u\|^2}\,\innerprod{u}{d_i}.
\]
The first term is zero because $v-v\odot u$ is supported on $[n]\setminus\supp(u)$, whereas $d_i$ is supported on
$\supp(s)\subseteq\supp(u)$. For the second term,
\[
\innerprod{u}{d_i}
=
\innerprod{u}{v\odot s}-\frac{\innerprod{v}{s}}{\|s\|^2}\innerprod{u}{s}
=
\innerprod{v}{s}-\frac{\innerprod{v}{s}}{\|s\|^2}\|s\|^2
=0,
\]
where we used $u\odot s=s$ and $\innerprod{u}{s}=\|s\|^2$. Hence $\innerprod{v-d_j}{d_i}=0$ for every $j<i$, and
\eqref{eq:simplex-ip-sum} gives $\innerprod{x_0+\lambda_i^- v-x_i}{d_i}=0$. This completes the induction and proves
\eqref{eq:simplex-ip-test}. Therefore, by Theorem~\ref{invariance}, \textsc{Trace} performs only maximal shadow
steps on $\Delta_n$ (i.e., no nontrivial in-face steps occur).

Finally, by \eqref{eq:simplex-support-nesting}, the supports $\supp(u_i)$ form a strictly descending chain of subsets of $[n]$ until termination. Once $|\supp(u_i)|=1$, \eqref{eq:simplex-di-form} gives $d_i=0$, which corresponds to reaching the endpoint of the projection curve. Thus, there can be at most $n$ breakpoints, and in particular $\beta=O(n)$. \hfill$\square$

Using the previous Theorem, we know that we can trace the entire projections curve over $\Delta_n$ in $O(n^2 \log n)$ time. The resulting complexity implications are summarized in Table~\ref{tab:complexities over special cases}.

\subsection{The Oracles used in our Experiments}\label{sec:more-oracles-exp}
In addition to the standard LMO used by conditional-gradient methods, our algorithms require (or benefit from) the
following auxiliary oracles. Recall that we write $\mathcal P=\{x\in\mathbb R^n\mid A x\leq b\}$ and let
$I(x)\subseteq[m]$ denote the active set at $x$.

\begin{itemize}
\item[(i)] \textbf{Linear optimization (LMO).}
Given $c\in\mathbb R^n$, compute
\[
v=\argmin_{x\in\mathcal P}\ \innerprod{c}{x}.
\]
The implementation of this subproblem is application-dependent. In our video co-localization experiments, the
feasible region is a \emph{flow polytope} (convex hull of $s$--$t$ paths in a DAG), so the LMO reduces to a shortest-path computation on the DAG, {which in this application has special structure: the feasible region decomposes across frames, so the LMO amounts to selecting, for each frame, the bounding box with minimum gradient entry, which can be computed in $O(n)$ time.} In our lasso experiments (using the standard simplex extended formulation of the $\ell_1$ ball; see \cite{combettes2020}), the LMO admits a closed form: it chooses an extreme point supported on the index $j\in[n]$ with largest magnitude entry of $c$, i.e., $j\in\argmax_{i\in[n]} |c_i|$, and returns the corresponding vertex $v\in\mathcal P$ that minimizes $\langle c,x\rangle$ in $O(n)$ time.

\item[(ii)] \textbf{Projection oracle.}
Given $y\in\mathbb R^n$, compute the Euclidean projection $\Pi_{\mathcal P}(y)$.
This oracle is used by projected-gradient descent (PGD) in our experiments. For the flow-polytope instance in
Section~\ref{sec:video}, we compute $\Pi_{\mathcal P}(y)$ by solving the corresponding quadratic program with
\texttt{Gurobi}. For the simplex instances arising in our Lasso experiments (Section~\ref{regression}), we use the $O(n\log n)$ projection
algorithm of Moondra et al.~\cite{moondra2021reusing}.

\item[(iii)] \textbf{Shadow computation.}
Given $x\in\mathcal P$ and $w\in\mathbb R^n$, compute the shadow direction $d_x^\Pi(w)$.
In general polytopes, we compute $d_x^\Pi(w)$ by solving the tangent-cone projection QP
\[
d_x^\Pi(w)\;=\;\argmin_{d\in\mathbb R^n}\Bigl\{\ \|-w-d\|^2 \ \Bigm|\ A_{I(x)}d\leq 0\ \Bigr\},
\]
which we again solve using \texttt{Gurobi} in the video co-localization experiments (Section~\ref{sec:video}).
In the lasso experiments (Section~\ref{regression}), the feasible region is (via an extended formulation) a simplex, so we instead compute $d_x^\Pi(w)$ and trace the projection curve using our specialized simplex routine (Lemma \ref{lem:simplex-shadow-structure}).

\item[(iv)] \textbf{Feasibility / maximum step size.}
Given $x\in\mathcal P$ and a direction $d\in\mathbb R^n$, compute
\[
\gamma^{\max} \;=\; \max\{\delta\geq 0 : x+\delta d\in\mathcal P\}
\;=\;
\min_{\substack{j\in[m]\setminus I(x)\\ \innerprod{a_j}{d}>0}}
\frac{b_j-\innerprod{a_j}{x}}{\innerprod{a_j}{d}}.
\]
In our experiments, this computation is inexpensive because the relevant polytopes have a polynomial number of
constraints (and in some cases, such as the simplex/hypercube/extended $\ell_1$ formulations, the constraints are
highly structured and the minimum can be computed in a single pass).

\item[(v)] \textbf{Line-search.}
Given $x\in\mathcal P$ and $d\in\mathbb R^n$, solve the one-dimensional problem
\[
\min_{\gamma\in[0,\gamma^{\max}]}\ h(x+\gamma d).
\]
We use a standard bracketing routine (golden-section search; see, e.g., \cite{bertsekas1997}) to obtain the
minimizer to a prescribed tolerance.
\end{itemize}

\section{Experiments}\label{sec:experiments}
{We implemented all algorithms in \texttt{Python} 3.12, and used \texttt{Gurobi} 12 \cite{Gurobi2025} as a black box solver for some of the oracles assumed in the paper as explained in the previous section. All experiments were performed on a 14-inch 2023 MacBook Pro with an Apple M2 Pro chip and 16~GB of unified memory, running macOS Tahoe (version~26.1)\footnote{This code and datasets used for our computations are available at \url{https://github.com/hassanmortagy/Walking-in-the-Shadow}.}.}

\begin{figure}[!t]
\vspace{-12 pt}
    \centering
    \includegraphics[scale = 0.6]{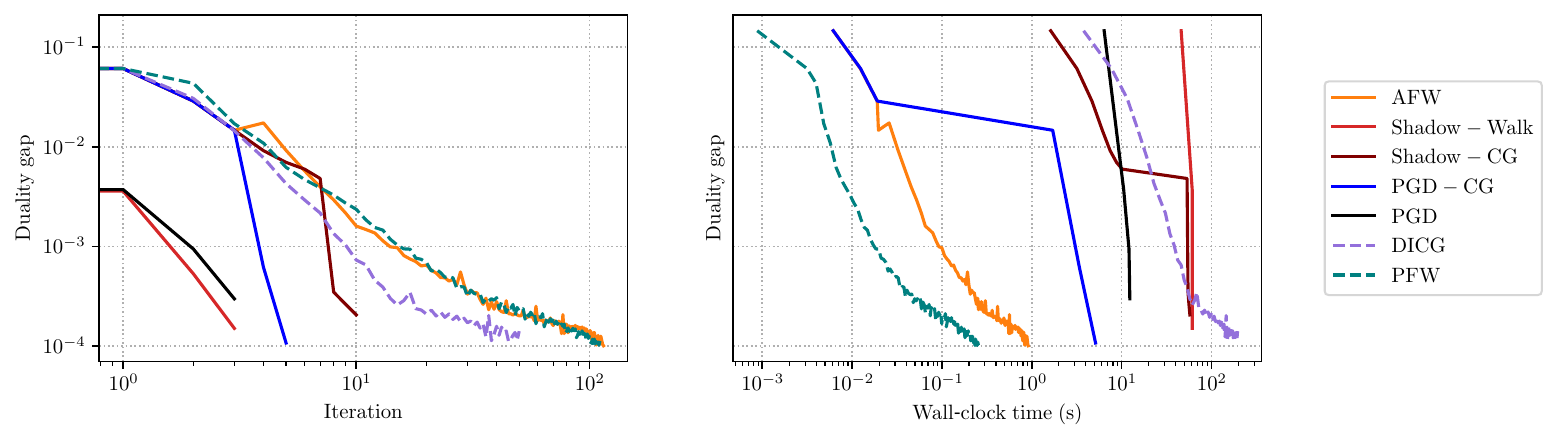}
    \caption{Duality gap for the video co-localization problem. Left plot compares iteration count, while the right plot compares wall-clock time.}
    \label{fig: computations - video2}
\end{figure}
\begin{figure}[!t]
\vspace{-10 pt}
    \centering
    \includegraphics[scale = 0.6]{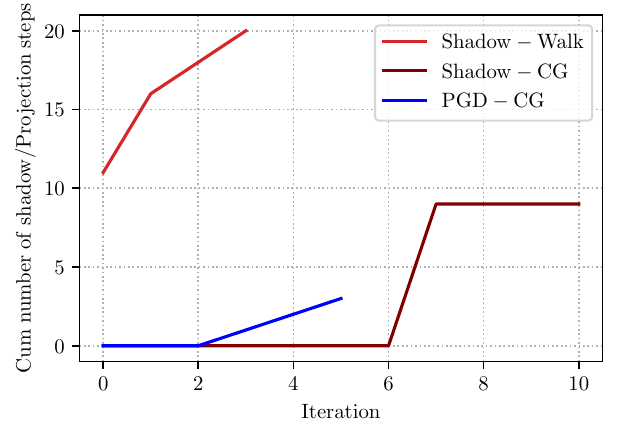}
    \includegraphics[scale = 0.6]{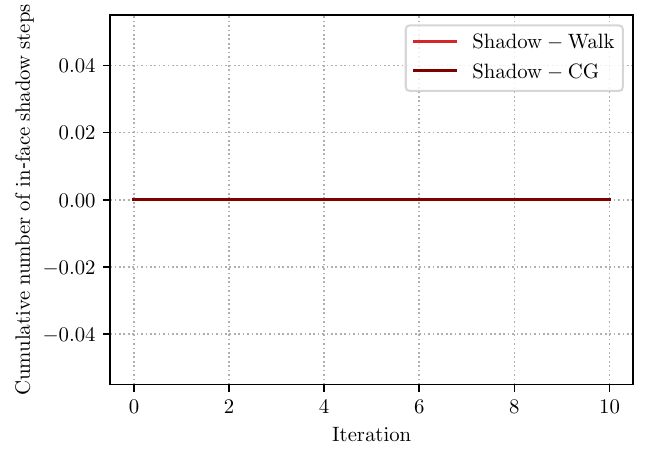}
    \caption{{Oracle calls for the video co-localization problem. Left plot shows the cumulative number of shadow or projection oracle calls made by {\sc Shadow-Walk}, {\sc  Shadow-CG}, and {\sc PGD-CG} against the iteration count. Right plot compares the cumulative number of in-face shadow steps taken. Here we see that none of the algorithms took any in-face steps.}}
    \label{fig: computations - video3}
\end{figure}

\subsection{Video Co-localization} \label{sec:video}
{
The first application we consider is the video co-localization problem from computer vision, where the goal is to track an object across different video frames. We used the YouTube-Objects dataset\footnote{We obtained the data from \url{https://github.com/Simon-Lacoste-Julien/linearFW}.} and the problem formulation of Joulin et. al \cite{Joulin2014}. This consists of minimizing a quadratic function $h(x) = \frac{1}{2}x^\top Ax + {b}^\top  x$, where $x \in \mathbb{R}^{660}$, $A \in \mathbb{R}^{660 \times 660}$ and ${b}\in \mathbb{R}^{660}$, over a flow polytope, the convex hull of paths in a network. Our linear minimization oracle over the flow polytope amounts to computing a shortest path in the corresponding directed acyclic graph. We now present the computational results.

As is standard in the analysis of FW variants, we compare the duality gap $\innerprod{-\nabla h(\iter{x}{t})}{\iter{v}{t} - \iter{x}{t}}$ \eqref{Wolfe-gap} across the algorithms proposed in this paper and several popular CG variants from the literature: Away-step FW (AFW), Pairwise FW (PFW), and Decomposition-Invariant CG (DICG) in Figure~\ref{fig: computations - video2}. We first observe that {\sc Shadow-Walk} requires the fewest iterations among all methods, thanks to its ability to trace the projection curve (up to the $1/L$ step size), which was particularly advantageous in this experiment. Furthermore, {\sc Shadow-CG} and {\sc PGD-CG} achieve roughly \emph{two orders of magnitude fewer iterations} than the CG variants DICG, AFW, and PFW (and only slightly more than {\sc PGD}), as expected. In particular, while {\sc PGD-CG} uses slightly more iterations than {\sc PGD}, it attains substantially lower wall-clock time—comparable to the CG variants—thus combining the advantages of both approaches.

We also plot the cumulative number of projection, shadow, and in-face steps taken by our proposed algorithms in Figure~\ref{fig: computations - video3}. We first observe that none of the methods ever takes an in-face step (these are only applicable to {\sc Shadow-Walk} and {\sc Shadow-CG}). Furthermore, {\sc Shadow-CG} and {\sc PGD-CG} both require fewer projection/shadow steps than {\sc Shadow-Walk}, with {\sc PGD-CG} using the fewest projections overall and converging substantially faster than both {\sc Shadow-Walk} and {\sc Shadow-CG} as expected.}

\begin{figure}[t]
    \centering
    \includegraphics[scale = 0.6]{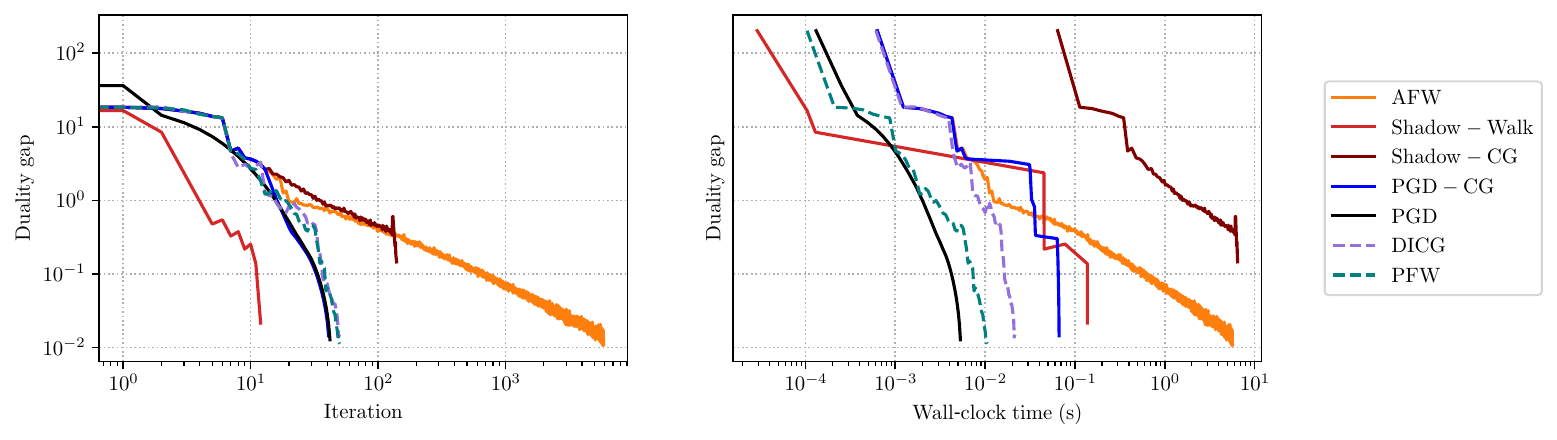}
    \caption{{Duality gap for the lasso regression problem. Left plot compares iteration count, while the right plots compare wall-clock time.}}
    \label{fig: computations - lasso2}
\end{figure}

\begin{figure}[t]
    \centering
    \includegraphics[scale = 0.6]{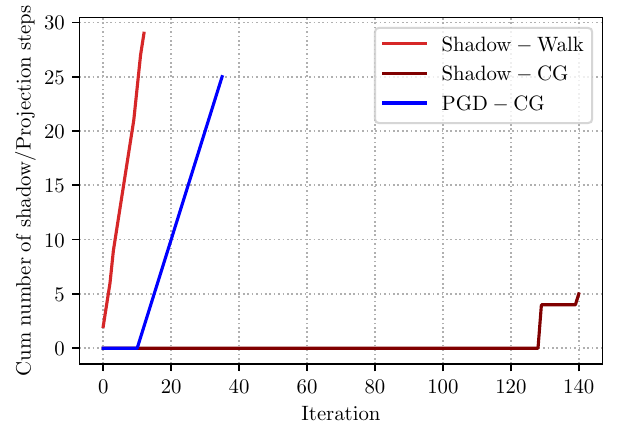}
    \includegraphics[scale = 0.6]{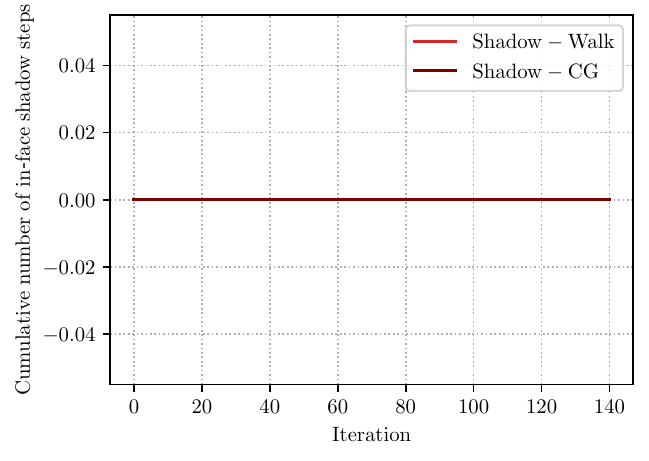}
    \caption{{Oracle calls for the lasso regression problem. Left plot shows the cumulative number of shadow or projection oracle calls made by {\sc Shadow-Walk}, {\sc  Shadow-CG}, and {\sc PGD-CG} against the iteration count. Right plot compares the cumulative number of in-face shadow steps taken. We find that we do not take any in-face steps and the number of iterations spent in the {\sc Trace} procedure is well within the $O(n)$ bound we prove in Theorem \ref{simplex breakpoints}.}\\}
    \label{fig: computations - lasso3}
\end{figure}

\subsection{Lasso Regression} \label{regression}
{The second application we consider is Lasso regression, i.e.,  $\ell_1$-regularized least squares regression. This consists of minimizing a quadratic function $h(x) = \|Ax - {b}\|^2$ over a scaled $\ell_1$-ball. The $\ell_1$ ball could be formulated as a scaled probability simplex using an extended formulation (see for example \cite{combettes2020}). This choice allows us to directly leverage the simplex-specific oracle results developed in Section~\ref{sec:oracles}. We considered a random Gaussian matrix $A \in \mathbb{R}^{50\times 100}$ and a noisy measurement ${b} = A x^* + \xi$ with $x^*$ being a sparse vector with 25 entries $\pm 1$, {where $\xi$ is some additive noise}. Linear minimization over the $\ell_1-$ball, simply amounts to selecting the column of $A$ with the {largest-magnitude} inner product with the residual vector $Ax - {b}$. In this experiment, we again observe that {\sc Shadow-Walk} is superior in iteration count as demonstrated in Figure \ref{fig: computations - lasso2}, and requires the fewest iterations among all methods. Furthermore, {\sc PGD-CG} achieves roughly two orders of magnitude fewer iterations than AFW, (and only slightly more than {\sc PGD}), as expected. In particular, while {\sc PGD-CG} uses slightly more iterations than {\sc PGD}, it attains substantially lower wall-clock time—comparable to the CG variants—thus combining the advantages of both approaches. Moreover, we again find that the addition of FW steps causes the {\sc Shadow CG} and {\sc PGD-CG} algorithms to take a smaller number of projection/shadow steps than {\sc Shadow-Walk} does.
}

\section{Conclusion} In this work, we showed connections between various descent directions for constrained minimization, such as for Frank-Wolfe direction, away-steps, pairwise steps. Further, we characterized the structure of the projections curve, and showed that it consists of two directions: the shadow, and the in-face shadow. Using this structure, we showed that the simplex and hypercube enjoy a linear number of breakpoints in the projections curve. We analyzed the continuous-time dynamics of moving along the shadow, and showed a possible discretization using {\sc Shadow-Walk}. Using the insight that Frank-Wolfe vertices are simply greedy projections, we proposed {\sc Shadow-CG} (and its practical variant {\sc PGD-CG}) which can switch between taking the Frank-Wolfe direction and shadow-steps (and projection steps) as needed, to obtain linear convergence. {\sc Shadow-CG} enjoys a linear rate of convergence that avoids facial-distance constants such as the pyramidal width. %{\color{red}(it's rate depends on the Euclidean diameter, which is not affine invariant)}. 
These results answer the key question considered in this work about optimal descent directions for constrained minimization. Moreover, our results offer a fresh perspective on the CG literature, showing that projections can in fact be beneficial in practice. Our interpolation framework provides a new way to trade off between projected and projection-free steps, depending on the user’s tolerance for projections, thereby achieving the best of both worlds. 

We hope that this analysis helps inform future studies, especially those that improve oracle running times for computing the shadow, as well as machine learning to find good descent directions. Although we focused on analyzing descent directions, these ideas can be applied to accelerated proximal methods \cite{Beck2017, wirth2023acceleration} as well, simply due to our ability to unpack the projections curve through {\sc Trace}. We leave exploring such extensions to future work. Further, we showed linear bounds on the number of breakpoints for two simple polytopes, by analyzing the KKT conditions for the breakpoints. Though we can only show an $O(2^m)$ bound for the number of breakpoints in general, we conjecture that it should be $O(\text{poly}(m))$, where $m$ is the number of facet-defining inequalities (see Appendix \ref{app:breakpoints} for discussion). 

%\sge{Add open question: If we just move along the shadow direction until sufficient progress, does that still converge? Do we have an example for non-convergence of this? }
%\sge{Add about reduction of projections curve to linear optimization, and how if the number of breakpoints was polynomial, this would give a polynomial time algorithm for LP.}
% Acknowledgments here

\section*{Acknowledgments.}
The research presented in this paper was partially supported by the Georgia Institute of Technology ARC TRIAD fellowship and the NSF grant CRII-1850182.

\textbf{Use of AI tools.} The authors used OpenAI's ChatGPT (GPT-5.5 Pro) to assist with exploratory proof structuring for Theorems \ref{thm:pw-lower} and \ref{thm:pw-upper}, and coding the jupyter notebook to construct Figure \ref{fig:example2}. The authors further used Claude Fable 5 for a consistency check throughout the paper. All statements and proofs have been reviewed, carefully verified, and edited by the authors, who take full responsibility for their correctness.

\singlespacing
%\section*{References}
\bibliographystyle{IEEEtran}
\bibliography{related}
\newpage

\begin{appendix}

\setcounter{footnote}{0}

\section{Shadow ODE from Section \ref{sec:shadowode} }\label{sec:appendixODE}

In this section we present our results on the continuous time dynamics of moving in the shadow of the gradient and mirror descent for \textit{general mirror maps} (beyond the Euclidean case considered in the main body of the paper), for the sake of completeness. To that end, let $X(t)$ denote the continuous-time trajectory of our dynamics and $\dot{X}$ denote the time-derivative of $X(t)$, i.e., $\dot{X}(t) = \frac{d}{dt}X(t)$. In \cite{Krichene_2015}, Krichene et al. propose the following coupled dynamics $(X(t), Z(t))$ for mirror descent, where $X(t)$ evolves in the primal space (i.e. domain of $\nabla \phi$), and $Z(t)$ evolves in the dual space (i.e. domain of $\nabla \phi^*$) as follows, initialized with $Z(0) = \iter{z}{0} \in \mathrm{dom}(\nabla \phi^*) \text{ with } \nabla \phi^*(\iter{z}{0}) = \iter{x}{0} \in \mathcal{P}$: 
\begin{equation} \label{MD ODE}
    \dot{Z}(t) = -\nabla h(X(t)), \quad X(t) = \nabla \phi^*(Z(t)). 
%    \end{cases}
\end{equation}
Let ${d}_{X(t)}^{\phi}$ be the directional derivative with respect to the Bregman projections in the mirror descent algorithm, i.e., ${d}_{X(t)}^{\phi}= \lim_{\epsilon \downarrow 0}  \frac{\nabla \phi^* (\nabla \phi(X(t)) - \epsilon \nabla h(X(t))) - X(t)}{\epsilon}$. The continuous time dynamics of tracing this directional derivative are simply \begin{equation}
    \dot{X}(t) = {d}_{X(t)}^\phi. \label{ODe for shadow 1}
\end{equation} 
{These dynamics in \eqref{ODe for shadow 1}} solely operate in the primal space and one can initialize them with $X(0) = \iter{x}{0}\in \mathcal{P}$
and show that they are equivalent to \eqref{MD ODE} under mild conditions:

\begin{restatable}{theorem}{odetheorem1}
\label{ode theorem full}
Assume {that $\phi$ is a differentiable Legendre-type mirror map} and that the directional derivative ${d}_{X(t)}^\phi$ exists for all $t \geq 0$. Then, the dynamics for mirror descent \eqref{MD ODE} are equivalent to the shadow dynamics $\dot{X}(t) = {d}_{X(t)}^\phi$ with the same initial conditions $X(0) = \iter{x}{0}\in \mathcal{P}$.
\end{restatable}

\noindent \textit{Proof.}
Consider the general dynamics for mirror descent as follows: 
\begin{equation} \label{MDODE-general}
    \dot{Z}(t) = -\nabla h(X(t)), \quad X(t) = \nabla \phi^*(Z(t)). 
\end{equation}
By the definition of the directional derivative of $\nabla \phi^*$ in the direction $\dot Z(t)$ (see for example \cite{rockafellar_1970}), we have
$$ \dot X(t) = \lim_{\epsilon \downarrow 0}
    \frac{\nabla \phi^*(Z(t)+\epsilon \dot Z(t)) - \nabla \phi^*(Z(t))
    }{\epsilon}.$$
Using $\dot Z(t)=-\nabla h(X(t))$ and $X(t)=\nabla \phi^*(Z(t))$, this becomes
$$ \dot X(t)
    =
    \lim_{\epsilon \downarrow 0}
    \frac{
        \nabla \phi^*(Z(t)-\epsilon\nabla h(X(t)))
        -
        X(t)
    }{\epsilon}.$$
{Since $\phi$ is Legendre-type, $\nabla \phi$ and $\nabla \phi^*$ are inverse maps on the interiors of their respective domains (see Theorem 26.5 in \cite{rockafellar_1970})}. Moreover, by definition of the mirror descent ODE given in $\eqref{MD ODE}$, we have $X(t) = \nabla \phi^*(Z(t))$. Using these facts we get $Z(t) =  (\nabla \phi^*)^{-1}(X(t)) = \nabla \phi (X(t))$. Thus,
\begin{equation*}
       \dot{X} (t) = \lim_{\epsilon \downarrow 0}  \frac{\nabla \phi^* (\nabla \phi(X(t)) - \epsilon \nabla h(X(t))) - X(t)}{\epsilon} = {d}_{X(t)}^\phi
\end{equation*}
which coincides with the primal-space dynamics in \eqref{ODe for shadow 1}.
\hfill $\square$    

\begin{remark}
Theorem \ref{ode theorem full} follows the differentiability setting used by Krichene et al. \cite{Krichene_2015}, who assume that $\psi$ is differentiable on the image of $\nabla \psi^*$ so that $\nabla \psi = (\nabla \psi^*)^{-1}$. The constrained Euclidean case of Section \ref{sec:shadowode} falls slightly outside this smooth setting: with $\phi(x)=\frac{1}{2}\|x\|^2+\delta_P(x)$, we still have $\nabla \phi^*(z)=\Pi_P(z)$,
but $\phi$ is not differentiable on $\partial P$. We therefore treat the Euclidean shadow dynamics directly through the directional derivative of the projection operator.
\end{remark}

\section{Number of Breakpoints}\label{app:breakpoints}\label{sec:appB}
We conjecture that the number of breakpoints of the projection curve  is polynomial in the number of facets of the polytope. To search for a counterexample to this conjecture, one might consider extended formulations with a logarithmic number of facets compared to the dimension of the polytope (extended formulations for regular polygons \cite{gilneur2000computational, ben2001polyhedral}, and exponential instances for the successive shortest path algorithm for max flow \cite{zadeh1973bad}). For example, consider a polyhedron with $O(\log n)$ facets whose two-dimensional projection is the regular $n$-gon, which thus has exponentially many vertices in the number of facets of the extended formulation. 

%We conjecture that the number of breakpoints of the polytope is polynomial in the facets. To break this conjecture, one might consider extended formulations with a logarithmic number of facets compared to the dimension of the polytope (extended formulations for regular polygons \cite{gilneur2000computational, ben2001polyhedral}, and exponential instances for the successive shortest path algorithm for max flow \cite{zadeh1973bad}). For a example, let's consider an n-gon polyhedron which has a two-dimensional projection that with exponentially many vertices, in the number of facets of the extended formulation. 

Let $Q$ be the n-gon in $\mathbb{R}^2$, and $P$ be its extended formulation with $O(\log n)$ facets \cite{fiorini2012extended}, connected through the projection operator $Q = \Pi_{\mathbb{R}^2}(P)$. The curve studied in the paper is the Euclidean projection curve in the original space over any polyhedron $K$, for fixed $d$ and $x_0 \in K$, varying $\lambda \in \mathbb{R}_+$: 
\[
    g(\lambda)=\Pi_K(x_0-\lambda d).
\]
Its shadow on the n-gon $\Pi_{Q}(g(\lambda))$ is generally \emph{not} the same object as the
Euclidean projection of the the projected line onto the projected polytope $Q$, i.e.,:
\[
    \Pi_{Q}\!\left(\Pi_{P}(x_0 - \lambda d)\right)
    \;\neq\;
    \Pi_{P}\!\left(\Pi_{Q}(x_0 -\lambda d)\right).
\]

Thus, projection of a line onto an extended polyhedron can have a small (polynomial in facets, i.e., $O(\text{poly}(\log n))$) number of breakpoints, while the projection of a line onto the n-gon can still have $O(n)$ breakpoints (see Figure \ref{fig:example2} for an illustration). 

%Further, we have reduced the dependence of our results on the number of breakpoints. Our algorithmic statements are now explicitly parameterized either by the numberof breakpoints $\beta$ or by the cost of the relevant oracles. For the simplex and hypercube, we prove $O(n)$ breakpoint bounds. We also discuss the follow-up result showing analogous behavior foracute polytopes \cite{ahunbay2024first}. Our additional results on switching between a single projection step and Frank-Wolfe in the PGD-CG algorithm do not rely on the number of breakpoints.

\begin{figure}[H]
    \centering
    \includegraphics[width=0.5\linewidth]{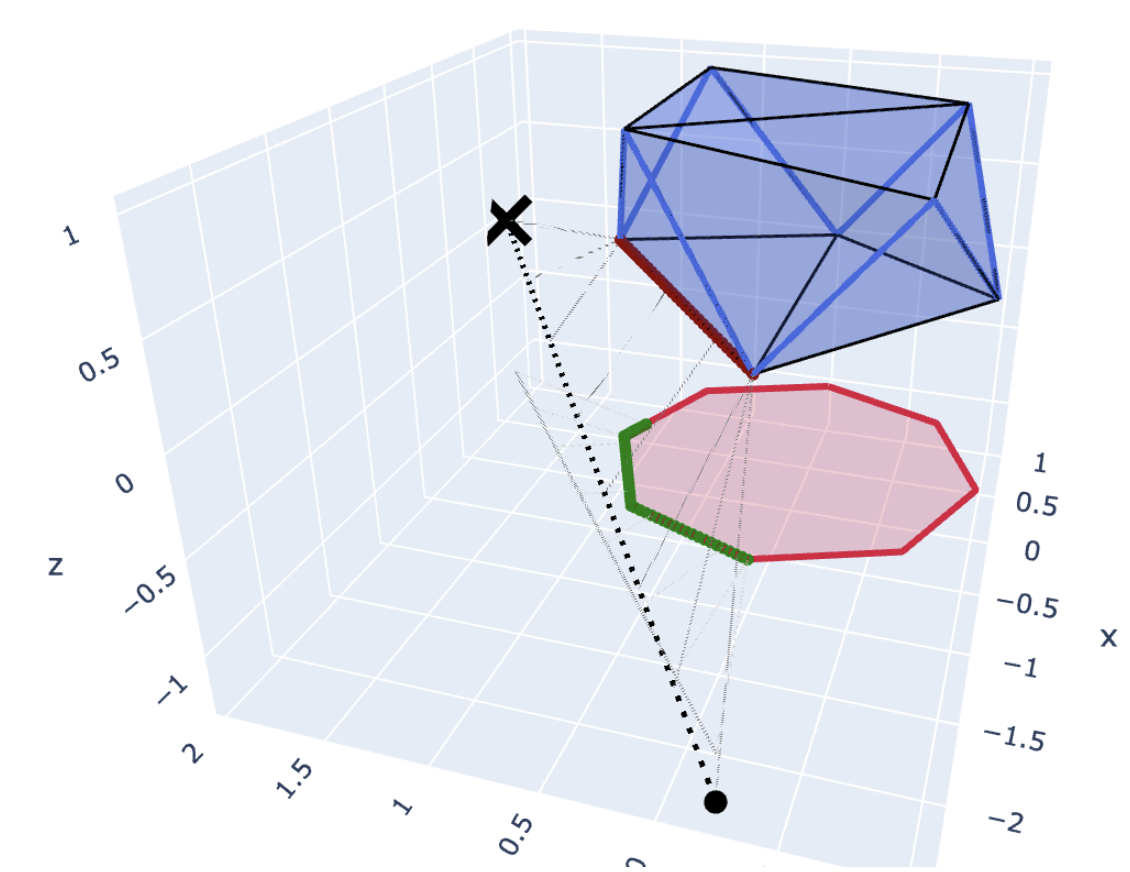}
    \caption{Interaction of the projections curve with the projection}
    \label{fig:example2}
\end{figure}

\section{Pyramidal Width of Base Polytopes}\label{sec:app-pyramidal}

%\sge{from response doc: We further positioned our work in the context of recent surveys of Bomze et al. \cite{bomze2024frank}, directions of descent in the Frank-Wolfe literature \cite{Marcotte1986, Lacoste2015, Garber_2016, Bashiri_2017}, geometric results on the linear convergence of conditional gradient methods \cite{wirth2025fast, iommazzolinear, garber2015faster}, and acute polytopes \cite{ahunbay2024first}.}  

Let $E$ be a finite ground set. For $x\in\mathbb R^E$ and $S\subseteq E$
write $x(S):=\sum_{i\in S}x_i$. For a normalized submodular function
$f:2^E\to\mathbb R$, its \emph{base polytope} is
\[
    B(f)
    =
    \bigl\{x\in\mathbb R^E : x(S)\le f(S)\ \forall S\subseteq E,\
    x(E)=f(E)\bigr\}.
\]
We use the facial-distance characterization of pyramidal width: if
$P=\conv(V)$ is a polytope with vertex set $V$, then
\[
    \rho \coloneqq \PWidth(P)
    =
    \min_{\emptyset\neq F\subsetneq P}
    \dist\bigl(F,\conv(V\setminus F)\bigr),
\]
where the minimum is over all nonempty proper faces $F$ of $P$. For
integral submodular base polytopes on an $N$-element ground set (with at
least two vertices), we first show a lower bound on the pyramidal width. 

\begin{theorem}[Lower bound on the pyramidal width of $B(f)$]\label{thm:pw-lower}
Let $f:2^E\to \mathbb{Z}$ be normalized, submodular and integral, with $|E|=N\ge 2$, and suppose that $B(f)$ has at least two vertices. Then
\begin{equation}\label{eq:universal-lb}
    \PWidth(B(f))
    \ge
    \sqrt{\frac{6}{N(N-1)(2N-1)}}
    =
    \Omega(N^{-3/2}).
\end{equation}
\end{theorem}

\textit{Proof.} Let $V$ be the vertex set of $B(f)$. By the facial-distance characterization, it is enough to lower-bound $\dist\bigl(F,\conv(V\setminus F)\bigr)$ uniformly over all nonempty proper faces $F\subsetneq B(f)$. Fix such a face $F$. Using Edmonds' greedy theorem for submodular base polytopes, there is an ordered partition $E=P_1\uplus\cdots\uplus P_t$ and prefix sets
$S_j=P_1\cup\cdots\cup P_j,$ for $j=1,\dots,t-1$, such that
$$
F=
\left\{
x\in B(f):
x(S_j)=f(S_j)\ \text{for }j=1,\dots,t-1
\right\}.
$$
Set $c_F:=\sum_{j=1}^{t-1}\one_{S_j}$ and $\beta_F:=\sum_{j=1}^{t-1}f(S_j)$. Then $c_F^\top x=\beta_F$ for every $x\in F$.

Now take any vertex $v\in V\setminus F$. Since $v\notin F$, at least one
prefix equality fails: $v(S_j)<f(S_j)$ for some $j\in\{1,\dots,t-1\}$. Since $f$ is integer-valued and $B(f)$ is integral, it follows that $v(S_j)\le f(S_j)-1$. Moreover, %for all other prefixes $S_\ell$ where $\ell \in\{1,\dots,t-1\} \setminus\{j\}$, we still have $v(S_\ell)\le f(S_\ell)$ 
since $v$ is feasible in $B(f)$, we have
$$
c_F^\top v
=
\sum_{\ell=1}^{t-1}v(S_\ell)
\le
\sum_{\ell=1}^{t-1}f(S_\ell)-1
=
\beta_F-1.
$$
By convexity, every $y\in\conv(V\setminus F)$ satisfies $c_F^\top y\le \beta_F-1$. Therefore, for arbitrary $x\in F$ and
$y\in\conv(V\setminus F)$, we have $c_F^\top(x-y) \ge 1$.
By Cauchy--Schwarz,
$$
1
\le
c_F^\top(x-y)
\le
\|c_F\|_2\|x-y\|_2 \implies
\|x-y\|_2\ge \frac{1}{\|c_F\|_2}.
$$
Taking the minimum over $x\in F$ and $y\in\conv(V\setminus F)$ gives $\dist\bigl(F,\conv(V\setminus F)\bigr)
\ge
\frac{1}{\|c_F\|_2}$.

It remains to upper-bound $\|c_F\|_2$. If $i\in P_a$, then $i$ belongs to exactly $t-a$ prefix sets. Thus, $(c_F)_i=t-a$, and
$$\|c_F\|_2^2 =\sum_{a=1}^{t-1} |P_a|(t-a)^2.$$
Since all $P_a$ are nonempty, this quantity is maximized, for fixed $t$, by taking
$$
|P_1|=N-t+1,
\qquad
|P_2|=\cdots=|P_t|=1.
$$
Consequently,
$$
\|c_F\|_2^2
\le
(N-t+1)(t-1)^2+\sum_{k=0}^{t-2}k^2
\le
\sum_{k=0}^{N-1}k^2
=
\frac{N(N-1)(2N-1)}{6}.
$$
Taking the minimum over all nonempty proper faces $F$ proves the theorem.
\hfill $\square$

The preceding theorem gives a universal ($\Omega(N^{-3/2})$) lower bound for integral submodular base polytopes. We now show that this dependence on $N$ is tight up to constant factors by considering the standard permutahedron. We use the normalized permutahedron $P_m \coloneqq \conv\{\pi(0,1,\ldots,m-1):\pi\in S_m\}$, where $S_m$ is the symmetric group.

\begin{theorem}[Upper bound for the permutahedron]\label{thm:pw-upper}
For the standard permutahedron $P_m$ ($m \geq 2$),
$$
\PWidth(P_m)
\le
\sqrt{\frac{12}{m(m^2-1)}}.
$$
\end{theorem}

\textit{Proof.} Let $v=(m-1,m-2,\dots,0)$. This is a vertex of $P_m$. Since a singleton vertex is a face, the facial-distance characterization gives $\PWidth(P_m)
\le \dist\bigl(\{v\},\conv(V\setminus\{v\})\bigr)$, where $V$ is the vertex set of $P_m$. It suffices to construct a point $
p\in\conv(V\setminus\{v\})$ such that
$\|v-p\|_2=\sqrt{\frac{12}{m(m^2-1)}}$.

For $i=1,\dots,m-1$, let $u_i$ be the vertex obtained from $v$ by
swapping the adjacent coordinates $i$ and $i+1$, i.e.  $u_i=v-e_i+e_{i+1}$.  Define the centered point $\bar v:=v-\frac{m-1}{2}\one$ and coefficients
$$
\alpha_i
:=
\frac{i(m-i)}{2\|\bar v\|_2^2},
\qquad i=1,\dots,m-1.
$$
Since
$$\sum_{i=1}^{m-1} i(m-i) = \frac{m(m^2-1)}{6} = 2\|\bar v\|_2^2,$$
the coefficients $\alpha_i$ are nonnegative and sum to $1$. Therefore $p:=\sum_{i=1}^{m-1}\alpha_i u_i$ belongs to $\conv(V\setminus\{v\})$. We next show that this point $p$ satisfies the distance bound in the theorem statement.

Note that
$$
p
=
\sum_{i=1}^{m-1}\alpha_i(v-e_i+e_{i+1})
=
v-\sum_{i=1}^{m-1}\alpha_i(e_i-e_{i+1})
\implies
v-p
=
\sum_{i=1}^{m-1}\alpha_i(e_i-e_{i+1}).
$$
We claim that $v-p=\frac{\bar v}{\|\bar v\|_2^2}$. Indeed, the first coordinate is
$
\alpha_1
=
\frac{m-1}{2\|\bar v\|_2^2}
=
\frac{\bar v_1}{\|\bar v\|_2^2}.
$
For $2\le k\le m-1$, the $k$-th coordinate is
$$
\alpha_k-\alpha_{k-1}
=
\frac{k(m-k)-(k-1)(m-k+1)}
{2\|\bar v\|_2^2}
=
\frac{m+1-2k}{2\|\bar v\|_2^2}
=
\frac{\bar v_k}{\|\bar v\|_2^2}.
$$
The last equality follows from the definition of $\bar v$. Finally, the $m$-th coordinate is
$$
-\alpha_{m-1}
=
-\frac{m-1}{2\|\bar v\|_2^2}
=
\frac{\bar v_m}{\|\bar v\|_2^2}.
$$
Thus $v-p=\frac{\bar v}{\|\bar v\|_2^2}$ and consequently,
$
\|v-p\|_2
=
\frac{1}{\|\bar v\|_2}
=
\sqrt{\frac{12}{m(m^2-1)}},
$
which concludes the proof. \hfill $\square$

\end{appendix}

\end{document}